\documentclass[11pt,final]{amsart}
 
\usepackage[margin=1.1in]{geometry}

\usepackage{amsmath, amsfonts, amsthm, amssymb,amscd, enumerate,dsfont}
\usepackage{graphics}
\usepackage{color}
\usepackage[colorlinks, linkcolor=blue, citecolor=red, urlcolor=blue]{hyperref}

\usepackage{amsthm,amssymb,amsmath,amsfonts}
\usepackage{mathtools}
\usepackage{enumerate}
\usepackage[all]{xy}

\usepackage{verbatim}

\newtheorem{theorem}{Theorem}[section]

\newtheorem{Prop}[theorem]{Proposition}
\newtheorem{prop}[theorem]{Proposition}
\newtheorem{Lemma}[theorem]{Lemma}
\newtheorem{lemma}[theorem]{Lemma}

\newtheorem{corollary}[theorem]{Corollary}
\newtheorem{Conj}[theorem]{Conjecture}

\newtheorem*{problem}{Problem}
\newtheorem*{Quest*}{Question}
\newtheorem*{theorem*}{Theorem}

\newtheorem{Th}[theorem]{Theorem}
\newtheorem{Cor}[theorem]{Corollary}
\newtheorem{Question}{Question}
\newtheorem*{Question*}{Question}
\newtheorem*{Th*}{Theorem}

\theoremstyle{definition} %%%%
\newtheorem{definition}[theorem]{Definition}
\newtheorem{Def}[theorem]{Definition}
\newtheorem{remark}[theorem]{Remark}
\newtheorem{Remark}[theorem]{Remark}

\newcommand{\ON}{\operatorname}

\newcommand{\mult}{\operatorname{mult}}

\newcommand{\gem}{\geqslant}

\newcommand{\lct}{\operatorname{lct}}
\newcommand{\mct}{\operatorname{mct}}

\newcommand{\Cox}{\operatorname{Cox}}

\newcommand{\GL}{\operatorname{GL}}

\newcommand{\Spec}{\operatorname{Spec}}

\newcommand{\Pic}{\operatorname{Pic}}

\newcommand{\SL}{\operatorname{SL}}

\newcommand{\Sym}{\operatorname{Sym}}
\newcommand{\Ch}{\operatorname{Ch}}

\newcommand{\GG}{\mathbb G}
\newcommand{\CC}{\mathbb C}
\newcommand{\QQ}{\mathbb Q}

\newcommand{\cQ}{\mathcal Q}

\newcommand{\HH}{\mathrm H}

\newcommand{\PP}{\mathbb P}

\newcommand{\cL}{\mathcal L}
\newcommand{\cO}{\mathcal O}
\newcommand{\CO}{\mathcal O}
\newcommand{\cI}{\mathcal I}

\DeclareMathOperator{\mon}{mon}

\DeclareMathOperator{\Hyp}{Hyp}
\DeclareMathOperator{\Int}{Int}
\DeclareMathOperator{\diag}{diag}

\DeclareMathOperator{\Aut}{Aut}

\DeclareMathOperator{\lcm}{lcm}
\DeclareMathOperator{\val}{\mathit{val}}
\DeclareMathOperator{\Models}{Models}

\newcommand{\MP}{\mathbb{P}}
\newcommand{\VV}{\mathbb{V}}

\newcommand{\ZZ}{\mathbb{Z}}

\newcommand{\CM}{\mathcal{M}}

\newcommand{\cA}{\mathcal{A}}

\newcommand{\fD}{\mathfrak D}

\newcommand{\co}{\colon}

\DeclareMathOperator{\spec}{Spec}
\DeclareMathOperator{\mspec}{\mathfrak{m}-Spec}

\DeclareMathOperator{\proj}{Proj}

\DeclareMathOperator{\Supp}{Supp}
\DeclareMathOperator{\pr}{pr}

\newcommand{\MQ}{\mathbb{Q}}

\AtBeginDocument{%
   \def\MR#1{}
}

\numberwithin{equation}{section}

\title{Stability of fibrations over one-dimensional bases}%, and standard models of del Pezzo fibrations}
\author{Hamid Abban, Maksym Fedorchuk, Igor Krylov}
\begin{document}

\address{\emph{Hamid Abban}
\newline
\textnormal{Department of Mathematical Sciences, Loughborough University, Loughborough LE11 3TU, UK
\newline
 \texttt{h.ahmadinezhad@lboro.ac.uk}}}

\address{ \emph{Maksym Fedorchuk}\newline \textnormal{Department of Mathematics, Boston College, 140 Commonwealth Ave, Chestnut Hill, MA 02467, USA
\newline
\texttt{maksym.fedorchuk@bc.edu}}}

\address{ \emph{Igor Krylov}\newline \textnormal{Korea Institute for Advanced Study, 85 Hoegiro, Dongdaemun-gu, Seoul 02455, Republic of Korea \newline
\texttt{IKrylov@kias.re.kr}}}

\begin{abstract}
We introduce and study a new notion of stability for varieties fibered over curves,
motivated by Koll\'ar's stability for homogeneous polynomials with integral coefficients \cite{kollar-polynomials}. 
We develop tools to study geometric properties of stable birational models
of fibrations whose fibers are complete intersections in %low-dimensional  
weighted projective spaces. As an application, we prove the existence of standard models of threefold degree $1$  %and two  
del Pezzo fibrations, 
settling a conjecture of Corti \cite{corti-annals}. 
\end{abstract}

\maketitle

\setcounter{tocdepth}{1}
\tableofcontents

%%%%%%%%%%%%%%%%%%%%%%%%%%%%%
%%%%%%%%%%%%%%%%%%%%%%%%%%%%%
\section{Introduction} 

Finding {\it good models} of varieties fibered over one-dimensional schemes is a central problem in geometry and arithmetic, with some of the key %best known 
examples being: 
Tate's minimal  %Weierstrass  
models of elliptic curves \cite{tate-algorithm}, %and, more generally, 
N\'eron models of abelian varieties \cite{neron-book,neron}, semistable reduction for curves \cite{artin-winters}, Sarkisov's standard models of conic bundles \cite{Sarkisov}, Koll\'ar's theory of 
semistable hypersurfaces over PIDs \cite{kollar-polynomials}.

In this paper, we solve a concrete problem in birational geometry of threefolds (Corti's conjecture on the existence of standard models of degree $1$ del Pezzo fibrations) 
by developing a new theory of semistability for fibrations over a one-dimensional base,
which we call \emph{Koll\'ar semistability}. For the main application{\textemdash}Koll\'ar  semistability of weighted hypersurfaces (or their intersections) in weighted projective spaces{\textemdash}our theory is a common generalization of both Tate's and Koll\'ar's theories.  We give an overview of Corti's conjecture in \S\ref{S:corti} and of the general theory in \S\ref{S:stability}.

\subsection{Corti's conjecture on standard models of degree $1$ del Pezzo fibrations}\label{S:corti}
Finding \emph{standard} models for Mori fiber spaces (MFS) $f\colon X\to Z$ over a positive-dimensional base $Z$ is crucially useful for birational geometry, and is already an interesting problem when $X$ is a threefold. 
It could be thought of as a \emph{post-MMP} step to further simplifying birational models of a given uniruled variety $X$. Standard models behave nicely under predictions in terms of birational rigidity, 
see Section \ref{QandC} for further explanation and speculations.  

When $\dim X=3$, the two types of MFS are conic bundles and del Pezzo fibrations (of degree $d\in\{1,\dots, 9\}$). 
Sarkisov developed a satisfying theory of standard models of conic bundles over a surface \cite[Theorem~1.13]{Sarkisov}. Motivated by this, Corti proposed the following 
notion of a standard model for a threefold del Pezzo fibration:

\begin{Def}[{\cite[Definitions~1.8~and~1.13]{corti-annals}}] %when $k=\CC$]
\label{D:standard-model} Suppose $Z$ is an (essentially) smooth irreducible one-dimensional scheme over an algebraically closed field $k$.
Let $K=k(Z)$ be the function field of $Z$
%Let $R:=\cO_{C,p}$ be the local ring of a closed point $p\in C$. Suppose that for $Z=\spec(R)$ or $Z=C$, 
and $\pi\colon X \to Z$ be a flat projective morphism with the generic fiber a smooth del Pezzo surface $X_K$ over $K$ of degree $d:=K_{X_K}^2 \in \{1,\dots, 9\}$.
We say that $\pi\colon X \to Z$ is \emph{a standard model} of $X_{K}$ over $Z$, or \emph{a standard del Pezzo fibration of degree $d$}, if:
\begin{enumerate}
\item $X$ has terminal singularities.  
\item $\pi$ has integral fibers.
\item $-aK_{X}$ is a $\pi$-ample line bundle, where 
\[
a=\begin{cases} 1 & \text{if $d\geq 3$}, \\ 2 &\text{if $d=2$,} \\ 6 & \text{if $d=1$}.\end{cases}
\]
\end{enumerate}
\end{Def}

In this paper, we address the following:
\begin{problem}
Given a del Pezzo surface $X_K$ over $K$, can we find a standard model of $X_K$ over $Z$?
\end{problem}

\begin{remark}
A standard model $\pi\colon X\to Z$ is a %$\QQ$-factorial $a$-Gorenstein 
Mori fiber space if and only if $\Pic(X_{K})=\ZZ$.
%In particular,
%our result implies the existence of distinguished models in every equivalence class of $dP_1$ Mori fiber space fibrations.
\end{remark}

\begin{remark}[Reduction to the local case]\label{local-to-global}
By descent, the problem of finding a standard model immediately reduces to the local case $Z=\spec \cO_{C,p}$, where $\cO_{C,p}$ is a local ring of a closed point on a smooth curve $C$ over $k$. 
Indeed, suppose $\pi\colon X\to C$ is any model of $X_K$ over a curve $C$.  Let $T\subset C$ be the finitely many points where the fiber of $\pi$ is singular.
The disjoint union of $C\setminus T$ and $\{\spec \cO_{C,p}\}_{p\in T}$ is an fpqc covering of $C$. Since fpqc descent is effective for projective schemes
endowed with a choice of a very ample line bundle (see e.g., \cite[Theorem 4.38]{vistoli}), in our case, the $a$-Gorenstein del Pezzo surfaces, we can glue a smooth family over $C\setminus T$
and standard models over spectra of the local rings $\{\spec \cO_{C,p}\}_{p\in T}$ into a global standard model over $C$.
\end{remark}

In the complex case, Corti established the existence of standard models for $d\geq 2$:
\begin{Th}[Corti {\cite[Theorems~1.10~and~1.15]{corti-annals}}] \label{Corti-Th} Suppose $d\geq 2$ and $k=\CC$. 
Let $X_K$ be a smooth del Pezzo surface of degree $d$ over $K=\CC(Z)$, the function field of a smooth complex curve $Z$.  Then there exists a standard model of $X_{K}$ over $Z$.
\end{Th}
In this paper, we prove Corti's conjecture \cite[Conjecture 1.16]{corti-annals}, establishing the existence of standard models in the remaining degree $1$ case: % which was the original motivation of our project:
\begin{theorem}\label{corti-theorem} Suppose $Z$ is an (essentially) smooth irreducible one-dimensional scheme over an algebraically closed field $k$ with $\ON{char}(k)\neq 2,3$.
Let $K=k(Z)$ be the function field of $Z$. Suppose $X_K$ is a smooth del Pezzo surface of degree $1$ over $K$. Then there exists a standard model of $X_K$ over $Z$.
\end{theorem}
A more precise result will be given in Theorem \ref{blackbox} below. Our methods also extend Corti's Theorem \ref{Corti-Th} for $d=2$ to every  %arbitrary 
algebraically closed field $k$ with $\ON{char}(k)\neq 2$ (see Theorem \ref{T:degree-2}).

%%%%%%%%%%%%%%%%%%%%%%%%%%%%%
\subsection{Koll\'ar stability: First examples} \label{S:stability}
%%%%%%%%%%%%%%%%%%%%%%%%%%%%%
Let $R$ be a DVR with a uniformizer $t$ and the fraction field $K=\ON{Frac}(R)$. Set $\Delta=\spec R$.  Suppose we are given a flat projective morphism $\pi\colon X \to \Delta$, where $X$ is normal and
$\pi_* \cO_X=\cO_{\Delta}$. 
We denote this by $X/\Delta$, and refer to it as a \emph{fibration}.
We say that $\pi^\prime\colon X^\prime \to \Delta$ is a (birational) \emph{model} of $X/\Delta$ if there is a birational map $\chi\colon X\dashrightarrow X^\prime$  % and a birational map $\psi\colon Z \dashrightarrow Z^\prime$
 such that  
$\pi^\prime\circ\chi = \pi$ and such that $\chi$ induces an isomorphism between the generic fibers of $X/\Delta$ and $X'/\Delta$:
\begin{equation*}
\begin{aligned}
\xymatrix
{ 
	X\ar@{-->}[rr]^\chi \ar[dr]^\pi &  & \ar[dl]_{\pi^\prime} X^\prime \\
	& \Delta &
}
\end{aligned}
\end{equation*}

A priori there are many models for a given fibration $X/\Delta$. Hence it is natural to ask: What is the {\it best model} for a given fibration $X/\Delta$? Or, equivalently, given a projective scheme $X_K$ over $K$,
what is its best model over $R$? It is crucial to note that we \emph{do not allow base change}, motivated by questions in birational geometry. 

We will answer these questions, in a great generality, by defining a notion of Koll\'ar stability of $X/\Delta$
in Section \ref{kollar}. Here is a brief overview of our theory: As to be expected, stability depends on several choices.  
Suppose that there exists a proper parameter space $M$ over $R$ %(i.e., a fine moduli space for some moduli functor) 
such that the generic fiber of $X/\Delta$ is a $K$-valued
point of $M$. 
By properness of $M$, after possibly passing to a birational model, we can assume that $\pi \colon X\to \Delta$ is a pullback of the universal family over $M$ via a morphism
$f_{\pi}\colon \Delta \to M$.  Assume that $M$ is endowed with an additional structure given by an action of a group scheme $G$, a choice $\cL \in \Pic^G(M)$ of a $G$-linearized line bundle, and a choice of its 
$G$-invariant section $\fD \neq 0 \in \HH^0(M,\cL)^G$.  
Given all of this, we can make the following: %have the following notion of stability for the maps from the curve $Z$ to $M$:

\begin{Def} \label{Stab-global-def} 
We say that $\pi\colon X \to \Delta$ is a \emph{$\fD$-semistable} model of $X_K$ if $f_{\pi}(K) \notin \Supp(\fD)$ and the $t$-valuation of the Cartier divisor $f_\pi^*(\fD)$  on $\Delta$ is minimal among all 
maps $f\colon \Delta \to M$ with $ f_\pi(K)\in G(K)\cdot f(K)$.
\end{Def}

Before we delve into the general theory, we give two examples.
\subsubsection{Koll\'ar stability of cubics in $\PP^3$, after \cite{kollar-polynomials}} Suppose $X_K$ is a smooth degree $3$ del Pezzo surface over $K$. Then $X_K$ is isomorphic to a cubic hypersurface in $\PP^3_K$, 
given by some form $\tilde{F} \in \PP\HH^0(\PP^3_K, \cO(3))$. A cubic form $F\in \PP\HH^0(\PP^3_R, \cO(3))$ is called a \emph{semistable (cubic) model of $X_K$ over $R$} if $\tilde{F}\in \GL_4(K) \cdot F$ and for every 
weight system $\rho=(w_1,w_2,w_3,w_4)\in \ZZ^4$ and every choice of homogeneous coordinates $x,y,z,w$ on $\PP^3_R$, we have
\begin{equation}\label{E:cubic}
\mult_{\rho} (F) \leq \frac{3}{4}(w_1+w_2+w_3+w_4), \ \text{where}
\end{equation}
 \[
\mult_{\rho} (F):=\text{the minimum of the $t$-valuations of all the coefficients of $F(t^{w_1}x, t^{w_2}y, t^{w_3}z, t^{w_4}w)$.}
\]

This definition is obtained from Definition \ref{Stab-global-def} by taking $M=\PP\HH^0(\PP^3_R, \cO(3))$, $G=\GL_4=\Aut_{gr}(\Cox(\PP^3))$, and $\fD$ to be the discriminant divisor on $M$.
\begin{theorem}[Koll\'ar, cf. {\cite[Proposition 6.4.1]{kollar-polynomials}}] \label{T:cubic} \hfill \begin{enumerate}
\item A semistable cubic model  over $R$ exists for every smooth degree $3$ del Pezzo over $K$. 
\item Suppose $R=\cO_{C,p}$ is the local ring of a closed point of a smooth algebraic curve over an algebraically closed field $k$. %with $\ON{char}(k)\neq 2, 3$, 
Then a semistable cubic model is a standard del Pezzo fibration of degree $3$ over $\Delta=\spec R$. 
\end{enumerate}
\end{theorem}
A remarkable feature of this result is that semistability has a numerical criterion, given by \eqref{E:cubic}, phrased in terms of the equation $F$. Once the reader accepts the existence of a semistable model, its geometric properties,
such as terminality of the total space,  %of the fibration 
can be derived from this numerical criterion.  

\subsubsection{T-semistable degree $1$ del Pezzo fibrations}
The novelty in our approach to Theorem \ref{corti-theorem} (Corti's conjecture) is to treat it as a stability problem, and not as a problem in the birational geometry of threefolds. 
Namely, we find a correct analogue Theorem \ref{T:cubic} for degree $1$ del Pezzos, which we now explain. 

A smooth degree $1$ del Pezzo surface $X_K$ over $K$ can be written as a sextic hypersurface in a weighted projective space $\PP_K(1,1,2,3)$.
Unfortunately, no notion of semistability for such sextics leads to standard models (see Remark \ref{why} for a technical explanation), so instead we write
$X_K$ as a complete intersection 
\begin{equation}\label{E:generic-dp1}
\tilde F(x,y,z,w,s)=\tilde H(x,y,z,w,s)=0,
\end{equation}
in $\PP_K(1_x,1_y,2_z,3_w,3_s)$, where $\deg \tilde F=6$ and $\deg \tilde H=3$. 
For example, we can write $X_K$ as a sextic in the variables $x,y,z,w$ in $\PP_K(1,1,2,3)$ and take $\tilde H=s$.
We call the resulting ideal $I=(\tilde{F}, \tilde{H})\subset K[x,y,z,w,s]$ a $(6,3)$-intersection over $K$.  %(Here, $x,y,z,w,s$ is some choice of elements of $\Cox(\PP_K(1,1,2,3,3))$ of degrees $1,1,2,3,3$, respectively, that
%generate this graded ring. We call such a choice a system of quasihomogeneous coordinates on $\PP_K(1,1,1,2,3)$.)

We say that an ideal $I=(F,H) \subset R[x,y,z,w,s]=\Cox(\PP_R(1,1,2,3,3))$ is a $(6,3)$-intersection over $R$ if $\deg F=6$, $\deg H=3$, and $H\notin (t)$ and $F\notin (H,t)$. 
By properness, every $(6,3)$-intersection $(\tilde F, \tilde H)$ over $K$ can be completed uniquely to a $(6,3)$-intersection
$I=(F,H)$ over $R$ with $(\tilde F, \tilde H)=(F,H)\otimes_{R} K$.  

Given a choice of quasihomogeneous coordinates 
$x, y, z, w, s$ --- namely a sequence of degree $1,1,2,3,3$ elements of $\Cox(\PP_R(1,1,2,3,3))$ that generate this graded $R$-algebra, 
a weight system $\rho=(w_1,w_2,w_3,w_4,w_5)\in \ZZ^5$, and an element $A\in R[x,y,z,w,s]$, we define 
$\mult_{\rho}(A)$ to be the minimum of the valuations of all the coefficients of $A(t^{w_1}x, t^{w_2}y, t^{w_3}z, t^{w_4}w, t^{w_5}s)\in K[x,y,z,w,s]$.

 We can now make the following:
\begin{definition}\label{T-semi} Fix $0<\epsilon \ll 1$.  A $(6,3)$-intersection $I=(F,H)$ over $R$ is called a \emph{T-semistable} $(6,3)$-intersection model of $X_K$ over $R$ if:
\begin{enumerate}
\item The ideal $(F,H)\otimes_{R} K$ is a $(6,3)$-intersection over $K$ defining a smooth del Pezzo surface in $\PP_K(1,1,2,3,3)$ isomorphic to $X_K$ over $K$.
\item For every choice of generators $I=(F,H)$, for every choice of quasihomogeneous coordinates in $R[x,y,z,w,s]$, and for every weight system $\rho=(w_1,w_2, w_3, w_4, w_5)$, we have:
\begin{equation}\label{num-ss}
\mult_\rho(F)+ \frac{6}{7} \left( \mult_{\rho}(H) - \sum_{i=1}^{5} w_{i} \right)+\epsilon (2 \mult_{\rho}(H) +3w_3 -2 (w_4 + w_5)) \leq 0.
\end{equation}
\end{enumerate}
If for some system of coordinates and a weight system $\rho$, the inequality \eqref{num-ss} is satisfied (resp., violated), we
will say that $(F,H)$ is \emph{$\rho$-semistable} (resp., \emph{$\rho$-unstable}.)
\end{definition}

%\begin{remark} While the degree $6$ generator of the ideal $(F,H)$ is defined only up to a multiple of $H$, when verifying semistability of $(F,H)$ with respect
%to a fixed weight system $\rho$, we can always fix a choice of $F$ that maximizes $\mult_{\rho}(F)$. Conversely, to show that $(F,H)$ is unstable with respect to 
%$\rho$, it suffices to show that \eqref{E:mult-3} fails for a single choice of $F$.
%\end{remark}

The following result establishes Theorem \ref{corti-theorem} (Corti's conjecture):
\begin{theorem}[{$\Rightarrow$ Theorem \ref{corti-theorem}}]\label{blackbox} \hfill
\begin{enumerate}
\item Every smooth degree $1$ del Pezzo $X_K$ over $K$ has a T-semistable $(6,3)$-intersection model over $R$.
\item Suppose $R=\cO_{C,p}$ is the local ring of a closed point of a smooth algebraic curve over an algebraically closed field $k$ with $\ON{char}(k)\neq 2, 3$. Suppose
$I=(F,H)$ is a T-semistable $(6,3)$-intersection model. Let $X:=\{F=H=0\}\subset \PP_R(1,1,2,3,3)$. Then $X\to \Delta$ is a standard del Pezzo fibration of degree $1$ over $\Delta$. 
\end{enumerate}
\end{theorem}

Part (1) of this theorem uses the machinery of Koll\'ar semistability (see Definition \ref{Stab-global-def}) % for intersections of hypersurfaces in weighted projective spaces. 
and necessitates a careful choice of a parameter space $M$ with a group action and of an invariant Cartier  divisor $\fD$ that governs stability. Specifically, it follows immediately 
from the following result that we establish in \S\ref{dP1-stab}:
\begin{Th}[={Proposition \ref{P:T-semistability}}] 
Let $M$ be the parameter space of intersections of weighted hypersurfaces of degree $6$ and $3$ in $\MP(1,1,2,3,3)$.  Let $G=\Aut_{gr}(\PP(1,1,2,3,3))$. 
Then there exists a $G$-invariant divisor $\fD^{ter}$ on $M$, supported on the locus of singular intersections, such that every $\fD^{ter}$-semistable model in $M$ is a T-semistable $(6,3)$-intersection.
\end{Th}
The reader who wishes to bypass entirely Sections \ref{kollar}, \ref{numeric} and \ref{examples}, can take Part (1) of Theorem \ref{blackbox} as a blackbox. 
Part (2) of Theorem \ref{blackbox} is proved as Theorem \ref{mainthm2} in Section \ref{Sing-and-stab},
where we use the numerical criterion of Definition \ref{T-semi} to verify that a T-semistable model is a standard del Pezzo fibration.

\subsection*{Acknowledgements} We would like to thank Brendan Hassett who suggested this collaboration. We would like to thank Ivan Cheltsov and Takuzo Okada for fruitful conversations. We benefitted from various research visits, which were made possible with support from a Heilbronn Institute's Focused Research Grant and the Loughborough University Institute of Advanced Studies. Present work took place while the third author %, Igor  Krylov,  
 was a Postdoc employed in the ERC Advanced grant n. 340258, `TADMICAMT' and later in Korea Institute for Advanced Study supported by KIAS Individual Grant n. MG069801.
The second author was partially supported by the NSA Young Investigator grant H98230-16-1-0061 and the Simons Collaboration Grant for Mathematicians 582030.
The first author is partially supported by EPSRC grant EP/T015896/1. The authors would like to thank the referees whose questions and suggestions led to a far more general theory in the final version.

%%%%%%%%%%%%%%%%%%%%%%%%%%%%%
%%%%%%%%%%%%%%%%%%%%%%%%%%%%%
\section{Koll\'ar stability}\label{kollar}
%%%%%%%%%%%%%%%%%%%%%%%%%%%%%
%%%%%%%%%%%%%%%%%%%%%%%%%%%%%

In this section, we introduce \emph{Koll\'ar stability}, a new notion of stability for PID-valued points of proper schemes admitting a group action.
It generalizes Koll\'ar's notion of %(semi)
stability for families of hypersurfaces over PIDs as developed in \cite{kollar-polynomials} and described in Section \ref{S:GL}.
A more technical %and descriptive 
name of $\fD$-semistability will be introduced and explained below.
 
We work over a fixed principal ideal domain $R$ with the fraction field $K$. We set $\Delta=\Spec(R)$.
For a maximal ideal $(t) \in \mspec{R}$, we denote by $\val_t\colon K \to \ZZ\cup \infty$ the corresponding $t$-valuation of $K$. 
Unless stated otherwise, all schemes and morphisms are over $R$.
  
For most applications, $R$ will be a DVR and, in our principal application---the proof of Corti's conjecture on degree $1$ del Pezzo fibrations, $R=\cO_{C,p}$, the local ring of a closed
point of a smooth curve $C$ over an algebraically closed field $k$.

%, and denote by $0$ the closed point of $\Delta$. 

\subsection{$\fD$-semistability} \label{stability-setup}
%Here, we define $\fD$-semistability. 
\subsubsection{} Suppose that $M$ is a proper scheme with an action of a group scheme $G$, that $\cL\in \Pic^{G}(M)$ is a $G$-linearized line bundle on $M$ and $\fD\in \HH^0(M, \cL)^{G}$
is a nonzero $G$-invariant section of $\cL$.
%\footnote{
%The quadruple $(G, M, \cL, \fD)$ is sometimes called a \emph{gauged Landau-Ginzburg model}.
%}  

%
The valuative criterion of properness gives canonical identifications $R$-valued and $K$-valued points of $M$:
\begin{equation}\label{pt-iden}
%\Mor_{R}(\Delta, M)%=\{f\co \Delta \to M \text{is an $R$-morphism}\}=
M(R)=M(K)=M(R_{(t)}) \quad \text{for every $(t)\in \mspec(R)$},
\end{equation}
and defines an action of the group $G(K)$ of $K$-points of $G$ on $M(R)$ as follows. 

\subsubsection{$G(K$)-action on $R$-valued points}\label{G(K)-action} Recall that $\Delta=\spec(R)$. Suppose $f\colon \Delta \to M$ is an $R$-point of $M$,
and $\rho\in G(K)$, then we define $\rho\cdot f\in M(R)$ to be the unique extension to $R$ of the $K$-point $\rho\cdot f(\spec K) \in M(K)$ given by \eqref{pt-iden}.

%%% Models 
We call $\rho\cdot f$ \emph{a model of $f$}. The set of all models of $f$ is then simply the $G(K)$-orbit of $f$ in $M(R)$:
\[
\Models(f)=\{\rho\cdot f \mid \ \rho\in G(K)\} \subset M(R).
\]

We are now in position to formulate our definition of stability:
\begin{definition}[$\fD$-(semi)stability]\label{D:stability} An $R$-point $f\colon \Delta \to M$ is \emph{$\fD$-semistable} if
\begin{enumerate}
\item\label{ss1} $f(\spec K) \notin \Supp(\fD)$ (equivalently, $f^*(\fD)\neq 0$ as a Cartier divisor on $\Delta$).
\item\label{ss2} \begin{equation}\label{D-minimizer} \deg f^*\fD:=\sum_{(t)\in \mspec(R)} \val_t(f^*\fD)\leq \sum_{(t)\in \mspec(R)} \val_t((f')^*\fD)=\deg (f')^*\fD\end{equation} for every model $f'$ of $f$.
\end{enumerate}
We say that $f\colon \Delta \to M$ is \emph{$\fD$-stable} if it is $\fD$-semistable and in addition every model $f'\colon \Delta\to M$ for which the equality holds in \eqref{D-minimizer}
satisfies $f'(\Delta)\in G(R)\cdot f(\Delta)$.
\end{definition}
In view of \eqref{D-minimizer}, we will also call a $\fD$-semistable model, a $\fD$-minimizer. Stability then means that the $\fD$-minimizer is unique, up to an action of $G(R)$. 
If $\rho\in G(K)$ is such that $$\sum_{(t)\in \mspec(R)} \val_t(f^*\fD)> \sum_{(t)\in \mspec(R)} \val_t((\rho\cdot f)^*\fD),$$ then we say that $\rho$ destabilizes $f$,
and that $f$ is unstable with respect to $\rho$, or simply $\rho$-unstable. 
%When $\fD$, and hence quadruple $(G, M, \cL, \fD)$, is understood, we call $\fD$-semistabiltiy a \emph{Koll\'ar stability}

For the applications we have in mind, the distinction between semistability and stability will play no role and so we rarely invoke the concept of $\fD$-stability. 

When $M$ is a parameter space representing a functor of flat families of schemes, such as for example the Hilbert scheme of a fixed projective space, 
the $R$-points of $M$ are simply fibrations over $\Delta$ fibered in the objects parameterized by $M$. Our applications will be essentially of this form. 

\subsubsection{Reduction to DVRs} % Localization of stability} 
%We can ask whether an $\fD$-semistable $R$-point of $M$ defines an $\fD$-semistable $R_{(t)}$-point of $M$ 
%for every maximal ideal $(t)$ of $R$.
\begin{lemma}\label{L:localization}
Suppose a group scheme $G$ satisfies the following condition: 

$(\dagger)$ For every $\rho\in G(K)$, there exist finally many irreducible elements $\{t_i\}_{i=1}^{n}\in R$ and elements $\rho_i\in G(R[1/t_i])$ such that 
\begin{equation*}
%\tag{$\dagger$}
\rho=\rho_1\cdots \rho_n.
\end{equation*}
Then an $R$-point $f\in M(R)$ is $\fD$-semistable if and only if the induced point $f\in M(R_{(t)})$ is $\fD$-semistable for every $(t)\in \mspec(R)$.
\end{lemma}
\begin{proof}
If $\rho_i\in G(R[1/t_i])$, then $\rho=\rho_1\cdots \rho_n\in G(K)$ destabilizes $f\in M(R)$ if and only if some $\rho_i$ destabilizes $f\in M(R_{(t_i)})$.
\end{proof}

Condition $(\dagger)$ of Lemma \ref{L:localization} is satisfied for $\GL_n$ and $\SL_n$ by the elementary divisor theorem, and for $G=\Aut_{gr}(\Cox(\PP(c_1,\dots,c_n)))$,
the group of graded automorphisms of the Cox ring of a weighted projective space by a similar argument. These are the only groups considered in our applications.  

\subsubsection{Stack-theoretic interpretation}  The quadruple $(G, M, \cL, \fD)$ defines quotient stack $\mathfrak{M}:=[M/G]$, a 
line bundle $\cL$ on $\mathfrak{M}$, and a global section $\fD$ of $\cL$.  We then say that an $R$-point $f\colon \Delta \to \mathfrak{M}$ is $\fD$-semistable 
if $f^*(\fD) \neq 0$ and $\val_t(f^*\fD) \geq \val_t((f^*)' \fD)$ for every other $R$-point $f'$ of $\mathfrak{M}$ such that $f(\spec K)$ and $f'(\spec K)$ 
are isomorphic $K$-points of $\mathfrak{M}$. There is however a subtle distinction between the resulting $\fD$-stability on the quotient stack $\mathfrak{M}$ and on the original 
space $M$ arising from the possibility that $R$-points of $M$ do not necessarily surject onto the $R$-points of $\mathfrak{M}$. Working over a field $k$,  if $R$ is a complete ring, or $G$ is a special group
in the sense of Serre \cite{serre-special} (see also \cite{special-groups} for a modern exposition)
every $R$-point of $\mathfrak{M}$ comes from an $R$-point of $M$: In these cases every \'etale $G$-torsor over $\Delta$ is a Zariski $G$-torsor,
and so the two notions of $\fD$-semistability of $R$-points of $M$ and $\mathfrak{M}$ are equivalent.

%\subsubsection{Hilbert-Mumford-Koll\'ar index}
\subsection{Numerical criterion for $\fD$-semistability}
With Lemma \ref{L:localization} in mind, we now assume that $R$ is a DVR. 
\subsubsection{Hilbert-Mumford-Koll\'ar index}\label{HMK} Suppose $f\in M(R)$ and $\rho\in G(K)$. Let
$f'=\rho\cdot f\in M(R)$ as defined in \eqref{G(K)-action}.
Fix isomorphisms 
\begin{align}\label{E:i-isoms}
i\co f^*(\cL) &\simeq \cO_{\Delta}=\widetilde{R}, \\
i'\co (f')^*(\cL) &\simeq \cO_{\Delta}=\widetilde{R},
\end{align}
where $\widetilde{R}$ is simply the coherent sheaf given by the free rank one $R$-module on $\spec R$. (These isomorphisms are defined up to a unit in $R$.)

The $G$-linearization of $\cL$ induces a $K$-linear isomorphism 
\[
\rho^{-1}\colon \cL_{f'(\spec K)} \rightarrow \cL_{f(\spec K)}
\] and so, by composition, we obtain an isomorphism
\begin{equation}\label{E:isom-HMK}
K=\cO_{\spec K} \simeq \cL_{f'(\spec K)} \rightarrow \cL_{f(\spec K)} \simeq \cO_{\spec K}=K.
\end{equation}
By construction, the isomorphism in \eqref{E:isom-HMK} is given by a canonical element in $K/R^*$. We denote the valuation of this element 
by $\mu_{\rho}^{\cL}(f)$.  In plain terms, if $e_{f'}$ is the generator of $ \cL_{f'(\spec K)}$ that extends to a generator of $\cL\vert_{f'(\Delta)}$, then \[\rho^{-1}(e_{f'})=t^{\mu_{\rho}^{\cL}(f)}(e_f),\]
where $e_{f}$ is a generator of $ \cL_{f(\spec K)}$ that extends to a generator of $\cL\vert_{f(\Delta)}$.  We call $\mu_{\rho}^{\cL}(f)$ 
the \emph{Hilbert-Mumford-Koll\'ar index of $f$ with respect
to $\rho$} (abbreviated as HMK-index). Our naming convention will be explained below. 

It is immediate from the above definition that
for every $G$-invariant regular section $\fD$ of $\cL$, we have 
\[
(f^*\fD)=(t^{\mu^{\cL}_{\rho}(f)}(f')^*\fD),
\]
and so 
\begin{equation}\label{E:valulations}
\val_t(f^*\fD))=\val_t((f')^*\fD)+\mu^{\cL}_{\rho}(f).
\end{equation}

It follows that
\begin{prop}
A map $f\colon \Delta \to M$ 
is $\fD$-semistable if and only if $\mu^{\cL}_{\rho}(f)\leq 0$ for all $\rho\in G(K)$.
\end{prop}

We record for the future use some functorial properties of the HMK-index that follow immediately from the definition:
\begin{prop}\label{P:HMK-functorial} Suppose $f\colon \Delta \to M$ is an $R$-point and 
$\rho, \phi \in G(K)$.
The following holds:
\begin{enumerate}
\item If $\cL_1$ and $\cL_2$ are two $G$-linearized line bundles on $M$, then 
\begin{equation}
\mu_{\rho}^{\cL_1+\cL_2}(f)=\mu_{\rho}^{\cL_1}(f)+\mu_{\rho}^{\cL_2}(f).
\end{equation}
%for any two 
\item If $h\colon M \to N$ is a $G$-morphism and $\cL$ is a $G$-linearized line bundle on $N$, 
then 
\begin{equation*}
\mu_{\rho}^{f^*\cL}(f)=\mu_{\rho}^{\cL}(h\circ f).
\end{equation*}
\item $\mu_{\phi\circ \rho}^{\cL}(f)= \mu_{\phi}^{\cL}(\rho\cdot f)+\mu_{\rho}^{\cL}(f)$.
\item If $\rho\in G(R)$, then $\mu_{\rho}^{\cL}(f)=0$.
\end{enumerate}
\end{prop}

%%%%%%%%%%%%%%%%%%%%%%%%%%%%%
\subsubsection{One-parameter subgroups}\label{1-PS}
%%%%%%%%%%%%%%%%%%%%%%%%%%%%%
Fix a morphism $\iota\colon \spec(K) \to \GG_{m,R}=\spec R[z,z^{-1}]$ given by $z\mapsto t$. 
Then for every one-parameter subgroup 
$\rho\colon \GG_{m,R}\to G$, the morphism $\iota$ defines a corresponding 
$K$-point of $G$ given by $\rho \circ \iota \colon \spec K \to G$.  Following Mumford,
we denote the resulting $K$-point of $G$ by $\langle \rho \rangle$. We also write
$\mu^{\cL}_{\rho}$ instead of $\mu^{\cL}_{\langle \rho\rangle}$.

For every one-parameter subgroup $\rho\colon \GG_{m,R}\to G$ and every $\rho$-fixed $K$-point $x\in M(K)$,
each of the line bundles $\cL\in \Pic^{G}(M)$ defines a one-dimensional representation $\cL_{x}$ of $\GG_{m,K}$. Then the integer obtained by pairing of this character with $\rho$ is 
precisely $\mu^{\cL}_{\rho}(x)$. Plainly, 
$\rho$ acts on the fiber $\cL_{x}$ by 
\begin{equation}\label{E:character}
z\cdot w=t^{\mu^{\cL}_{\rho}(x)}w,  \quad \text{for $w\in \cL_{x}$}.
\end{equation}

\subsubsection{Cartan-Iwahori-Matsumoto decomposition}
Recall that for $G=\GL(n)$ and $G=\SL(n)$, the elementary divisors theorem 
says that the double coset of every element of $G(K)$ with respect to the subgroup $G(R)$
contains an element of the form $\langle \rho\rangle$ for some one-parameter subgroup
$\rho$ of $G$. (If the DVR $R$ is complete with an algebraically closed residue field, this is true more generally for reductive groups by a theorem of Iwahori \cite[p.52]{GIT}).
In this case, we say that $G$ has Cartan-Iwahori-Matsumoto decomposition, see \cite{cartaniwahorimatsumoto}.

Combining properties (3) and (4) of Proposition \ref{P:HMK-functorial}, we have: 
\begin{prop}\label{reductive-NC} Suppose $G$ has Cartan-Iwahori-Matsumoto decomposition.
Then a map $f$ is semistable if and only if $\mu^{\cL}_{\rho}(f)\leq 0$ for all one-parameter
subgroups $\rho\colon \GG_{m,R}\to G$. % that are $M$-adapted to $f$. 
\end{prop}

\begin{remark} If we work over a base field $k$, and $f\colon \Delta \to M$ is a constant map whose image is the $k$-point 
$x\in X$, then $\mu^{\cL}_{\rho}(f)$ is the usual Hilbert-Mumford index of $x$ with respect to $\cL$ and $\rho$.
This explains naming $\mu^{\cL}_{\rho}(f)$ Hilbert-Mumford-Koll\'ar index.
   %%% Can there be constant polynomials over Z? Not really, since Z is the final object. 
\end{remark}

%%%%%%%%%%%%%%%%%%%%%%%%%%%%%

\subsection{Motivating example apr\`es Koll\'ar}\label{S:GL}
 
In this subsection, we recast Koll\'ar's notion of stability for homogeneous polynomials,
which motivates our whole approach, in the language of $\fD$-semistability.   We work over a field $k$ in this example.

Let $M=\overline{M}=\PP(U)$\footnote{Our convention is that $\PP(U)$ means the space of lines in $U$.}, where $U$ is an algebraic representation of a reductive algebraic group $G$.

We have $\Pic^{G}(M)=\ZZ\oplus \Ch(G)$, where $\Ch(G)$ is the character group of $G$. 
The first summand of $\Pic^{G}(M)$ is generated by $\cO(1)$ with its canonical 
$G$-linearization induced by the $G$-action on $U$, and the second by the $G$-linearized
line bundles
 \[\{\cO^{\chi}\mid \chi\in \Ch(G)\}\] 
where $\cO^{\chi}$ is the trivial line bundle linearized by the character $\chi$.  

Suppose $\cL=\cO(m)^{\chi}:=\cO(m)\otimes \cO^{\chi}$, and $\fD$ is a $G$-invariant section 
of $\cL$.  We now explicate the numerical criterion of Proposition \ref{reductive-NC} for $\fD$-semistability of a map $f\colon \Delta \to M$.

Suppose $\dim U=n$ and $\rho$ is a 1-PS (one-parameter subgroup) $\rho$ of $G$. Choose a basis $u_1,\dots, u_n$ of $U$ on
which $\rho$ acts diagonally as
\[
\rho(t)=\diag (t^{w_1},\dots, t^{w_n}).
\]
Consider now an $R$-point $f\colon \Delta\to M$ given by an equation $F(t)=\sum_{i=1}^n F_i(t)u_i \in U\otimes_k R$,
where $F_i(t)\in R$, and $F(0) \neq 0\in  U$. Then we have the following notion of multiplicity as defined by Koll\'ar:
\begin{equation}\label{E:mult}
\mult_{\rho}(F):=\max \{N \mid \rho\cdot F(t) \in (t^N)\}=\min \{\val_t \left(F_i(t)t^{w_i}\right)\mid i=1,\dots, n\}.
\end{equation}
Setting 
\begin{equation}\label{E:f-rho}
F^{\rho}(t):=\frac{\rho\cdot F(t)}{t^{\mult_{\rho}(F)}} = \frac{1}{t^{\mult_{\rho}(F)}} \sum_{i=1}^n F_i(t)t^{w_i} u_i,
\end{equation}
we see that the $R$-point $\rho \cdot f\colon \Delta \to \PP(U)$ is given precisely by the equation $F^{\rho}(t)\in R\otimes U$.

\begin{lemma} With setup as above, we have
\begin{align}
\mu^{\cO(1)}_{\rho}(f)&=\mult_{\rho}(F(t)), \label{HMK-O1} \\
\mu^{\cO^{\chi}}_{\rho}(f)&=-\langle \chi, \rho\rangle, \label{HMK-chi} 
\end{align}
where $\langle \chi, \rho\rangle$ is the integer obtained by pairing $\chi$ and $\rho$. 
Consequently, for $\cL=\cO(m)^{\chi}$,
\begin{equation}
\mu^{\cL}_{\rho}(f)=m\mult_{\rho}(F(t))-\langle \chi, \rho\rangle.
\end{equation}

\end{lemma}

\begin{proof} We first establish \eqref{HMK-O1}. By Proposition \ref{P:HMK-functorial}, we have
\[
\mu^{\cO(1)}_{\rho}(f)=-\mu^{\cO(-1)}_{\rho}(f).
\]
Since $f\colon \Delta \to \PP(U)$ is given by $F(t) \in R\otimes U$ with $F(0)\neq 0$, 
we have a canonical isomorphism \[
\HH^0(\Delta, f^*\cO(-1))=R F(t).
\]  
Similarly, we have a canonical isomorphism 
\[
\HH^0(\Delta, (\rho\cdot f)^*\cO(-1))=R F^{\rho}(t).
\]
Since $t^{\mult_{\rho}(F)}=\dfrac{\rho\cdot F(t)}{F^\rho(t)}$, by definition, 
we see that $\mu^{\cO(-1)}_{\rho}(f)=-\mult_{\rho}(F(t))$ and the claim follows.

For \eqref{HMK-chi}, note that 
$\rho$ sends the generator $1$ of $\HH^0(\Delta, f^*\cO^{\chi})=R$ to $t^{\langle \chi, \rho\rangle}$ times the generator $1$ of  $\HH^0(\Delta, (\rho\cdot f)^*\cO^{\chi})=R$.

\end{proof} 

\subsubsection{Homogeneous polynomials} \label{def-equiv}
Suppose $V$ is an $n$-dimensional vector space and $G=\GL(V)$. Then 
$U=\Sym^d V$ is the space of degree $d$ homogeneous polynomials in $n$ variables. The group $G$ has exactly one character, namely the determinant $\chi=\det$.
We take $\fD$ to be the discriminant divisor on $\PP(U)$.  Then $\fD$ is a $G$-invariant section of $\cL=\cO(m)^{-\frac{m}{n}\chi}$, where $m=n(d-1)^{n-1}$ is the degree of
the discriminant.  
We see that for $F\in R\otimes U$, with $F(0)\neq 0$, our definition of multiplicity of $F$ with respect to any one-parameter subgroup $\rho$ of $G$ coincides with Koll\'ar's definition 
of multiplicity. Furthermore, the Hilbert-Mumford-Koll\'ar index of $f$ with respect to $\rho$ acting in some basis of $V$ diagonally with weights $w_1,\dots, w_n$ is 
\begin{equation}\label{E:homogeneous-HMK}
\mu_{\rho}^{\cL}(f)=m\left(\mult_{\rho}(F)-\frac{d}{n}(w_1+\cdots+w_n)\right).
\end{equation}
Summarizing, we conclude that our definition of $\fD$-semistability in $\PP(\Sym^{d} V)$ coincides with Koll\'ar's definition of semistability for generically smooth families
of degree $d$ hypersurfaces in $\PP^{n-1}$ as given in \cite[Definition (3.3)]{kollar-polynomials}.

%%%%%%%%%%%%%%%%%%%%%%%%%%%%%
\section{Koll\'ar stability of weighted hypersurfaces}\label{numeric}

In this section, we develop Koll\'ar stability of weighted (Cartier) hypersurfaces in arbitrary weighted projective spaces, generalizing Koll\'ar's theory for ordinary
hypersurfaces in \cite{kollar-polynomials}. We start with some generalities.  While cumbersome at the first glance, they allow us to work over an arbitrary DVR. 

\subsection{Parameter spaces with group action} 
Fix $n$ positive integers
%\begin{equation}\label{E:weights-c}
$\{c_i\}_{i=1}^n$
%\end{equation}
and let $S:=\ZZ[x_{1},\dots, x_n]$ be the graded ring with grading 
given by $\ON{wt}(x_{i})=c_i$. We let $\PP_{\ZZ}(c_1,\dots,c_n)=\proj_{\ZZ} S$.
Given a ring $A$, we let $S_A:=S\otimes A=A[x_1,\dots,x_n]$ and 
$\PP_A(c_1,\dots,c_n)=\proj_{A} S_A$. %and $M_A:=M\times \spec(A)$. 
We set 
\begin{equation}
G=\Aut_{gr} \Cox(\PP(c_1,\dots,c_n))
\end{equation}
to be the group scheme of graded automorphisms of the Cox ring of $\PP(c_1,\dots,c_n)$. %; precisely,  
For a ring $A$
we have that $G(A)=\Aut_{gr}(S_A)$ is the group of graded $A$-algebra automorphisms
of $S_A$.

We emphasize that the generators $x_i$'s of $A[x_1,\dots, x_n]$ are not fixed throughout, but 
are determined only up to an element of $G(A)$. Any such choice 
of generators with $\ON{wt}(x_{i})=c_i$ will be called a system of quasihomogeneous coordinates (or simply, coordinates) in $S_A$. 

We will be interested in two kinds of parameter spaces associated to $\PP(c_1,\dots,c_n)$. To define them, let $\pi\colon \PP_{\ZZ}(c_1,\dots,c_n)\to \spec \ZZ$ be the structure morphism.
The first is the space of degree $d$ hypersurfaces in $\PP(c_1,\dots,c_n)$:
\begin{equation}
\Hyp(d):=\mathbf{P}(\pi_*\cO_{\PP(c_1,\dots,c_n)}(d)).
\end{equation}
Note that $\Hyp(d)$ is smooth and projective over $\spec \ZZ$ and $\Pic(\Hyp(d))\simeq \ZZ$. 

Suppose now $e<d$ are positive integers.
Let $q\colon \cQ\to \Hyp(e)$ be the universal degree $e$ hypersurface, where $\cQ\hookrightarrow \Hyp(e)\times_{\spec \ZZ} \PP_{\ZZ}(c_1,\dots,c_n)$. 
We define 
\begin{equation}
\Int(d,e):=\mathbf{P}(q_* \cO_{\cQ}(d)),
\end{equation}
to be the space of $(d,e)$-intersections in $\PP(c_1,\dots,c_n)$.
Note that $\Int(d,e)$ is smooth and projective over $\spec \ZZ$ and $\Pic(\Int(d,e))\simeq \ZZ^2$. 

The group scheme $G$ acts on $\Hyp(d)$ and $\Int(d,e)$ via its natural action on $S$. This will be the only group action we consider for the rest of the paper. 

 \subsection{Stability of DVR-valued points} Suppose now $R$ is a DVR with a uniformizer $t$ and the fraction field $K$. We let $k=R/(t)$ be the residue field.
 For $F\in S_R=R[x_1,\dots,x_n]$, we denote by $F_0$ its image in $S_k=k[x_1,\dots,x_n]$. 
 
 Our goal is to understand stability of $R$-points of $\Hyp(d)$ and $\Int(d,e)$
 with respect to various $G$-linearized line bundles on these parameter spaces. We begin with describing $R$-points of $\Hyp(d)$ and $\Int(d,e)$:
 
 % Then \begin{equation}\label{E:weighted-P}
%\PP_R:=\proj_{R} S_R=\PP_R(c_1,\dots,c_n)=\PP_R(v_1^{r_1},\dots,v_s^{r_s})
%\end{equation} is an $(n-1)$-dimensional 
%weighted projective space over $R$. 

\subsubsection{}
An $R$-point $f\colon \spec(R)\to \Hyp(d)$ is simply an element $F\in R[x_1,\dots,x_n]_d$ 
such that $F_0 \neq 0 \in k[x_1,\dots,x_n]$. We call such $F$ the equation of $f$.  (It is of course, defined up to a unit in $R$.) 
%\end{definition}

%where $\mathbf{P}(\cE)$ stands for the pre-Grothendieck projectivization of a locally free sheaf \cE: \Proj \Sym \cE^{\vee}

\subsubsection{} An $R$-point $f\colon \spec(R)\to \Int(d,e)$ is an ideal $I=(F,H)\subset R[x_1,\dots,x_n]$ such 
that $H\in R[x_1,\dots,x_n]_{e}$ and $F\in R[x_1,\dots,x_n]_{d}$ and 
such that $H_0\neq 0 \in k[x_1,\dots,x_n]_{e}$ and $F_0\notin (H_0)$. We call such ideal $I$ a $(d,e)$-intersection, and $F$ and $H$ the equations of $I$.  
Note that $H$ is defined up to a unit in $R$, but $F$ is defined only up to addition of a multiple of $H$. 

\subsubsection{The action of $G(K)$ on $R$-points of $\Hyp(d)$} 
Suppose $\rho\in G(K)$ and $f\colon \spec(R) \to \Hyp(d)$ is an $R$-point with the equation 
$F \in R[x_1,\dots,x_n]_d$.  Just as in the case of homogeneous polynomials (cf. \eqref{E:mult}), we define the $\rho$-multiplicity of $F$ to be:
\begin{equation}\label{E:mult-2}
\mult_{\rho}(F):=\max \left\{N \mid \frac{\rho\cdot F}{t^N}\in  R[x_1,\dots,x_n]\right\}.
\end{equation}
Set
\begin{equation}\label{E:f-rho-2}
F^{\rho}:=\frac{\rho\cdot F}{t^{\mult_{\rho}}(F)}.
\end{equation}
Then $F^{\rho}\in R[x_1,\dots,x_n]$ but $F^{\rho}_0\neq 0 \in k[x_1,\dots,x_n]$ and so the equation of the $R$-point $\rho \cdot f$ is precisely $F^{\rho}$. %\in R[x_1,\dots,x_n]_{d}$.

\subsubsection{The action of $G(K)$ on $R$-points of $\Int(d,e)$} \label{equation-int}
Similarly, for a fixed $\rho\in G(K)$ and an $R$-point $f\colon \spec(R) \to \Int(d,e)$ given by a $(d,e)$-intersection ideal $I=(F,H)$,  %%%% satisfying the condition  $F_0 \notin (H_0)$
where $F$ and $H$ are equations of $I$ as above, we consider $H^{\rho}=\dfrac{\rho\cdot H}{t^{\mult_{\rho}(H)}}$ and $F^{\rho}=\dfrac{\rho\cdot F}{t^{\mult_{\rho}(F)}}$.
If
\begin{equation}\label{E:independence}
(F^{\rho})_0 \notin (H^{\rho})_0 \subset k[x_1,\dots,x_n],
%\frac{\rho \cdot F}{t^{\mult_{\rho}(F)}} \Big\vert_{t=0} 
%\notin \left(\frac{\rho \cdot H}{t^{\mult_{\rho}(H)}} \Big\vert_{t=0}\right).
\end{equation}
then the generators $(F,H)$ are called the equations of $f$ adapted to $\rho$. We can always choose a representation $F^{\rho}=QH^{\rho}+t^{q} A$ in $S_R$ with the maximal possible $q\geq 0$
(and $q=0$ if and only if $(F,H)$ are adapted to $\rho$). Then the $R$-point $\rho\cdot f$ is given by the $(d,e)$-intersection ideal $(A, H^{\rho})$. 

\subsubsection{One-parameter subgroups} 
We will say that $\rho\in G(K)$ is a one-parameter subgroup of $G$ if in some coordinates $x_1, \dots, x_n$ %(resp., $x_{ij}$) 
on $S_R=R[x_1,\dots,x_n]$, we have that $\rho$ acts diagonally:
\[
\rho\cdot x_i = t^{w_i} x_i. %, \qquad (\text{resp.,}\  \rho \cdot x_{ij}=t^{w_{ij}}x_{ij}).
\]
The vector %of weights 
$(w_1,\dots,w_n)\in \ZZ^n$ will be called a weight system of $\rho$, and if $x_i$'s are understood, we will often write $\rho=(w_1,\dots,w_n)$ to define $\rho$. 

For an integer $v\in \{c_1,\dots,c_n\}$, we denote by
\[
%W_v(\rho):=\sum_{\text{$x_j$ has weight $v$ in $S$}} w_j,
W_v(\rho):=\sum_{\ON{wt}(x_j)=v} w_j,
\]
the sum of the $\rho$-weights of the coordinates of graded weight $v$ in $S$. 

\begin{remark} 
The name one-parameter subgroup comes from a group scheme morphism $\GG_{m,R}:=\spec (R[z, z^{-1}]) \to G$ and the $K$-point of $\GG_{m,R}$ given by $t\in K$; see \ref{1-PS}.
\end{remark}

\subsection{$G$-linearized line bundles} 

To obtain a numerical criterion for Koll\'ar semistability of an $R$-point $f\colon \spec(R)\to M=\Hyp(d)$ (or $M=\Int(d,e)$) %$F\in R[x_1,\dots,x_n]_d$  
in terms  of its defining equation over $R$,
we must compute the Hilbert-Mumford-Koll\'ar index $\mu^{\cL}_{\rho}(f)$ 
for every $G$-linearized line bundle $\cL\in \Pic^{G}(M)$ and every one-parameter subgroup $\rho$ of $G$. 

To understand $\Pic^{G}(M)$, note that every $G_K$-linearized line bundle on $M_K$ extends uniquely to a $G_R$-linearized line bundle on $M_R$.  
Working over the field $K$, so that the standard theory (and language) of algebraic groups over a field applies,
we have that 
$\Pic^{G}(M)=\Pic(M)\oplus \Ch(G)$, where $\Ch(G)$ is the character group of $G$. 
\begin{remark} 
The $\Ch(G)$-portion of $\Pic^{G}(M)$ is the pullback of $\Pic^{G}(\Spec(R))$ so is independent of $M$. 
\end{remark}

Suppose there are exactly $s$ distinct integers among $c_i$'s, namely \begin{equation} v_1 < v_2 < \cdots < v_s.\end{equation}
For each $i=1,\dots,s$, we have a character of $G_K$ given by the determinant of a $K$-linear transformation obtained by restricting an element of $G_K$ to 
the $K$-subspace \[
K\langle x_j \mid \text{$x_j$ has weight $\leq v_i$ in $S$}\rangle \subset S/(x_1,\dots,x_n)^2.
\]
The resulting $s$ characters freely generate the character group of $G_K$. It follows that $\Ch(G)\simeq \ZZ^s$ and so
\begin{equation}
\Pic^{G}(M) \simeq  \Pic(M)\oplus \ZZ^{s}.
\end{equation} 

\subsubsection{Irrelevant one-parameter subgroup} \label{irrelevant}
Notice that given a system of quasihomogeneous coordinates $x_1,\dots, x_n$, we have the irrelevant one-parameter subgroup $\rho_{irr}$
acting via
\[
\rho_{irr} \cdot (x_1,\dots, x_n)=(t^{c_1}x_1, \dots, t^{c_n}x_n).
\]
Since $\rho_{irr}$ fixes all the points of $M$, a necessary
condition for a $G$-linearized line bundle $\cL$ to have a nonzero $G$-invariant section is 
that $\mu_{\cL}^{\rho_{irr}}([X])=0$ for some (equivalently, every) $[X]\in M$. 
This implies that the subspace of $\Pic^{G}(M)\otimes \QQ$ 
spanned by $G$-linearized line bundles with invariant nonzero sections has 
dimension at most $s$.

\subsubsection{Tautological line bundles}
We treat parameter spaces $M=\Hyp(d)$ and $M=\Hyp(d,e)$ at the same time.  Let $\pi\colon \PP(c_1,\dots,c_n) \to \spec R$ be the structure morphism.

Suppose $q\colon \cQ\to M$ is the universal family over $M$, where $\cQ \hookrightarrow M\times_{R} \PP(c_1,\dots,c_n)$. 
(Note that $q$ is flat when $M=\Hyp(d)$ but not when $M=\Hyp(d,e)$). Let $\cI$ be the ideal sheaf 
of $\cQ$ in $M\times_{R} \PP(c_1,\dots,c_n)$ and $\pr_1\colon M\times_{R} \PP(c_1,\dots,c_n)\to M$ the projection morphism.

We obtain the following $G$-linearized line bundles %\footnote{The choice of inverse/sign here is a matter of personal taste.}
\begin{align}
T_m&:=\det \pi_*\bigl(\cO_{\PP_R}(m)\bigr)=\det (\pr_1)_*\bigl(\pr_2^*\cO_{\PP_R}(m)\bigr), m\geq 1 \\
J_m&:=\det (\pr_1)_*\bigl(\cI(m)\bigr), \quad  \text{where} \ \begin{cases} m\geq d, \text{if $M=\Hyp(d)$} \\ d\geq m\geq e, \text{if $M=\Int(d,e)$.} \end{cases}
\end{align}
Note that $(\pr_1)_*\bigl(\cI(m)\bigr)$ is a vector bundle on $\Int(d,e)$ for $m\leq d$.

%\begin{definition}
%Every $R$-point $f\colon \Delta \to M$ is given by a unique (up to a unit in $R$) equation 
%$F\in S_{d}$ satisfying $F \not\in (t)$.  
%We call such $F$ the equation of $f$.
%\end{definition}

\subsubsection{Hilbert-Mumford-Koll\'ar indices}
Let $\rho$ be a one-parameter subgroup acting by $\rho\cdot x_{i}=t^{w_{i}} x_{i}$ in some coordinates. 
Let $\ON{mon}(S_r)$ be the set of all monomials in $S_r$ in same coordinates. 
%Set $\mathfrak{M}_r$ %=\{x_1^{d_1}\cdots x_n^{d_n} \mid \sum_{i=1}^n c_id_i=r\}
%to be the set of all monic degree $r$ monomials in $S$.
Denote the $\rho$-weight of a monomial by $w_\rho (\prod x_{i}^{d_{i}}) = \sum d_{i} w_{i}$ and by $w_{\rho}(S_r)$ the sum of the $\rho$-weight of all 
monomials in $S_r$.
The following result computes HMK-indices of the tautological line bundles on $M$: 
%\marginpar{Decide if weights are $\rho_i$ or $w_i$}
\begin{prop}\label{HMK-tautological} Suppose $\rho$ is a one-parameter subgroup acting by $\rho\cdot x_{i}=t^{w_{i}} x_{i}$ in some coordinates. 
\item[(A)] $M=\Hyp(d)$. For $f\colon \spec(R) \to M$ given by an equation $F=F(x_1,\dots,x_n)\in S_d$, we have 
\begin{align}
\mu_{\rho}^{T_m}(f) &= -w_{\rho}(S_{m}), \\
\mu_{\rho}^{J_m}(f) &= -(w_{\rho}(S_{m-d})+\vert \ON{mon}(S_{m-d})\vert \mult_{\rho}(F)).
\end{align}
\item[(B)] $M=\Int(d,e)$. For $f\colon \spec(R) \to M$ given by a $(d,e)$-intersection ideal $(F,H)$, we have
\begin{align}
\mu_{\rho}^{T_m}(f) &= -w_{\rho}(S_{m}), \\
\mu_{\rho}^{J_m}(f) &= -(w_{\rho}(S_{m-e})+\vert \ON{mon}(S_{m-e}) \vert \mult_{\rho}(H)), \qquad e\leq m\leq d-1, \\
\mu_{\rho}^{J_d}(f) &\leq -(w_{\rho}(S_{d-e})+\vert \ON{mon}(S_{d-e}) \vert \mult_{\rho}(H)+\mult_{\rho}(F)). \label{HMK-int}
\end{align}
If $(F,H)$ are adapted to $\rho$, then equality holds in \eqref{HMK-int}.
%and consequently 
%\begin{equation}\label{E:HMK-Lambda_m}
%\mu_{\rho}^{\Lambda_m}(f) = 
%\vert \mathfrak{M}_{mC-d}\vert \mult_{\rho}(F_t)+w_{\rho}(\mathfrak{M}_{mC-d})-w_{\rho}(\mathfrak{M}_{mC}).
%\end{equation}
%and consequently 
%\begin{equation}\label{E:HMK-Lambda_m}
%\mu_{\rho}^{\Lambda_m}(f) = 
%\vert \mathfrak{M}_{mC-d}\vert \mult_{\rho}(F_t)+w_{\rho}(\mathfrak{M}_{mC-d})-w_{\rho}(\mathfrak{M}_{mC}).
%\end{equation}
\end{prop}
\begin{proof} (A) $M=\Hyp(d)$.
For $f\colon \Delta \to M$ with the equation 
$F\in S_{d}$, %,satisfying $F_0 \neq 0$, 
we have natural identifications of $\cO_{\Delta}$-modules:
\begin{align}
f^*(J_m) &=\widetilde{R \bigwedge_{p\in \ON{mon}(S_{m-d})} pF}, \\
f^*(T_m) &= \widetilde{R \bigwedge_{p\in \ON{mon}(S_{m})}p}.
\end{align}

Similarly, $\rho\cdot f\colon \Delta \to M$ has equation
$F^{\rho}$ and 
we also have natural identifications: 
\begin{align}
(\rho\cdot f)^*(J_m) &= \widetilde{R \bigwedge_{p\in \ON{mon}(S_{m-d})} pF^{\rho}}, \\
(\rho\cdot f)^*(T_m) &= \widetilde{R \bigwedge_{p\in \ON{mon}(S_{m})} p}.
 \end{align}
Since 
\begin{align*}
\rho\cdot \left(\bigwedge_{p\in \ON{mon}(S_{m-d})} pF \right) &= \bigwedge_{p\in \ON{mon}(S_{m-d})} (\rho\cdot p) (\rho\cdot F)
=\bigwedge_{p\in \ON{mon}(S_{m-d})} (t^{w_{\rho}(p)} p)(t^{\mult_{\rho}(F)} F^{\rho}), \\ \\
&=t^{w_{\rho}(S_{m-d})+\vert \ON{mon}(S_{m-d})\vert \mult_{\rho}(F)} \bigwedge_{p\in \ON{mon}(S_{m-d})} pF^{\rho}. \\
\rho\cdot \left(\bigwedge_{p\in \ON{mon}(S_m)} p\right) &=\bigwedge_{p\in \ON{mon}(S_m)} (t^{w_{\rho}(p)} p),
\end{align*}
the claims follow from the definition of the HMK-index in \S\ref{HMK}.

(B) $M=\Hyp(d,e)$. The first two equalities are identical to Part (A), so we only need to prove \eqref{HMK-int}. Suppose $f\colon \Delta \to M$ is given by a $(d,e)$-intersection $(F,H)$.  
We have natural identifications of $\cO_{\Delta}$-modules:
\begin{equation}
f^*(J_d) = \widetilde{R \quad F\wedge \bigwedge_{p\in \ON{mon}(S_{d-e})} pH}.
\end{equation}
The point $\rho\cdot f\colon \Delta \to M$ has equation $(A, H^{\rho})$, where $F^{\rho}=t^{q} A \pmod{(H^{\rho})}$, where $q\geq 0$, see \S\ref{equation-int}.
Since 
\begin{multline}
\rho \cdot \left( F\wedge \bigwedge_{p\in \ON{mon}(S_{d-e})} pH\right)=t^{\mult_{\rho}(F)} F^{\rho}\wedge  \bigwedge_{p\in \ON{mon}(S_{d-e})} (t^{w_{\rho}(p)} p)(t^{\mult_{\rho}(H)} H^{\rho}) \\
=t^{q+\mult_{\rho}(F)+w_{\rho}(S_{d-e})+\vert \ON{mon}(S_{d-e})\vert \mult_{\rho}(H)} A \wedge \bigwedge_{p\in \ON{mon}(S_{d-e})} pH^{\rho},
\end{multline}
we have 
\[
\mu_{\rho}^{J_d}(f) = -(q+w_{\rho}(S_{d-e})+\vert \ON{mon}(S_{d-e}) \vert \mult_{\rho}(H)+\mult_{\rho}(F)),
\]
and the claim follows.
\end{proof}

\begin{corollary}\label{C:HMK} Let $M=\Hyp_{\PP(c_1,\dots,c_n)}(d)$. Let $v_1<\cdots <v_s$ be the distinct integers among $\{c_1,\dots,c_n\}$.  
For a one-parameter subgroup $\rho$ with weight system $(w_1,\dots,w_n)$, we will write $W_{v_j}(\rho)=\sum_{i: c_i=v_j} w_i$, and 
let $r_j$ denote the number of variables with weight $v_j$.  

{\rm (1)} For every $G$-linearized line bundle $\cL \in \Pic^{G}(M)$, there exist unique rational numbers $\{a_i(\cL)\}_{i=0}^{s}$ % \subset \QQ$  
such that for every one-parameter subgroup $\rho$ of $G$ with weight system $(w_1,\dots,w_n)$ and every $R$-point with the equation $F(x_1,\dots,x_n)$:
\begin{equation}\label{E:HMK-L}
\mu_{\rho}^{\cL}(F) = a_0(\cL) \mult_{\rho}(F) - \sum_{j=1}^s a_j(\cL) W_{v_j}(\rho) %=a_0(\cL) \mult_{\rho}(F) - \sum_{i=1}^s a_i(\cL) \sum_{j=1}^{r_i} w_{ij}
\end{equation}

{\rm (2)} Moreover, if $\cL$ has a nonzero $G$-invariant section, then $\{a_i(\cL)\}_{i=0}^{s}$ must satisfy
\begin{equation}
da_0(\cL)= \sum_{j=1}^s r_j v_j a_j(\cL) .
\end{equation}

{\rm (3)} Conversely, for every sequence of rational numbers $\{a_j\}_{j=0}^{s}$ there exists a unique line bundle $\cL(a_0,\dots,a_{s})$ whose HMK-index with respect to every 
one-parameter subgroup $\rho$ as above is 
\begin{equation}\label{E:HMK-L2}
\mu_{\rho}^{\cL(a_0,\dots,a_{s})}(F) = a_0 \mult_{\rho}(F) - \sum_{j=1}^s a_j W_{v_j}(\rho) .
\end{equation}
\end{corollary}
\begin{proof}  For a fixed $m$, $w_{\rho}(S_m)$ is a $\ZZ$-linear combination of $s$ linear functionals
$W_{v_j}(\rho)$, where $j=1,\dots, s$.  It follows from Proposition \ref{HMK-tautological}(A) that each $\mu_{\rho}^{T_m}(F)$ and $\mu_{\rho}^{J_m}(F)$ is a
$\ZZ$-linear combination of $\mult_{\rho}(F)$, and $W_{v_j}(\rho)$'s, $j=1,\dots, s$. Moreover, 
the vectors $\{T_{v_j}\}_{j=1}^{s}$ and $J_{d}$ , form an upper-triangular matrix in the basis $W_{v_1}(\rho),\dots, W_{v_s}(\rho), \mult_{\rho}(F)$,
and so span the full $\QQ$-vector space of $\QQ$-linear combinations of $W_{v_1}(\rho),\dots, W_{v_s}(\rho), \mult_{\rho}(F)$.
Since the dimension of this space is $s+1$, the rank of $\Pic^{G}(M)$, (1) and (3) follow. 

For (2),
recall from \S\ref{irrelevant} that a $G$-linearized line bundle with an invariant nonzero global section must satisfy 
\[
0=\mu_{\rho_{irr}}^{\cL}(F) = a_0(\cL) d - \sum_{i=1}^s r_i c_i a_i(\cL) \quad \text{(for every $F\in S_d$)}.
\]
\end{proof}

\begin{remark}\label{R:index} Given a $G$-linearized line bundle $\cL$ with an invariant nonzero global section,
it will be convenient to write $\cL=[a_1(\cL), \dots, a_s(\cL)]$, where the numbers $a_i(\cL)$ are uniquely determined by Corollary \ref{C:HMK}.
\end{remark}

\subsubsection{Equations of $G$-invariant divisors on $M=\Hyp_{\PP(c_1,\dots,c_n)}(d)$}
In some coordinates $x_1,\dots, x_n$, consider the universal weighted homogeneous polynomial of degree $d$ in $\PP(c_1,\dots,c_n)$:
\begin{equation}\label{E:universal}
U=\sum_{\mon(S_d)} b_{i_1,\dots, i_n} x_1^{i_1}\cdots x_n^{i_n}, 
\end{equation}
where $b_{i_1,\dots, i_n}$ are coordinates on $\HH^0(\PP(c_1,\dots, c_n), \cO(d))$.
A $G$-invariant effective divisor $\mathfrak{D}$ on $M=\Hyp_{\PP(c_1,\dots,c_n)}(d)$ is defined by a $G$-semi-invariant homogeneous form
in the variables $b_{i_1,\dots, i_n}$ (with coefficients in $R$), which we denote by the same letter $\mathfrak{D}$.
In particular, this form is semi-invariant with respect to every one-parameter subgroup $\rho$ of $G$. For $$\rho_\ell:=(\underbrace{0,\dots, 0}_{\ell-1}, 1,0,\dots, 0)$$
acting with weight $1$ on $x_{\ell}$ and weight $0$ on all other variables, we have that $\rho_{\ell}$ acts on $b_{i_1,\dots, i_n}$ 
by
\begin{equation}\label{rho-action}
\rho_\ell(b_{i_1,\dots, i_n}) = t^{i_{\ell}} b_{i_1,\dots, i_n}.
\end{equation}
Then the $\rho_\ell$-degree of $\mathfrak{D}$ is precisely the integer $a_\ell(\cO(\mathfrak{D}))$ given by Corollary \ref{C:HMK}. 
The homogeneous degree of $\mathfrak{D}$ is 
$$
\deg \mathfrak{D}=\frac{1}{d}\sum_{\ell=1}^n c_{\ell} a_\ell(\cO(\mathfrak{D}))=a_0(\cO(\mathfrak{D})).
$$ 

Given now an $R$-point $F$ of $M$, such that $F(K) \notin \Supp(\mathfrak{D})$, we have by Corollary \ref{C:HMK}
that for every one-parameter subgroup $\rho=(w_1,\dots, w_n)$, 
\begin{equation}
\mu^{\cO(\mathfrak{D})}_{\rho}(F)=\deg \mathfrak{D} \mult_{\rho}(F)-\sum_{\ell=1}^n a_\ell(\cO(\mathfrak{D})) w_\ell.
\end{equation}
Note that this is in complete analogy with the homogeneous case described by \eqref{E:homogeneous-HMK}.

\subsection{Boundary divisors}\label{S:boundary} 

We proceed to define geometrically meaningful boundary divisors on $\Hyp(d)$, with the view towards understanding their HMK indices. 
Throughout, we work under the following assumptions:
\begin{itemize}
\item $\ON{char}(K) \nmid d$. %This allows the use of Euler's formula. 
\item
$
\text{$\lcm(c_1,\dots, c_n) \mid d$ and $d>c_i$ (equivalently, $d\geq 2c_i$) for all $i$.}
$ This ensures that the degree $d$ hypersurface is Cartier, and that the singular points of all such hypersurfaces form a dense subset of $\PP(c_1,\dots,c_n)$.
\end{itemize}

%\subsection{Boundary divisors}\label{S:boundary} 

\subsubsection{The discriminant divisor}\label{S:disc}  For some system of coordinates $x_1, \dots, x_n$, let $T\subset \PP(c_1,\dots,c_n)$ be the torus defined by $x_1x_2\cdots x_n\neq 0$.  
Let $\delta=\delta_{\PP(c_1,\dots,c_n)} \subset M=\Hyp(d)$ be the discriminant divisor defined as the closure of the locus parameterizing %weighted degree $d$ 
hypersurfaces that are singular at a point of $T$.  %(If $n=1$, by singular we mean the zero polynomial.). ----> but \Hyp(d) is a projectivization so how to rephrase this
The weighted projective space being a toric variety and $\cO(d)$ a very ample line bundle, this is an example of a more general construction of an A-discriminant
from \cite[Chapter 8]{GKZ}.

The divisor $\delta$ is defined by an irreducible $G$-semi-invariant polynomial, homogeneous in the variables $b_{i_1,\dots, i_n}$,
and semi-invariant with respect to the $\rho_{\ell}$-action for every $\ell=1,\dots, n$.  

By the assumption $\ON{char}(K) \nmid d$, we can define  the A-discriminant as the A-resultant of the partials $\frac{\partial U}{\partial x_1},\dots, \frac{\partial U}{\partial x_n}$ of the universal degree $d$ polynomial $U$
from \eqref{E:universal}.
\begin{remark}\label{remark-point}
If $n=1$, we bypass the geometric definition and simply take the discriminant of $U=b_{d/c_1}x_1^{d/c_1}$ to be given by the equation $b_{d/c_1}=0$.
\end{remark}

\subsubsection{The boundary divisors from singularities of $\PP(c_1,\dots,c_n)$}
For an integer $m \geq 1$, let $J_m\subset S$ be the ideal generated by the variables whose weight is not divisible by $m$. We define $V_m \subset \PP(c_1,\dots, c_n)$ to be 
closed subscheme defined by $J_m$. If $m_1,\dots, m_k$ is the subsequence of $c_1,\dots, c_n$ of the multiples of $m$, then $V_m \simeq \PP(m_1,\dots, m_k)$. 
We define $\delta[m]$ to be the closure of the locus parameterizing weighted degree $d$ hypersurfaces whose restriction to $V_m$ has a vanishing A-discriminant 
as a hypersurface in $\PP(m_1,\dots, m_k)$. Clearly, $\delta[1]=\delta_{\PP(c_1,\dots,c_n)}$ is just the discriminant divisor from \S\ref{S:disc}.
In affine terms, we have a surjective linear map 
\[
\pi_m\colon \HH^0(\PP(c_1,\dots,c_n), \cO(d)) \to \HH^0(V_m, \cO_{V_m}(d))=\HH^0(\PP(m_1,\dots, m_k), \cO(d)).
\]
Then 
\begin{equation}\label{E:pullback-boundary}
\delta[m]=\pi_m^*(\delta_{\PP(m_1,\dots,m_k)}).
\end{equation}
Algebraically, $\delta[m]$ is the discriminant of the weighted homogeneous polynomial obtained from $U$ by setting $x_i=0$ for every $i$ such that $m\nmid c_i$. 
It is also homogeneous in the variables $b_{i_1,\dots, i_n}$ and semi-invariant with respect to the $\rho_{\ell}$-action for every $\ell=1,\dots, n$.   

\begin{remark}
In our applications, all the weights $c_i$'s will be coprime.  In this case, we have that for $m\geq 2$, $V_m$ is a point if $m=c_i$, for some $i$, and empty otherwise.
Then $\delta[m]$ is defined simply as the vanishing locus of the coefficient of $x_i^{\frac{d}{c_i}}$ in $U$; cf. Remark \ref{remark-point} above.
\end{remark}

\subsubsection{A detour on resultants}
To understand the numbers $a_\ell(\cO(\delta[m]))$ that are necessary for the computation of the HMK-indices with respect to $\cO(\delta[m])$, we need to recall some background 
on A-resultants, as it is applicable in our case.

Consider a slightly more general situation of the parameter space of $(d_1,\dots, d_n)$-intersections in $\PP(c_1,\dots,c_n)$. %, where we assume that $c_i \mid d_i$ for every $i=1,\dots, n$.  
It parameterizes $n$-tuples of weighted homogeneous polynomials $$g_i=\sum_{m \in \mon(S_{d_i})} q_{i m} \, m, \ (i=1,\dots, n)$$ in $S$ of weighted degrees $d_1, \dots, d_n$, respectively. It is an affine space with coordinates $\{q_{i m}: i=1,\dots, n, m\in \mon(S_{d_i})\}$. 

Let  $L$ be the algebraic closure of the field $K(q_{im})$. We have the following result:
\begin{prop}\label{P:resultant-degree} Assume $c_1=1$. The A-resultant of $g_1,\dots,g_n$ is a homogeneous irreducible polynomial in the variables $\{q_{i m}: i=1,\dots, n, m\in \mon(S_{d_i})\}$. For every $\ell$, 
it is separately homogenenous in the variables $\{q_{\ell m}: m\in \mon(S_{d_\ell})\}$
of degree $A_{\ell}(d_1,\dots,d_n)$, where $A_{\ell}(d_1,\dots,d_n)$ equals to: 
\begin{enumerate}
\item The length of the finite $L$-algebra 
\[
\cA_\ell:=L[x_2, x_2^{-1}, \dots, x_n, x_n^{-1}]/(g_j(1,x_2, \dots, x_n): j\neq \ell).
\]
\item The length of a subscheme of the torus $T$ given by the ideal $(g_j : j\neq \ell)$.  

\item $\frac{1}{c_1\cdots c_n}$ times the length of the subscheme of $\PP_L^{n-1} \setminus \VV(x_1\cdots x_n)$ given by the ideal $$(g_j(x_1^{c_1}, x_2^{c_2}, \dots,  x_n^{c_n}) : j\neq \ell).$$
\end{enumerate}
\end{prop}
\begin{remark} We can apply the above proposition to compute the degrees of homogeneity of the resultant also in the case $c_1=\gcd(c_1,\dots, c_n)$ by passing to 
$\PP(1, c_2/c_1, \dots, c_n/c_1)$. % replacing $c_i$ with $c_i/c_1$ and $d_i/c_1$.  
\end{remark}

\begin{proof} We dehomogenize by setting $x_1=1$.  We can now use the definition of the A-resultant of the $n$-tuple $\{g_i(1,x_2,\dots,x_n)\}_{i=1}^n$ as in \cite[Proposition-Definition 1.1, p.252]{GKZ}.
Specifically, using $c_1=1$, the affine lattices associated to the Newton polytopes of 
$g_i(1,x_2,\dots,x_n)$, $i=1,\dots, n$, naturally embed into and generate $\ZZ^{n-1}$ via $x_1^{i_1}\cdots x_n^{i_n} \mapsto (i_2, \dots, i_{n})$.
The statement about the degree of homogeneity of the resultant in the coefficients of $g_\ell$ now follows by \cite[Theorem 1.1]{pedersen-sturmfels} and the subsequent discussion on \cite[p.379]{pedersen-sturmfels};
see also \cite[p.255]{GKZ}.
%. Then the semigroup algebra
%in their Equation (1.6) is simply $L[x_2, x_2^{-1}, \dots, x_n, x_n^{-1}]$,. This shows that $A_{c_i}(d_1,\dots,d_n)$ is equal to the length of 
%$L[x_2, x_2^{-1}, \dots, x_n, x_n^{-1}]/(g_j(1,x_2, \dots, x_n): j\neq i)$.  By homogenizing with respect to $x_1$ and passing to the $\ZZ_{c_1}\times \cdots \times \ZZ_{c_r}$ cover of $\PP(c_1,\dots, c_n)$ given
%by the ordinary projective space $\PP^{n-1}$, we conclude.  
\end{proof}

\begin{corollary}\label{P:characters-boundary}
 (1) Suppose $c_1=1$. The weight of the semi-invariant $\delta[1]=\delta_{\PP(c_1,\dots,c_n)}$ with respect to the $\rho_{\ell}$-action is 
\[
\frac{d}{c_{\ell}} A_{\ell}(d-c_1,\dots, d-c_n),
\] 
where $A_{\ell}(d-c_1,\dots, d-c_n)$ is as defined in Proposition \ref{P:resultant-degree}.

(2) For a positive integer $m$, let $m_1,\dots, m_k$ be the subsequence of $c_1,\dots, c_n$ of the multiples of $m$, i.e., $\{m_1,\dots, m_k\}=\{c_i \colon m\mid c_i\}$. Assume $m_1=m$. Then
the degree of the semi-invariant $\delta[m]$ with respect to the $\rho_{\ell}$ is 
\[
\frac{d}{c_{\ell}} A_{\ell}((d-m_1)/m,\dots, (d-m_k)/m), \quad \text{if $m\mid c_{\ell}$,}
\] 
and is zero otherwise. 
\end{corollary}
\begin{proof} By \eqref{E:pullback-boundary}, (1) implies (2). 

To prove (1), we use a standard specialization trick:
Consider the Fermat-like polynomial $$F=\sum_{i=1}^n u_i x_i^{\frac{d}{c_i}}.$$ By the assumption $\ON{char}(K)\nmid d$, and so the discriminant of $F$ is the resultant of 
$g_i=u_i x_i^{\frac{d}{c_i}-1}$, $i=1,\dots, n$. Since this resultant does not vanish when $u_1\cdots u_n\neq 0$, it must be equal (up to a scalar) to the monomial $\prod_{i=1}^n u_i^{A_{i}(d-c_1,\dots,d-c_n)}$,
where we used the degree of homogeneity with respect to $u_i$ established in Proposition \ref{P:resultant-degree}. Since $u_i$ has $\rho_\ell$-weight $\frac{d}{c_\ell}$ if $i=\ell$, and zero otherwise, we conclude.
\end{proof}

The following lemma gives a useful estimate on $A_{i}(d_1,\dots,d_n)$ under some further simplifying assumptions:
\begin{lemma}\label{L:bound} Suppose $c_{1}=1$ and $c_i \mid d_i$ for all $i=1,\dots, n$. Then we have
\[
A_{1}(d_1,\dots, d_n)=\prod_{j\neq 1} \frac{d_j}{c_j}.
\]
If $c_{\ell}>1$, we have a bound
\[
A_{\ell}(d_1,\dots, d_n)\leq \frac{d_1}{c_\ell} \prod_{j\neq 1,\ell}^n \frac{d_j}{c_j}.
\]
\end{lemma}

\begin{proof}
The polynomial $g_j$ defines a degree $d_j$ hypersurface in $\PP(c_1,\dots,c_n)$. Applying interpretation (2) of Proposition \ref{P:resultant-degree}, Bezout's theorem gives a bound 
\[
A_{\ell}(d_1,\dots,d_n)\leq \frac{1}{c_1\cdots c_n}\prod_{j\neq 
\ell}^n d_j,
\]
with equality exactly when we can find $g_1,\dots, \hat{g_\ell}, \cdots, g_n$ that have no common root in $\VV(x_1\dots x_n)$. If $c_1=1$, then 
using $c_i \mid d_i$, the specialization $g_i=x_i^{\frac{d_i}{c_i}}-x_1^{d_i}$ shows that this is the case for $\ell=1$.

If $c_{\ell}>1$, then $g_1,\dots, \hat{g_\ell}, \dots, g_n$ necessarily vanish at the point $x_1=\cdots =\hat{x_{\ell}}=\cdots =x_n=0$ whenever $c_{\ell}$ does not divide any of the $d_i, i\neq \ell$. 
We have a strict inequality in this case.
\end{proof}

\subsubsection{Example: The resultant of a $(5,5,4,3)$-intersection in $\PP(1,1,2,3)$}  

\begin{lemma} The degrees of homogeneity of the resultant of degree $5, 5, 4, 3$ forms $g_1,g_2,g_3,g_4$ in $\PP(1,1,2,3)$ are:
\begin{enumerate}
\item $A_1(5,5,4,3)=A_2(5,5,4,3)=10,$ 
\item $A_3(5,5,4,3)=12,$
\item $A_4(5,5,4,3)=16$.
\end{enumerate}
\end{lemma}
\begin{proof}
(1) follows directly from Lemma \ref{L:bound}.  For (2), we note that $g_4(1,x_2,x_3,x_4)$ is linear in the weight $3$ variables $x_4$, and so 
$A_2(5,5,4,3)$ is the number of intersections of two generic degree $5$ weighted hypersurfaces in $\PP(1,1,2)$ away from the vertex, which is easily seen to be $12=(25-1)/2$. 
For (3), consider the intersection of three degrees $5, 5, 4$ homogeneous forms $g_1(x_1, x_2, x_3^2, x_4^3), g_2(x_1, x_2, x_3^2, x_4^3), g_3(x_1, x_2, x_3^2, x_4^3)$ in $\PP^3$.
If $g_1,g_2,g_3$ are chosen generically, these hypersurfaces have no common roots in $x_1x_2x_3=0$ with the exception of the point $[0:0:0:1]$, where they all vanish
with multiplicities, $2, 2, 1$, respectively. Their tangent cones at this point have no common zero, and so passing to the blow-up of $\PP^3$ at $[0:0:0:1]$, we conclude that
\[
A_4(5,5,4,3)=\frac{1}{6} (100-4)=16.
\]
\end{proof}
\begin{corollary}\label{L:discriminant-sextic} Let $\delta \subset \Hyp_{\PP(1,1,2,3)}(6)$ be the discriminant divisor.  Then for an $R$-point $F\in \Hyp_{\PP(1,1,2,3)}(6)$ with $F(K)\notin \Supp(\delta[1])$, we have
\[
\mu_{\rho}^{\delta[1]}(F) =48\mult_{\rho}(F)-60(w_1+w_2)-36w_3-32w_4.
\]
\end{corollary}

\subsubsection{The cone generated by the boundary divisors}
It would be interesting to understand the cone of effective $\QQ$-linear combinations of the boundary divisors $\delta[m]$ for an arbitrary $(c_1,\dots,c_n)$,
but for our purposes it will be sufficient to understand the simplest case when 
\[
1=c_1=c_2=\cdots=c_{r-1}< c_{r}<c_{r+2}<\cdots <c_n,
\]
where $c_i$'s are all pairwise coprime. 
In this case, in the notation of Remark \ref{R:index}, we have:
\begin{equation}\label{E:HMK-delta0}
%\cO(\delta)=
\cO(\delta[1])= [\frac{d}{1}A_{1}(d-c_1,\dots,d-c_n), \frac{d}{c_r}A_{r}(d-c_1,\dots,d-c_n), \dots, \frac{d}{c_n}A_{n}(d-c_1,\dots,d-c_n)]
\end{equation}
%For $\ell=1,\dots, s$, we have
\begin{equation}\label{E:HMK-deltam}
\cO(\delta[c_i]) = [\underbrace{0,\dots, 0}_{i-r+1}, \frac{d}{c_i}, 0, \dots, 0], \quad i\geq r.  
\end{equation}
\begin{prop}[The balanced line bundle]\label{T:space-stability} The cone of effective linear combinations of the boundary divisors $\cO(\delta[m])$ is simplicial with extremal rays 
$\cO(\delta[c_i])$, $i=1,\dots, s$. Moreover, this cone contains a line bundle $$\cL^{bal}=\left[\frac{d}{\sum c_i}, \dots, \frac{d}{\sum c_i}\right].$$

% with a $G$-invariant section $\fD^{bal}$
%supported on the boundary such that 
%\[
%\mu_{\rho}^{\cL^{bal}}(f) = \mult_{\rho}(F)- \frac{d}{\sum_{i=1}^{n} c_i}\left(\sum_{i=1}^n w_i\right),
%\]
%for every $R$-point $f\colon \Delta\to M$ given by an equation $F$.
\end{prop}
\begin{proof}
By Lemma \ref{L:bound}, for $\ell\geq r$, we have (in)equalities 
\[
\frac{d}{1}A_{1}(d-c_1,\dots,d-c_n) =d \prod_{j=2}^n \frac{d-c_j}{c_j} >
\frac{d-1}{c_\ell} \frac{d}{c_\ell}  \prod_{j\neq 1,\ell}^n \frac{d_j-c_j}{c_j}
 > \frac{d}{c_\ell} A_{\ell}(d-c_1,\dots,d-c_n),
\] 
where we have used $d-c_{\ell}>(d-1)/c_{\ell}$.

\end{proof}
 
%%% The point is that d (d-v_j)/v_j > d (d-1)/v^2_j or d-v_j>(d-1)/v_j just like we did before. 

\subsubsection{Weighted hypersurfaces with $\GG_m$-action and their HMK-indices} 

Suppose the $c_1,\dots,c_n$ is a sequence of weights such as $c_1=1$ and all the $c_i$'s greater than $1$ are coprime. (This ensures
that $\PP(c_1,\dots, c_n)$ has isolated singularities.) We continue working under the assumption $\ON{char}(K) \nmid d$.

\begin{definition}[(Partial) Fermat hypersurfaces]
\label{D:partial-fermat}
Let $S_R=R[x_1,\dots, x_n]$ be the graded 
ring with grading given by the weight vector $\vec{c}=(c_1,\dots, c_n)$. 
Define a Fermat hypersurface of weighted degree $d$ to be a smooth hypersurface given by the equation 
\[
x_1^{\frac{d}{c_1}}+\cdots+x_n^{\frac{d}{c_n}}.
\]
%It is a quasi-smooth hypersurface. 
For $\ell\in \{1,\dots, n\}$, we define the $\ell^{th}$ partial Fermat hypersurface of weighted degree $d$ 
to be %the weighted degree $d$ hypersurface 
\begin{equation}\label{E:partial-fermat}
f_\ell:=x_1^{\frac{d}{c_1}}+\cdots+x_{\ell-1}^{\frac{d}{c_{\ell-1}}}+x_{\ell+1}^{\frac{d}{c_{\ell+1}}}
+\cdots +x_n^{\frac{d}{c_n}}.
\end{equation}
\end{definition}

\begin{comment}
Let 
$$p_\ell:=[\underbrace{0:\cdots :0}_{\ell-1}: 1: 0: \cdots:0] \in \PP(c_1,\dots, c_n).$$
%which is a `Fermat with the $\ell^{th}$ term omitted'.
Clearly, $p_{\ell}$ is the only point where all the partials of $f_{\ell}$ vanish. It follows that $f_{\ell}\in \Supp(\delta)$.
Furthermore, for every $m$, the restriction of $f_{\ell}$ to $V_m\simeq \PP(c_j \colon m\mid c_j)$ is:
\begin{itemize}
\item a quasi-smooth Fermat hypersurface of degree $d$ in $ \PP(c_j \colon m\mid c_j)$ if $m\nmid c_\ell$,
\item a $c_{\ell}^{th}$-partial Fermat of degree $d$ in $\PP(c_j \colon m\mid c_j)$ if $m\mid c_\ell$.
\end{itemize}
It follows that $f_\ell$ lies in $\Supp(\delta[m])$ if and only if $m\mid c_\ell$.
\end{comment}

Let $\rho_\ell$ be the one-parameter subgroup of $G$ given by 
$\rho_{\ell}(x_\ell) = t x_\ell$, and 
$\rho_{\ell}(x_i) = x_i$ for $i\neq \ell$.  Evidently,
$f_\ell$ is $\rho_{\ell}$-invariant.

\begin{prop}\label{P:characters-boundary-2}  %divide $\frac{d}{c_i}-1$ for all $i=1,\dots, n$. %[Stated over $\CC$ for now]
Consider the $\ell^{th}$ partial Fermat $f_{\ell}$ of degree $d$, with its $\GG_m$-action 
$\rho_{\ell}$.
%For a positive integer $m$, let $m_1,\dots, m_k$ be the subsequence of $c_1,\dots, c_n$ of the multiples of $m$, i.e., $\{m_1,\dots, m_k\}=\{c_i \colon m\mid c_i\}$.
% the characters of the line bundles $\cO(\delta_j)$ at $[f_{\ell}]\in M$ (see \S\ref{S:character}) are as follows:
\begin{align}
\mu_{\rho_{\ell}}^{\cO(\delta[1])}(f_{\ell})&= -\frac{d}{c_{\ell}} A_{\ell}(d-c_1,\dots, d-c_n),  \label{E:chi-delta-0}   \\
\mu^{\cO(\delta[m])}_{\rho_{\ell}}(f_{\ell}) &=\begin{cases} -\dfrac{d}{c_{\ell}}, & \text{if $m=c_{\ell}$}, \\
0, & \text{if $m\geq 2$ and $m\neq c_{\ell}$}. \end{cases}  \label{E:chi-delta-m}
\end{align}
\end{prop}
\begin{proof} 
This follows from Corollary \ref{P:characters-boundary}.
\end{proof}

\section{Examples}\label{examples} In this section, we give explicit examples of the theory developed so far. Three of the examples will treat stability of weighted hypersurfaces: 
\begin{enumerate}
\item Sextics in $\PP(1,2,3)$ and Tate's minimal models of elliptic fibrations from the point of view of Koll\'ar stability; see \S\ref{Tate}.
\item Quartics in $\PP(1,1,1,2)$ and interpretation of \cite[Section 4]{corti-annals} in terms of Koll\'ar stability; see \S\ref{quartics}.
\item Sextics in $\PP(1,1,2,3)$ and first steps towards establishing the existence of standard models of degree one del Pezzo fibrations.
\end{enumerate}
We also consider two example of $(d,e)$-intersections in weighted projective spaces:
\begin{enumerate}
\item Koll\'ar stability of $(4,2)$-intersections in $\PP(1,1,1,2,2)$ and application to the standard models of degree $2$ del Pezzo fibrations; see \S\ref{dP2-stab}.
\item Koll\'ar stability of $(6,3)$-intersections in $\PP(1,1,2,3,3)$.
\end{enumerate}
In the last example, we introduce the concept of T-stability of a $(6,3)$-intersection in $\PP(1,1,2,3,3)$ which will be the crucial tool in Section \ref{Sing-and-stab}.

Throughout $R$ is a DVR with a fraction field $K$ and valuation $\val\colon K \to \ZZ\cup \infty$. 

\subsection{Sextics in $\PP(1,2,3)$ and Tate's minimal Weierstrass models of elliptic curves}\label{Tate}
Every elliptic curve over $K$ has a model over $R$ given by a sextic equation in $\PP_R(1_x,2_z,3_w)$:
\begin{equation}\label{E:weierstrass}
F(x,z,w)=o w^2 + (p_1xz+p_2 x^3) w+ q_1 z^3+q_2 x^2z^2+q_3 x^4z+q_4 x^6=0.
\end{equation}
Moreover, one can always arrange for %\eqref{E:weierstrass} to further satisfy
\begin{equation}\label{tate-condition}
o=q_1=1,
\end{equation}
in which case \eqref{E:weierstrass} is known as the Weierstrass form. 
%%% In the classical sense, the Weierstrass form has \val(o(t))=\val(q_1(t))=0$

In \cite{tate-algorithm}, Tate famously defined the \emph{minimal Weierstrass model} to be the model over $R$ satisfying \eqref{tate-condition} and 
minimizing the valuation of the discriminant of $F(x,z,w)$. Tate used the minimal model in his algorithm for determining the Kodaira-N\'eron classification of the central fiber in the associated elliptic fibration.  
Here, we demonstrate that Tate's minimal model is an example of Koll\'ar $\fD$-semistable model for an appropriate choice of $\fD$ on the space of sextics in $\PP(1,2,3)$.

By \S\ref{S:boundary}, the three boundary divisors on $\Hyp_{\PP(1,2,3)}(6)$ are: 
\begin{enumerate}
\item $\delta=\delta[1]$,  the closure of the locus of sextics singular at a non-singular point of $\PP(1,2,3)$. 
\item $\delta[2]$, sextics through $\frac{1}{2}(1,1)$-point.
\item $\delta[3]$, sextics through $\frac{1}{3}(1,2)$-point.
\end{enumerate}
By Corollary \ref{P:characters-boundary}, with respect to a one-parameter subgroup $\rho=(w_1,w_2,w_3)$, acting diagonally in some coordinates $x, z, w$, 
the HMK-indices of these boundary divisors are:  
\begin{align*}
\mu_{\rho}^{\delta[1]}(F) &=7\mult_{\rho}(F)-12w_1-6w_2-6w_3,   \quad (\ON{char}(K)\neq 2,3) \\ %\label{elliptic-delta0}
\mu_{\rho}^{\delta[2]}(F) &=\mult_{\rho}(F)-3w_2.\\
\mu_{\rho}^{\delta[3]}(F) &=\mult_{\rho}(F)-2w_3.
\end{align*}

%\begin{comment}
\begin{remark} Explicitly, for the reader who might be surprised to learn that the degree of the discriminant of an elliptic curve can be $7$ (just as we were),
the equation of $\delta[1]$ is 
\begin{multline} p_1^{3} p_{2}^{3}q_1+27 o p_{2}^{4}q_{1}^{2}-p_{1}^{4}p_{2}^{2}q_{2}-36op_{1}p_{2}^{3}q_{1}q_{2}+8o
      p_{1}^{2}p_{2}^{2}q_{2}^{2}-16o^{2}p_{2}^{2}q_{2}^{3}+p_{1}^{5}p_{2}q_{3} \\ +30op_{1}^{2}p_{2}^{2}q_{1}q_{3} -8op_{1
      }^{3}p_{2}q_{2}q_{3}+72o^{2}p_{2}^{2}q_{1}q_{2}q_{3}+16o^{2}p_{1}p_{2}q_{2}^{2}q_{3}-op_{1}^{4}q_{3}^{
      2}-96o^{2}p_{1}p_{2}q_{1}q_{3}^{2} \\+8o^{2}p_{1}^{2}q_{2}q_{3}^{2}-16o^{3}q_{2}^{2}q_{3}^{2}+64o^{3}q_{1}q_{3}^{3}
      -p_{1}^{6}q_{4}-36op_{1}^{3}p_{2}q_{1}q_{4}-216o^{2}p_{2}^{2}q_{1}^{2}q_{4}+12op_{1}^{4}q_{2}q_{4} \\ +144o^{2}
      p_{1}p_{2}q_{1}q_{2}q_{4}-48o^{2}p_{1}^{2}q_{2}^{2}q_{4}+64o^{3}q_{2}^{3}q_{4}+72o^{2}p_{1}^{2}q_{1}q_{3
      }q_{4}-288o^{3}q_{1}q_{2}q_{3}q_{4}+432o^{3}q_{1}^{2}q_{4}^{2}.
      \end{multline}
\end{remark}
%\end{comment}
%Consider an effective divisor $\fD=\beta_0 \delta_0+\beta_1\delta_1+\beta_2\delta_2$. 

%%% SCRATCH WORK BEGIN
\begin{comment}
Socle $\Sigma(6;1,2,3)=(6-2)+(6-4)+(6-6)+1=4+2+1=7$.

Then $M_1(6;1,2,3)=\#\{x_1^5 x_2, x_1^7\}=2$,
$M_2(6;1,2,3)=\#\{x_2^2 x_1^3, x_2^3 x_1\}=2$,
$M_3(6;1,2,3)=\#\{x_3 x_1^4, x_3 x_2x_1^2, x_3^2 x_1\}=3$,

Define $M_{c_\ell}(d; c_1,\dots, c_n)$ to be 
the number of degree $\Sigma(d;c_1,\dots,c_n)$ monomials in $S=K[x_1,\dots,x_n]$
that are divisible by $x_\ell^{\frac{d}{c_\ell}-1}$
but are not divisible by any $x_j^{\frac{d}{c_j}-1}$ with $j \neq \ell$.  
\end{comment}
%%% SCRATCH WORK END

\begin{prop} Let $\fD^{Gor}:=\epsilon \delta[1]+\delta[2]+\delta[3]$, where $0< \epsilon \ll 1$.
Then $\fD^{Gor}$-semistable model is the Tate's minimal Weierstrass model.
\end{prop}
\begin{proof} Consider $F$ as in \eqref{E:weierstrass}.
If $\val(o)\geq 1$, then for $\rho=(1,1,1)$, % \cdot (x,z,w)=(t x, tz,  tw)$,  
we have $\mult_{\rho}(F)\geq 3$. 
Then 
\[
\mu_{\rho}^{\delta[1]}(F) \geq -3, \quad \mu_{\rho}^{\delta[2]}(F)\geq 0, \quad \mu_{\rho}^{\delta[3]}(F) \geq 1,
\]
%\begin{align*}
%\mu_{\rho}^{\delta[1]}(F) &=7\mult_{\rho}(F)-12w_1-6w_2-6w_3 \geq 21-24=-3, 
%\mu_{\rho}^{\delta[2]}(F) &=\mult_{\rho}(F)-3w_2 \geq 0,\\
%\mu_{\rho}^{\delta[3]}(F) &=\mult_{\rho}(F)-2w_3 \geq 1,
%\end{align*}
and, consequently, $\mu_{\rho}^{\fD^{Gor}}(F) > - 3\epsilon+1>0$.

If $\val(q_1)\geq 1$, then for $\rho=(1,1,2)$, we have $\mult_{\rho}(F)\geq 4$ and
\begin{align*}
\mu_{\rho}^{\delta[1]}(F) &=7\mult_{\rho}(F)-12w_1-6w_2-6w_3 \geq 28-30=-2, \\
\mu_{\rho}^{\delta[2]}(F) &=\mult_{\rho}(F)-3w_2 \geq 1,\\
\mu_{\rho}^{\delta[3]}(F) &=\mult_{\rho}(F)-2w_3 \geq 0,
\end{align*}
and, consequently, $\mu_{\rho}^{\fD^{Gor}}(F) >-2\epsilon+1>0$.

We conclude that a $\fD^{Gor}$-semistable model has $\val(o)=\val(q_1)=0$ and so avoids $\delta[2]$ and $\delta[3]$ entirely. Then %$\fD^{Gor}$-minimizing property
%implies that a 
a $\fD^{Gor}$-semistable model minimizes the degree of the discriminant $\delta[1]$ subject to the condition $\val(o)=\val(q_1)=0$, which is precisely the definition 
of the Tate's model.

\end{proof}
\begin{remark}  The formula for $\mu_{\rho}^{\delta[1]}(F)$ is derived under the assumption $\ON{char}(K)\neq 2,3$, in which case it suffices to take $\epsilon<1/6$ in the statement
of the proposition.  Note however that when $0< \epsilon \ll 1$, the exact formula for $\mu_{\rho}^{\delta[1]}$ is irrelevant, and so the proposition remains true even in characteristics $2$ and $3$.
\end{remark}

%\section{Stability of low degree del Pezzos}\label{S:del-pezzo}
%In this section, we work over an algebraically closed field $k$ and $R$ is a fixed DVR with residue field $k$, fraction field $K$, and uniformizer $t$. 
%Some restrictions on the characteristic of $k$ will be imposed
%at appropriate places. 
\subsection{Quartics in $\PP(1,1,1,2)$ and degree $2$ del Pezzo fibrations}\label{quartics}
Assume $\ON{char}(K)\neq 2$.
Consider the space $\Hyp_{\PP(1,1,1,2)}(4)$ of degree $4$ weighted hypersurfaces in $\PP(1,1,1,2)$.  By \S\ref{S:boundary}, we have two boundary 
divisors: 
\begin{enumerate}
\item $\delta[1]$, the closure of the locus of quartics singular at a non-singular point of $\PP(1,1,1,2)$. 
\item $\delta[2]$, quartics through the $\frac{1}{2}(1,1,1)$ point of $\PP(1,1,1,2)$.
\end{enumerate}
The $K$-valued points of $M\setminus (\delta_0\cup \delta_1)$ parameterize smooth degree $2$ 
del Pezzo surfaces over $K$ (this is true for any field $K$). 

By Corollary \ref{P:characters-boundary}, with respect to a one-parameter subgroup $\rho=(w_1,w_2,w_3,w_4)$ acting diagonally in some coordinates $x_1,x_2,x_3,x_4$, 
and an $R$-point of $\Hyp_{\PP(1,1,1,2)}(4)$ with the equation $F \in R[x_1,x_2,x_3,x_4]_4$, the HMK-indices of these boundary divisors are:  
\begin{align*}
\mu_{\rho}^{\delta[1]}(F) &=40\mult_{\rho}(F)-36(w_1+w_2+w_3)-26w_4, \\
\mu_{\rho}^{\delta[2]}(F) &=\mult_{\rho}(F)-2w_4.
\end{align*}

It follows that for
\begin{equation}
\fD^{bal}:=\frac{1}{45}(\delta[1]+5\delta[2]),
\end{equation}
given by Proposition \ref{T:space-stability},
we
have 
\begin{align*}
\mu_{\rho}^{\fD^{bal}}(f) %=\frac{1}{45}(45\mult_{\rho}(F)- 26w_1-36w_2-&36(w_3+w_4)-10w_1)=\\
&=\mult_{\rho}(F)-\frac{4}{5}(w_1+w_2+w_3+w_4).
\end{align*}
We have the following result recasting a bulk of \cite[Section 4]{corti-annals} in the language of Koll\'ar stability (hence we omit the proof): 
\begin{prop}\label{T:degree-2-naive} Suppose $R=\cO_{C,p}$ is the local ring of a smooth curve over an algebraically closed field of characteristic not $2$.
Suppose $F \in R[x_1,x_2,x_3,x_4]_4$ is $\fD^{bal}$-semistable. Then $F=0$ defines a threefold del Pezzo fibration $X\to \spec R$ of degree $2$ such that $X$ has 
terminal singularities and the central fiber $X_0$ is reduced, and either irreducible or a union of two hypersurfaces in $\MP(1,1,1,2)$ of degree $2$ each passing through the $\frac{1}{2}(1,1,1)$-point.
\end{prop}
For a general choice of $A, B\in R[x_1,x_2,x_3,x_4]_2$, the family
\[
t (x_4^2+G(x_1,x_2,x_3)) + A(x_1,x_2,x_3) B(x_1,x_2,x_3)  = 0,
\]
has a central fiber that is a union of two degree $2$ hypersurfaces each passing through the $\frac{1}{2}(1,1,1)$-point and is $\fD$-semistable for any choice of a $G$-invariant effective boundary divisor $\fD$. 
In the next subsection, we illustrate how to improve Proposition \ref{T:degree-2-naive} by working with $(4,2)$-intersections in $\PP(1,1,1,2,2)$.
%%%%%%%%%%%%%%%%%%%%%%%%%%%%%%%%%%%%%%%%%%%%%%%%%%%
%%%%%%%%%%%%%%%%%%%%%%%%%%%%%%%%%%%%%%%%%%%%%%%%%%%
\subsection{ $(4,2)$-intersections in $\PP(1,1,1,2,2)$} \label{dP2-stab}
%%%%%%%%%%%%%%%%%%%%%%%%%%%%%%%%%%%%%%%%%%%%%%%%%%%
%%%%%%%%%%%%%%%%%%%%%%%%%%%%%%%%%%%%%%%%%%%%%%%%%%%
Assume $\ON{char}(K)\neq 2$. Let $M=\Int_{\PP(1,1,1,2,2)}(4,2)$ be the parameter space of $(4, 2)$ intersections in $\PP(1,1,1,2,2)$, and $G=\Aut_{gr}(\Cox(\PP(1,1,1,2,2))$.

\subsubsection{Boundary divisors}\label{Sub:boundary}
Let $M^{0}$ be the locus in $M$ of $(4,2)$-intersection $(F,H)$ where $H=x_5$.
Then the translate $G\cdot M^{0}$  is an open $G$-invariant subscheme of $M$ whose complement has codimension $2$ (it is defined by the vanishing of the coefficients of $x_4$ and $x_5$ in $H$).
As $x_5=0$ defines a closed subscheme in $\PP(1,1,1,2,2)$ isomorphic to $\PP(1,1,1,2)$, we have an isomorphism $M^{0} \simeq \Hyp_{\PP(1,1,1,2)}(4)$ given by $(F,x_5) \mapsto F(x_1,\dots,x_4,0)$.
This isomorphism is $\Aut(\Cox(\PP(1,1,1,2))\times \GG_m$-equivariant, where $\Aut(\Cox(\PP(1,1,1,2))$ acts on $\ZZ[x_1,\dots,x_4]$ and $\GG_m$ acts by $\lambda\cdot x_5=\lambda x_5$. (In particular,
this $\GG_m$ acts trivially on $\Hyp_{\PP(1,1,1,2)}(4)$.)

By smoothness of  $\Hyp_{\PP(1,1,1,2)}(4)$ and $M$, the $G$-translates of the boundary divisors $\delta[1]$ and $\delta[2]$ inside $\Hyp_{\PP(1,1,1,2)}(4)$ that we defined in \S\ref{S:boundary} extend uniquely to divisors on $M$,
which we continue to denote $\delta[1]$ and $\delta[2]$. Moreover, if in some coordinates $x_1,\dots,x_5$, we have $I=(F(x_1,\dots,x_4), x_5)$, then for every one-parameter subgroup $\rho$ with weight system
$(w_1,\dots,w_4,w_5)$ in this coordinates, we have 
\begin{equation}\label{E:truncate}
\mu^{\delta[i]}_{\rho}(F,H)=\mu^{\delta[i]}_{\rho'}(F),
\end{equation}
where $\rho'=(w_1,\dots,w_4)$.

By Proposition \ref{HMK-tautological} (B), we compute the Hilbert-Mumford-Koll\'ar indices with respect to the line bundles $T_1, T_2, J_2, J_4$, and $J_4':=J_4-5J_2-T_2$ 
of an $R$-point $f\colon \Spec(R)\to M$ given by a $(4,2)$-intersection $(F,H)$:
\begin{equation}\label{E:degree2-lambdas}
\begin{aligned}
\mu_{\rho}^{T_1}(f) &=-w_{\rho}(\{x_1,x_2,x_3\})\\ &=-(w_1+w_2+w_3), \\
\mu_{\rho}^{T_2}(f) &=-w_{\rho}(S_2)=w_{\rho}(\{x_1^2,x_1x_2,x_2^2,x_1x_3,x_2x_3,x_3^2,x_4,x_5\})\\ & =-(4(w_1+w_2+w_3)+(w_4+w_5)), \\
\mu_{\rho}^{J_2}(f) &= -\mult_{\rho}(H), \\
\mu_{\rho}^{J_4}(f)  % &= 5 \mult_{\rho}(H_t)+w_{\rho}(P_{2})+\mult_{\rho}(F_t) \\
&\leq -(5 \mult_{\rho}(H_t)+\mult_{\rho}(F_t)+4(w_1+w_2+w_3)+(w_5+w_5)). \\
\mu_{\rho}^{J'_4}(f)  % &= 5 \mult_{\rho}(H_t)+w_{\rho}(P_{2})+\mult_{\rho}(F_t) \\
&\leq -\mult_{\rho}(F).
\end{aligned}
\end{equation}
Clearly, $T_1, T_2, J_2$, and  $J'_4$ form a basis of $\Pic^G(M)\otimes \QQ$. 

We proceed to express the boundary divisors $\delta[1]$ and $\delta[2]$ in terms of the above basis with the goal of computing their HMK-indices. To do so, 
let $\{f_{\ell}\}_{\ell=1,4}$ be the partial Fermats as in Definition \ref{D:partial-fermat}. Then we have the following two
distinguished points of $M$, $I_1:=(f_1, x_5)=(x_2^4+x_3^4+x_4^4, x_5)$ and $I_4:=(f_4, x_5) =(x_1^4+x_2^4+x_3^4, x_5)$.
The point $I_{\ell}$ is fixed by the one-parameter subgroup $\rho_{\ell}$ acting via $\rho\cdot x_{\ell}=tx_{\ell}$, 
by $\rho_{irr}=(1,1,1,2,2)$, and by $\rho_{5}=(0,0,0,0,1)$.   

We now collect the computations that we have obtained so far:

%The following table summarizes the characters of various line bundles (cf. \S\ref{S:character}) computed at the partial Fermats $[f_{\ell}]$ with respect to the one-parameter subgroups 
%fixing $[f_{\ell}]$:
\begin{center}
\renewcommand{\arraystretch}{1.3}
\begin{tabular}{|c|c|c|c|c|c|c|c|}
\hline
$\cL=$
 & $T_1$ & $T_2$ & $J_2$ & $J'_4$ & 
$\delta[1]$ & $\delta[2]$ \\
\hline

$\mu^{\cL}_{\rho_1}(I_1)=$ & $-1$ & $-4$ & $0$ & $0$ & $-36$ & $0$ \\
\hline
$\mu^{\cL}_{\rho_4}(I_4)=$ & $0$ & $-1$ & $0$ & $0$ & $-26$ & $-2$ \\
\hline
$\mu^{\cL}_{\rho_{irr}}(I_{\ell})=$ & $-3$ & $-16$ & $-2$ & $-4$ & $0$ & $0$ \\
\hline
$\mu^{\cL}_{\rho_5}(I_{\ell})=$ & $0$ & $-1$ & $-1$ & $0$ & $0$ & $0$ \\
\hline
\end{tabular}
\end{center}
%\begin{remark}\label{R:HMK}
%Here, the characters of boundary line bundles are computed in Proposition \ref{P:characters-boundary}, and the remaining characters are obtained from
%\eqref{E:degree2-lambdas} using the fact that for any $\rho$-fixed $k$-point $[f]\in M$, we have the equality $\chi_{\cL}([f], \rho)=\mu_{\rho}^{\cL}(f)$, where $f\colon \Delta\to M$ is 
%the constant morphism with image $[f]$.
%\end{remark}
Using the method of indeterminate coefficients, we obtain:
\begin{align*}
\delta[1]&=-68T_1+ 26T_2-26J_2-40J'_4, \\
\delta[2]&= -8 T_1 + 2T_2 - 2J_2-J'_4.
\end{align*}

In particular, from \eqref{E:degree2-lambdas}, we conclude that for $f$ with equations $(F,H)$, we have
\begin{align}\label{E:HMK-boundary-degree-2}
\mu_{\rho}^{\delta[1]}(f) &\geq 40\mult_{\rho}(F_t)+ 26\mult_\rho(H_t) -36(w_1+w_2+w_3)-26(w_4+w_5) ,
\\
\mu_{\rho}^{\delta[2]}(f) &\geq \mult_{\rho}(F_t)+ 2\mult_\rho(H_t) -2(w_4+w_5).
\end{align}

For $\fD^{bal}:=\frac{1}{45}\left(\delta[1]+5\delta[2] \right)$, we then have
\begin{equation}
\mu_{\rho}^{\fD^{bal}}(f) \geq \mult_\rho(F_t)+\frac{4}{5}\mult_\rho(H_t)-\frac{4}{5}(w_1+w_2+w_3+w_4+w_5).
\end{equation}
%as desired.

\subsubsection{Corti's standard models in degree $2$}

Let $\fD^{ter}=\fD^{bal}+\epsilon \delta_1$, where $0<\epsilon \ll 1$. Then we have the following result:
\begin{theorem}\label{T:degree-2} Suppose $R=\cO_{C,p}$ is the local ring of a smooth curve over an algebraically closed field of characteristic not $2$.
Suppose $(F,H) \in M$ is $\fD^{ter}$-semistable. Then $F=H=0$ defines a threefold del Pezzo fibration $X\to \Delta$ of degree $2$ such that $X$ has 
terminal singularities and the central fiber $X_0$ is integral.
\end{theorem}
The proof is similar, but easier, compared to Theorem \ref{mainthm2}, with a part of the proof following directly from \cite[Section~4]{corti-annals}.
Therefore, we omit details.

\subsection{Sextics in $\PP(1,1,2,3)$: Towards standard models of degree $1$ del Pezzo fibrations}\label{3-6}

%Let $S=k[x_1,x_2,x_3, x_4]$, where the grading is given by $w(x_1)=3$,
%$w(x_2)=2$, and $w(x_3)=w(x_4)=1$. Set $G=\Aut(S)$, the group of grading-preserving $k$-algebra automorphisms of $S$.
%Let $\PP:=\proj S=\PP(3, 2, 1, 1)$. 

Let $M:=\Hyp_{\PP(1,1,2,3)}(6)$ be the space of sextics in $\PP(1,1,2,3)$, and $G=\Aut(\Cox(\PP(1,1,2,3)))$.  
We have three boundary divisors in $M$: 
\begin{enumerate}
\item $\delta[1]$, the closure of the locus of quartics singular at a non-singular point of $\PP(1,1,1,2)$. 
\item $\delta[2]$, sextics through the $\frac{1}{2}(1,1,1)$-point. 
\item $\delta[3]$, sextics through the $\frac{1}{3}(1,1,2)$-point. 
\end{enumerate}
The $K$-valued points of $M\setminus (\delta[1]\cup \delta[2] \cup \delta[3])$ parameterize smooth degree $1$ 
del Pezzo surfaces over $K$ (for any field $K$).  

Assume now $\ON{char}(K)\neq 2, 3$. Given %$f\colon \Delta\to M$ with the equation 
$F\in R[x_1,x_2,x_3,x_4]_6$ and a one-parameter subgroup $\rho$ of $G$ acting by $\rho\cdot x_i=t^{w_i}x_i$,
we have by Proposition \ref{P:characters-boundary} and Lemma \ref{L:discriminant-sextic} that

\begin{equation}\label{E:HMK-sextics}
\begin{aligned}
\mu_{\rho}^{\delta[1]}(F) &=48\mult_{\rho}(F)-60(w_1+w_2)-36w_3-32w_4, \\
\mu_{\rho}^{\delta[2]}(F) &=\mult_{\rho}(F)-3w_3, \\
\mu_{\rho}^{\delta[3]}(F) &=\mult_{\rho}(F)-2w_4.
\end{aligned}
\end{equation}

It follows that for the divisor 
\begin{equation}
\fD^{bal}:=\frac{1}{70}(\delta[1]+8\delta[1]+14\delta[3]),
\end{equation}
given by Proposition \ref{T:space-stability},
we
have 
\begin{align*}
\mu_{\rho}^{\fD^{bal}}(f) %=\frac{1}{70}( (48+14+8)\mult_{\rho}(F)- 32w_1-36w_2-&60(w_3+w_4)-28w_1-24w_2)=\\
&=\mult_{\rho}(F)-\frac{6}{7}(w_1+w_2+w_3+w_4).
\end{align*}
In analogy with Theorem \ref{T:degree-2-naive}, we 
can prove
\begin{theorem}\label{T:degree-1-naive} Suppose $R=\cO_{C,p}$ is the local ring of a smooth curve over an algebraically closed field of characteristic not $2$ or $3$.
Suppose $F \in R[x_1,x_2,x_3,x_4]_6$ is $\fD^{bal}$-semistable. Then $F=0$ defines a threefold del Pezzo fibration $X\to \Delta$ of degree $1$ such that $X$ has 
terminal singularities and the central fiber $X_0$ is reduced, and either irreducible or a union of two hypersurfaces in $\MP(1,1,2,3)$ of degree $3$ each passing through the $\frac{1}{3}(1,1,2)$-point.
\end{theorem}
Since this result, and its proof, is subsumed by Theorem \ref{terminality}, we omit the details.

%%%%%%%%%%%%%%%%

\subsubsection{Gorenstein canonical models} We use Koll\'ar stability to give a quick self-contained proof 
of the fact that every degree $1$ del Pezzo fibration over a smooth curve has a model with Gorenstein canonical singularities. This fact has been shown by Corti \cite[Corollary-Definition~3.4 and Proposition~3.13]{corti-annals} and Loginov \cite[Corollary~B]{Loginov} in charactersitic $0$ using MMP techniques.

\begin{theorem} Suppose $R=\cO_{C,p}$ is the local ring of a smooth curve over an algebraically closed field of characteristic not $2$ or $3$.
Let $\fD^{Gor}:=\epsilon \delta_0+\delta_1+\delta_2$, where $0< \epsilon \ll 1$. 
Suppose $F \in R[x_1,x_2,x_3,x_4]_6$ is $\fD^{Gor}$-semistable. Then $F=0$ defines a threefold del Pezzo fibration $X\to \Delta$ of degree $1$ such that $X$ has 
Gorenstein canonical singularities and the central fiber $X_0$ is integral. 
\end{theorem}

%\begin{prop} $\fD^{Gor}$-stable family has integral fibers and Gorenstein total family with at worst canonical singularities.
%\end{prop}
 \begin{proof}
By \eqref{E:HMK-sextics}, $\fD^{Gor}$-semistability implies the following
condition for any 1-PS $\rho$ acting with weights $(w_1,w_2,w_3,w_4)$ in some coordinates on $\PP(1,1,2,3)$:
\begin{equation}\label{E:stab-gor-1}
(2+48\epsilon) \mult_{\rho}(F)\leq 3w_3 + 2w_4 +\epsilon (60w_1+60w_2)+36w_3+32w_4).
%-\frac{1}{N^2}(v_1+v_2)-(3-\frac{1}{N^2})v_3 - 2v_4 \leq 0, \qquad N\gg 0.
\end{equation}
Take $\epsilon=\frac{1}{N^2}$, where $N\gg 0$.
Write $F=a(t)x_4^2+b(t)x_3^3+G$, where $G\in (x_1,x_2)$.   
Then $\fD^{Gor}$-semistability with respect to $\rho=(N,N,1,2)$ implies that $\left(2+\frac{48}{N^2}\right)\mult_{\rho}(F)\leq 7+\frac{1}{N^2}(100+120N)$. It follows that $\mult_{\rho}(F)\leq 3$ and so $b(0)\neq 0$.
Similarly, $\fD^{Gor}$-semistability with respect to $\rho=(N,N,1,1)$ implies that $\mult_{\rho}(F)\leq 2$, and so $a(0)\neq 0$.
We conclude that every $\fD^{Gor}$-semistable family $X\to \Delta$ has a central fiber
of the form $x_4^2+x_3^3+\cdots=0$, a sextic that avoids the singularities of $\PP(1,1,2,3)$. Such sextics are Gorenstein and necessarily 
integral, by degree considerations.

Consider now a Gorenstein singular point $P\in X_0$, which after a change of coordinates, we can take to be $P=[1:0:0:0]$.
We have $F=x_4^2+x_3^2+G(x_1, x_2, x_3)$, where $G\in (x_1,x_2)^2$.  Then 
for $\rho=(0,1,2,3)$, \eqref{E:stab-gor-1} gives
\[
(2+48\epsilon) \mult_{\rho}(F)\leq 12+ \epsilon(96+72+60)=12+228 \epsilon.
\]
Thus $\mult_{\rho}(F)\leq 5$. By the recognition criterion of Du Val singularities given in Lemma \ref{duVal} below, 
the hyperplane section $x_2=ct$, where $c\in k$ is general,  has a Du Val singularity at $P$. It follows that $P$ is a cDV Gorenstein singularity, hence canonical. 
%This implies that $F(t^3x_1, t^2x_2, t x_3, 1)$ there must be a monomial of $t$-valuation at most $5$.
%This proves that the singularity of the hyperplane section $x_3=t$ is Du Val.
%Non-isolated cDV singularities are still canonical, and so the claim follows. 

An example of $x_4^2+x_3^3+t^2(x_1^6+x_2^6)=0$ shows that $X$ can have non-isolated singularities. 

%%Consider a Gorenstein point $[1:-1:0:0]$ -- smooth point of X_0.
\end{proof}

%\begin{Remark}
%For any $\CL$ on the space of quartics in $\MP(2,1,1,1)$ or sextics in $\MP(1,1,2,3)$, there exists a semistable del Pezzo surface $X$ over $\MC[[t]]$ of degree $2$ or $1$ respectively which is either not terminal or has non-integral central fiber.
%Thus, to find standard models as semistable models we need a different parameter space.
%\end{Remark}

%%%%%%%%%%%%%%%%%%%%%%%%%%%%%
%%%%%%%%%%%%%%%%%%%%%%%%%%%%%
\subsection{Stability of $(6,3)$-intersections in $\PP(1,1,2,3,3)$} \label{dP1-stab}
%%%%%%%%%%%%%%%%%%%%%%%%%%%%%
%%%%%%%%%%%%%%%%%%%%%%%%%%%%%
%Let $S=k[x_0, x_1,x_2,x_3, x_4]$, where the grading is given by $w(x_0)=w(x_1)=3$,
%$w(x_2)=2$, and $w(x_3)=w(x_4)=1$. Set $G=\Aut(S)$, the group of grading-preserving $k$-algebra automorphisms of $S$.
%Let $\PP:=\PP(3, 3, 2, 1, 1)=\proj S$ and note that 
%$x_0=0$ defines $\PP(1,1,2,3) \hookrightarrow \PP$.
%$\PP(1,1,2,3) \hookrightarrow \PP$ is defined by $x_0=0$.

We let $M:=\Int_{\PP(1,1,2,3,3)}(6,3)$ be the parameter space of %type  
$(6, 3)$-intersections in $\PP(1,1,2,3,3)$.
%\begin{equation}\label{E:3-6-space}
%M=\left\{(H, F) \mid H\in \vert \cO_{\PP}(3)\vert=S_3, \ F\in \vert \cO_{H}(6)\vert=S_6/S_3H \right\}.
%\end{equation}
%Clearly, $M$ is a projective bundle over a projective space. % and so is smooth of Picard number $2$.
%It follows that $\Pic(M)=\ZZ^2$.
%The character group of $G$ is free of rank $3$.
%It follows that $\Pic^{G}(M)\simeq \ZZ^{5}$.

Then we have
\begin{equation}\label{HMK-3-6}
\begin{aligned}
\mu_{\rho}^{T_1}(f) &=-w_{\rho}(\{x_1,x_2\}) =-(w_1+w_2). \\
\mu_{\rho}^{T_2}(f) &=-w_{\rho}(\{x_1^2,x_1x_2,x_2^2,x_3\}) = -(3(w_1+w_2)+w_3). \\
\mu_{\rho}^{T_3}(f) &=-w_{\rho}(\{x_1^3,x_1^2x_2,x_1x_2^2,x_2^3,x_1x_3,x_2x_3,x_4,x_5\}) \\  &= -(7(w_1+w_2)+2w_3+(w_4+w_5)). \\
\mu_{\rho}^{J_3}(f)&=-\mult_{\rho}(H) \\
\mu_{\rho}^{J_6}(f)&\leq -\mult_{\rho}(F)-8 \mult_{\rho}(H)-7(w_1+w_2)-2w_3-(w_4+w_5). \\
\mu_{\rho}^{J_6-T_3-8J_3}(f)&\leq -\mult_{\rho}(F),
%\mu_{\rho}^{\Lambda_5}(f)  %&= 4 \mult_{\rho}(H_t)+w_{\rho}(P_{2})-w_{\rho}(P_{5}) \\ 
%&=4 \mult_{\rho}(H_t)+(w_2+3w_3+3w_4)-(4w_0+4w_1+10w_2+28w_3+28w_4), \\
%\mu_{\rho}^{\Lambda_6}(f)  %&= \mult_{\rho}(F_t)+\dim S_3 \mult_{\rho}(H_t)+w_{\rho}(P_{3})-w_{\rho}(P_{6})\\ 
%&=\mult_{\rho}(F_t)+8\mult_{\rho}(H_t)+(w_0+w_1+2w_2+7w_3+7w_4)\\
%&-(9w_0+9w_1+18w_2+48w_3+48w_4).
\end{aligned}
\end{equation}
where the equality holds in the last two lines whenever $(F,H)$ are equations of $f$ adapted to $\rho$.
%It follows that $T_1, T_2, T_3, J_3, J_6':=J_6-T_3-8J_3$ form a basis of $\Pic^G(M)\otimes \QQ$.

Consider now the three partial Fermats $\{f_{\ell}\}_{\ell=1,3,4}$ 
as defined in Definition \ref{D:partial-fermat}.
Then we have the following three
distinguished points of $M$: $I_\ell:=(f_{\ell}, x_5)$, $\ell=1,3,4$. 
The point $I_{\ell}$ is fixed by the one-parameter subgroup $\rho_{\ell}$ acting via $\rho\cdot x_{\ell}=tx_{\ell}$, 
by $\rho_{irr}=(1,1,2,3,3)$, and by $\rho_{5}=(0,0,0,0,1)$.   

As in \S\ref{Sub:boundary}, the boundary divisors $\delta[1]$, $\delta[2]$, and $\delta[3]$ in $\Hyp_{\PP(1,1,2,3)}(6)$ that we defined in \S\ref{S:boundary} extend uniquely to divisors on $M$. 
We now collect the HMK-indices computations that we have obtained so far:

\begin{center}
\renewcommand{\arraystretch}{1.3}
\begin{tabular}{|c|c|c|c|c|c|c|c|c|}
\hline
$\cL=$ & $T_1$ & $T_2$ & $T_3$ & $J_3$ & $J_6'$ & 
$\delta[1]$ & $\delta[2]$ & $\delta[3]$ 
\\
\hline
$\mu_{\rho_1}^{\cL}(I_1)=$ & $-1$ & $-3$ & $-7$ & $0$ & $0$ & $-60$ & $0$ & $0$ \\
\hline
$\mu_{\rho_3}^{\cL}(I_3)=$ & $0$ & $-1$ & $-2$ & $0$ & $0$ & $-36$ & $-3$ & $0$ \\
\hline
$\mu_{\rho_4}^{\cL}(I_4)=$ & $0$ & $0$ & $-1$ & $0$ & $0$ & $-32$ & $0$ & $-2$ \\
\hline
$\mu_{\rho_{irr}}^{\cL}(I_\ell)=$ & $-2$ & $-8$ & $-24$ & $-3$ & $-6$ & $0$ & $0$ & $0$ \\
\hline
$\mu_{\rho_{5}}^{\cL}(I_\ell)=$ & $0$ & $0$ & $-1$ & $-1$ & $0$ & $0$ & $0$ & $0$  \\
\hline
\end{tabular}
\end{center}
The HMK-indices of boundary line bundles are computed using \eqref{E:HMK-sextics}, and the remaining indices are obtained from
\eqref{HMK-3-6}, noting that the equations $(f_\ell, x_5)$ are adapted to each $\rho$.

Solving a simple system of linear equations, we obtain the expression for the boundary divisors in terms of tautological line bundles (we use additive notation for line bundles here):
\begin{align*}
\delta[1]&=-80T_1-28T_2+32T_3-32J_3-48J'_6 .\\
\delta[2]&=-9T_1+3T_2-J'_6.
\\ \delta[3]&=-2T_1-4T_2+2T_3-2J_3-J'_6, 
\end{align*}

In particular, 
\begin{equation}\label{E:HMK-degree-1}
\begin{aligned}
\mu_{\rho}^{\delta[1]}(f) &=48\mult_\rho(F_t)+32\mult_\rho(H_t)-60(w_1+w_2)-36w_3-60(w_4+w_5). \\
\mu_{\rho}^{\delta[2]}(f) &\geq \mult_\rho(F_t)-3w_2. \\
\mu_{\rho}^{\delta[3]}(f) &\geq \mult_\rho(F_t)+2\mult_\rho(H_t)-2(w_4+w_5).
\end{aligned}
\end{equation}

For $\fD^{bal}:=\frac{1}{70}\left(\delta[1]+8\delta[2]+14\delta[3]\right)$, we have
\begin{equation}\label{E:HMK-degree-1-bal}
\mu_{\rho}^{\fD^{bal}}(f) \geq \mult_\rho(F_t)+\frac{6}{7}\mult_\rho(H_t)-\frac{6}{7}(w_1+w_2+w_3+w_4+w_5).
\end{equation}
%as desired.
Collecting these results, we arrive at a semistability notion that will be central in our proof of Corti's conjecture in Section \ref{Sing-and-stab}.
\begin{prop}[T-semistability]\label{P:T-semistability} Fix $0<\epsilon \ll 1$. Then for an effective boundary divisor 
\begin{equation}
\fD^{ter}:=\fD^{bal}+ \epsilon(\delta[3]-\delta[2])=\frac{1}{70}\left(\delta[1]+8\delta[2]+14\delta[3]\right)+ \epsilon(\delta[3]-\delta[2])
\end{equation}
a $(6,3)$-intersection $(F,H)$ in $\PP_R(1,1,2,3,3)$ is $\fD^{ter}$-semistable only if for every choice of quasihomogeneous coordinates on $\PP_R(1,1,2,3,3)$, 
and for every weight system $\rho=(w_1,\dots,w_5)$, we have
\begin{equation*}
%\mu_{\rho}^{\fD^{ter}}(F,H) =
\mult_\rho(F)+ \frac{6}{7} \left( \mult_{\rho}(H) - \sum_{i=1}^{5} w_{i} \right)+\epsilon (2 \mult_{\rho}(H) - 2 (w_4 + w_5) + 3 w_3) \leq 0.
\end{equation*}
We will call $\fD^{ter}$-semistable $(6,3)$-intersections \emph{T-semistable}.
\end{prop}
Consequently, every $R$-point of $\Int_{\PP(1,1,2,3,3)}(6, 3)$ whose $K$-point defines a smooth degree $1$ del Pezzo in $\PP_K(1,1,2,3,3)$, or, equivalently, does not lie in $\cup_{i=1}^{3}\Supp(\delta[i])$, 
has a T-semistable model over $R$.\footnote{As the next section will illustrate, T stands for ``terminality.''}

%%%%%%%%%%%%%%%%%%%%%%%%%%%%%
%%%%%%%%%%%%%%%%%%%%%%%%%%%%%
\section{T-semistable and standard models of degree $1$ del Pezzo fibrations} \label{Sing-and-stab}
%and Corti's conjecture}
%%%%%%%%%%%%%%%%%%%%%%%%%%%%%
%%%%%%%%%%%%%%%%%%%%%%%%%%%%%

%\begin{Th}\label{Term-Th} There exists a Koll\'ar stability such that 
%whenever $\pi: X \to \Spec \MC[[t]]$ is a semistable del Pezzo fibration of degree $1$ (resp., $2$),
%then $X$ is terminal, and $X_0$ is integral (=reduced and irreducible).
%\end{Th}

%%%%%%%%%%%%%%%%%%%%%%%%%%%%%
%\subsection{T-semistability} \label{Terminality-setup}

In this section, we prove Theorem \ref{blackbox} Part (2).  We work over an algebraically closed field $k$ with $\ON{char}(k)\neq 2,3$. Let $R$ be a DVR (and a $k$-algebra) with the residue field $k$.
Let $K$ be the fraction field of $R$, $t$ be the uniformizer, and $\val\colon K\to \ZZ\cup \infty$ the $t$-valuation on $K$. 
%Let $R$ be a DVR with a fraction field $K$, valuation $\val$, and the residue field $k$. 

We work in $\PP_R(1_{x},1_{y},2_{z},3_{w},3_{s})$. Set $S=\Cox(\PP_R(1,1,1,2,3))=R[x,y,z,w,s]$, where $x,y,z,w,s$ is some choice of quasihomogeneous coordinates on $\PP_R(1,1,2,3,3)$. 
For $F\in S$, we denote by $F_0$ its residue class in $k[x,y,z,w,s]$. 
For $f=f(x,y,z,w,s)\in K[x,y,z,w,s]$ (not necessarily homogeneous), we define $\val f$ to be the minimum of the valuations of all the coefficients of $f$.

Recall the definition of a $T$-semistable $(6,3)$-intersection in $\PP_R(1,1,2,3,3)$ from Definition \ref{T-semi} or Proposition \ref{P:T-semistability}. The goal of this section is to prove:
%Next we specialize to the geometric case, and state the result that will complete the proof of the Corti's conjecture:
\begin{Th}[{Theorem \ref{blackbox} Part (2)}] \label{mainthm2} 
For every T-semistable $(6,3)$-intersection $(F,H)$, the scheme $X=\{F=H=0\}\subset \PP_R(1,1,2,3,3)$ satisfies: 
\begin{enumerate}
\item The morphism $X\to \spec R$ is flat with integral fibers. 
\item Every point of $X$ is either $2$ or $3$-Gorenstein, and $-6K_X$ is relatively ample over $\spec R$.  
\item  $X$ has isolated singularities. For every $k$-point $P\in X$, the general elephant through $P$ has an isolated Du Val singularity at $P$. Consequently,  $X$ has terminal singularities. % and $6$-Gorenstein. 
\end{enumerate}
In particular, $X\to \spec R$ is a standard degree $1$ del Pezzo fibration. 
\end{Th}

The { plan of the proof} is as follows. Part (1) is proved in Lemma \ref{L:integral},
which also introduces a key dichotomy of the remainder of the proof. 
We prove that $X$ has isolated singularities in Lemmas \ref{isolated-1}, \ref{isolated-2}, and \ref{isolated-3}.  
Once we establish that $X$ is a $(6,3)$-complete intersection in $\PP_R(1,1,2,3,3)$ with isolated singularities, then adjunction is applicable and gives 
$K_X = \CO_{\PP_R(1,1,2,3,3)}(-1)\vert_X$. This proves Part (2) of Theorem \ref{mainthm2}. Finally, we will show that
the general element of $\vert-K_X\vert=\vert\CO_{X}(1)\vert$ through every $k$-point
of $X$ has at worst Du Val singularities.  This is the most involved part of the analysis. 
The terminality of the total space then follows by the structure theory of threefold terminal singularities; see Proposition \ref{P:elephants} below. %in \S\ref{terminality} below.

%\begin{Remark}
%For degree $2$ we also need to twist the line bundle $\cL^{bal}$.
%We set $\CM = \CO(\delta_1)$ and consider stability with respect to $\CL^{ter} = \CL^{bal} + \epsilon \CM$.
%\end{Remark}

%It is a case analysis and for every case we either prove that assertions of Theorem \ref{Term-Th} %hold or explicitly give a $1$-parameter subgroup $T$ with $\mu_T^\CL (X) >0$.

%%%%%%%%%%%%%%%%%%%%%%%%%%%%%
\subsection{Threefold terminal singularities and their elephants} \label{terminality}
%%%%%%%%%%%%%%%%%%%%%%%%%%%%%
%%%%%%%%%%%%%%%%%%%%%%%%%%%%%
%We work over an algebraically closed field $k$ of $\ON{char}(k)\neq 2,3$. 
We collect relevant results from the classification theory of terminal threefold singularities, that are well-known over $\CC$, but perhaps less familiar 
in our setting, over an algebraically closed field $k$ of $\ON{char}(k)\neq 2,3$. To state them, first recall that for a threefold $X$, \emph{an elephant} (resp., \emph{a general elephant}) is
an element (resp., a general element) of the anticanonical linear system $\vert -K_X\vert$. An overarching principle, 
first formulated by Reid \cite[p.393]{reid-YPG}, is that threefold terminal singularities
are precisely those with canonical (i.e., Du Val) general elephants. This is known over $\CC$ thanks to the explicit Koll\'ar-Shepherd-Barron-Mori classification (\cite{Mori},\cite[Theorem~6.5]{KSB}) 
of terminal threefold singularities,
but we are not aware of an analogous statement\footnote{The point being that, in positive characteristic, there 
could be other families of terminal singularities not covered by Mori's list.} in positive characteristic. However, for our purposes, only one direction is needed,  as given in a recent paper of Koll\'ar,
and valid in arbitrary characteristic: 
\begin{prop}[{\cite[Corollary 11]{kollar-MMP-recent}}]\label{P:elephants} Suppose $P\in X$ is an isolated normal threefold singularity and $E\subset X$ an elephant containing $P$. If $E$ has an isolated Du Val singularity at $P$, then $P\in X$ is a terminal singularity.
\end{prop}

%It is well known that $P\in X$ is terminal if $P \in E$ is canonical for a general elephant $E$, that is a general element $E \in \big| -K_X \big|$.
%Furthermore, if $P\in E$ is Du Val and characteristic of the base field is $>5$, then $P \in E$ is canonical.

%In this paper this is the primary way of proving terminality of a singularity $P \in X$.
We use the following recognition criterion for Du Val surface singularities:

\begin{Lemma} \label{Key-Sing}\label{duVal} 
Suppose a hypersurface in $\spec k[[x,y,z]]$ given by an equation
$f(x,y,z)=0$ has an isolated multiplicity $2$ singularity at $P=(0,0,0)$. 
Then:
\begin{enumerate}
\item If the tangent cone of $f$ is reduced, then $f$ is formally equivalent to $xy+z^{n+1}=0$, a type $A_n$ singularity, for some $n\geq 1$.
\item If $f= x^2 + g(y,z)$ and the tangent cone of $g(y,z)$ is a cubic with 
at least $2$ different linear factors, then $f$ is formally equivalent to $x^2+z(y^2+z^{n-2})=0$, a type $D_n$ singularity, for some $n\geq 4$.
\begin{comment} If the tangent cone of $F$ is $x^2$ and the tangent cone of $F(0,y,z)$ is a cubic with 
at least $2$ different linear factors, then $P$ is a $D_n$ singularity. 
%Note that $F(x,y,z)=x^2+2xC(x,y,z)+F(0,y,z)=(x+H)^2+F(0,y,z)-C(x,y,z)^2$, where $\deg C\geq 2$. In particular, the cubic part of $F(0,y,z)$ is unaffected by $C$.
\end{comment}
\item If $f=x^2 + y^3 + h(y,z)$, where $\deg h\geq 4$, and $\val h(t^2y, t z) \leq 5$, or, equivalently, $h(y,z)$ has a $z^4$, $yz^3$, or $z^5$ term, 
then $P$ is a type $E_6, E_7$, or $E_8$ singularity.
\begin{comment}
%% x^2+2xz^2+z^4+y^3+z&3=(x+z^2)+y^3 is a cautionary counterexample to the following statement:
 If the tangent cone of $F$ is $x^2$, the tangent cone of $F(0,y,z)$ is $y^3$, and $\mult_t F(t^3x, t^2y, t z) \leq 5$,
then $P$ is an $E_n$ singularity.
%Note that $F(x,y,z)=x^2+2xC(x,y,z)+(y^3+h(y,z))=(x+C)^2+(y^3+h(y,z))-C(x,y,z)^2$, where $\deg C\geq 2$. 
\end{comment}
\end{enumerate}
\end{Lemma}
\begin{proof} This follows by the standard argument as in \cite[p.375]{reid-YPG} (see also \cite{lipman, AGZV}), with the exception that when $\ON{char}(k)=5$, there are two distinct formal isomorphism classes
of $E_8$ singularities, namely $x^2+y^3+z^5$ and $x^2+y^3+z^5+yz^4$ \cite{artin-RDP-77}.
\end{proof}

\subsection{Integrality of the central fiber and the key dichotomy} \label{Step-1}
%%%%%%%%%%%%%%%%%%%%%%%%%%%
%%%%%%%%%%%%%%%%%%%%%%%%%%%%%
We begin by proving that a T-semistable model is flat over $\spec R$ and has an integral central fiber. In the process, we demonstrate the necessity of working 
in $\PP_R(1,1,2,3,3)$.  
Our analysis of singularities and the proof of terminality then fall into two distinct cases depending on the behavior of the central fiber. %$X_0$.
The following result describes this dichotomy:
\begin{Lemma}\label{L:integral}
Suppose $(F,H)$ is a T-semistable $(6,3)$-intersection. % in $\PP_R(1,1,2,3,3)$. 
Then $X:=\{F=H=0\}\subset \PP_R(1,1,2,3,3)$ is, relatively over $\spec R$, a complete $(6, 3)$-intersection %in $\PP_R(1,1,2,3,3)$  
with an integral central fiber. 
Moreover, in some system of coordinates, one of the following holds:
\begin{enumerate}
\item
$H = s$, its vanishing defines $\MP(1,1,2,3)\hookrightarrow \MP(1,1,2,3,3)$ in every fiber,
and $X_0$ is an integral degree $6$ hypersurface in $\MP_k(1,1,2,3)$.
\item
$H_0 = y^3 - xz$, its vanishing defines $\MP_{k}(1,2,9,9)\hookrightarrow \MP_{k}(1,1,2,3,3)$ in the central fiber,
and $X_0$ is an integral degree $18$ hypersurface in $\MP_k(1_{\alpha},2_{\beta},9_{w},9_{s})$ not containing the line $\alpha=\beta=0$.
\end{enumerate}
\end{Lemma}
The first case allows us to find a standard model of $X_K$ which is a weighted degree $6$ hypersurface in $\PP_R(1,1,2,3)$.
The second case is more involved because the ambient weighted projective space also degenerates. That the second case necessarily appears
is illustrated in the proof of this lemma. 

\begin{proof}[Proof of Lemma \ref{L:integral}] We claim $H_0$ is irreducible. If not, then
in some coordinates, % on $\PP(1,1,2,3,3)$, 
$x\mid H_0$. For $\rho = (1,0,0,0,0)$, we have $\mult_\rho(F) \geq 0$ and $\mult_\rho(H) = 1$, making $(F,H)$ $\rho$-unstable by \eqref{num-ss}. %\eqref{T-semi}.

\begin{remark} This is the only place in the proof of terminality where we use the perturbed stability condition (i.e., need $\epsilon>0$ in Definition \ref{T-semi}).  
Note also that we use in an essential way the closedness of the residue field $k$ here.
\end{remark}

Up to a change of coordinates, there are exactly two possibilities for an irreducible $H_0$: 
$H_0=s$ and $H_0=y^3-xz$.
\begin{enumerate}[(1)]
\item $H_0=s$. We can change coordinates so that $H = s$. Then $H = 0$ defines 
$\MP(1,1,2,3)$ in every fiber. The central fiber $X_0$ is non-integral if and only if $F_0(x,y,z,w,0)$ 
%Then $F_0(x,y,z,w,0)=0$ 
factors non-trivially in $k[x,y,z,w]_6$. Let $A$ be its smallest positive degree factor.
\begin{enumerate}[(i)]
\item If $\deg A = 1$, then, up to a change of coordinates, $A = y$ and $(F,H)$ is destabilized by $\rho = (0,1,0,0,0)$: $\mult_{\rho}(F)=1$ and $\mult_{\rho}(H)\geq 0$.
\item If $\deg A = 2$, then, up to a change of coordinates, $A = z$ and $(F,H)$ is destabilized by $\rho = (0,0,1,0,0)$.
\item Suppose now $\deg A = 3$. If $A = w$ in some coordinates, then $(F,H)$  is destabilized by $\rho = (0,0,0,1,0)$.
Otherwise, $A$ does not depend on $w$ and so we may write %$(F,H)$ as
%\begin{align*}
\[F(x,y,z,w,0)=A(x,y,z) B(x,y,z) + tG(x,y,z,w) .\]
%, \\
%H&=s.
%\end{align*}
Make a change of coordinates $s'=s-A(x,y,z)$. Then in the new coordinates $(F,H)$ becomes
\begin{align*}
(F,H) &=(s'B(x,y,z) +tG(x,y,z,w), s'+A(x,y,z))
\end{align*}
For $\rho=(0,0,0,0,1)$, we now have $\mult_\rho(F) = 1$ and $\mult_\rho(H) = 0$, making $(F,H)$ $\rho$-unstable by \eqref{num-ss}.
\end{enumerate}

\item $H_0 = y^3-xz$. This cubic is precisely
the image of a closed immersion 
\begin{equation}\label{immersion}
\MP_k(1_{\alpha},2_{\beta},9_w,9_s) \hookrightarrow \MP_k(1_x,1_y,2_z,3_s,3_w) \quad \text{via $x = \alpha^3$, $y = \alpha \beta$, $z = \beta^3$.}
\end{equation}
Semistability with respect to $\rho=(1,1,1,1,1)$ implies that $F_0\notin (x,y,z)$, and so 
the central fiber $X_0$ is a degree $18$ hypersurface %$\tilde{F}_0(\alpha,\beta,w,s):=
$F_0(\alpha^3,\alpha\beta,\beta^3,w,s)=0$ in $\MP_k(1_{\alpha},2_{\beta},9_w,9_s)$ with a non-zero quadratic term in $w,s$. 
If $F_0(\alpha^3,\alpha\beta,\beta^3,w,s)$ factors non-trivially in $k[\alpha,\beta,w,s]$, then necessarily it does into two degree $9$ factors with a non-zero linear term in $w,s$ each.
After a change of variables, %and up to a multiples of $H$, 
we can assume that $s\mid F$. Then $(F,H)$ is destabilized by $\rho=(0,0,0,0,1)$. 
\end{enumerate}
\end{proof}

\begin{remark}\label{why}
The case (1.iii) in the proof of this lemma is the only place in the proof of Theorem \ref{mainthm2} where starting with a family whose central fiber lives in $\MP(1,1,2,3)$, 
we must use a one-parameter subgroup which is not a $K$-point of $\Cox(\PP(1,1,2,3))$ but a $K$-point of $\Cox(\PP(1,1,2,3,3))$.  In other words, 
the reducibility of the central fiber in the family
\[
t (w^2+G(x,y,z)) + A(x,y,z) B(x,y,z) = 0,  %%% NOTE: In the prior version, the family without G term was not generically smooth.
\]
is the sole reason we consider models in $\MP(1,1,2,3,3)$ instead of $\MP(1,1,2,3)$. 
\end{remark}

%We have reduced to the case of a hypersurface in $\MP(1,1,2,3)$.

This concludes the proof of Part (1) of Theorem \ref{mainthm2}, which shows
that $X$ is a threefold fibered in integral surfaces over $\spec R$. Since the generic fiber of $X$ smooth, we conclude that $X$ is nonsingular in codimension $1$. 
By adjunction, $-K_X=\cO_{\PP_{R}(1,1,2,3,3)}(1)\vert_X$ and so $-6K_X$ is a relatively ample line bundle. This establishes Part (2) of Theorem \ref{mainthm2}.

The proof of Part (3) now proceeds according to one of the two possibilities given in Lemma \ref{L:integral}, 
and is done in \S\ref{Step-2}, \ref{Step-3}, respectively. In each case, we first show that $X$ has isolated singularities, and then for every singularity $P\in X$,
 find an elephant  (= an anticanonical divisor) $E$ through $P$ such that $P\in E$ is Du Val. 
 
Since  $-K_X=\cO_{\PP_{R}(1,1,2,3,3)}(1)\vert_X$, every hyperplane $ax+by$, with either $\val a=0$ or $\val b=0$,
restricts to an elephant on $X$ fibered in elliptic curves over the base. %The Bertini theorem over $K$ (at least in characteristic $0$) says that for a general $a,b$, the resulting elephant is smooth over $K$. 
We will need the following Bertini-like result: %however will work with a particular choice of $a,b$ necessitating the following lemma:
\begin{lemma}\label{elephant-K} For every choice of coordinates, let $E$ be the elephant given by the equation $y=cx$ (resp., $y=ctx$), where $c\in k$ is general. 
Then $E_K$ is smooth over $K$. 
\end{lemma}
\begin{proof}
After a change of variables affecting only $z,w,s$, the generic fiber $X_K$ can be written as a sextic in $\PP(1_x,1_y,2_z,3_w)$ in the Weierstrass form:
\[
w^2+z^3+p(x,y)z+q(x,y)=0,
\]
with $p(x,y)\in K[x,y]_4$ and $q(x,y)\in K[x,y]_6$.  The discriminant $D(x,y):=4p^3+27q^2$ is a non-zero binary sextic by the smoothness of $X_K$. The discriminant of the elliptic curve $E_K$ is 
$D(1,c)$ (resp., $D(1, ct)$), which will be nonzero for a general $c\in k$. The claim follows. 
\end{proof}

%%%%%%%%%%%%%%%%%%%%%%%%%%%%%
%%%%%%%%%%%%%%%%%%%%%%%%%%%%%
\subsection{Central fiber in $\MP(1,1,2,3)$} \label{Step-2}
%%%%%%%%%%%%%%%%%%%%%%%%%%%%%
%%%%%%%%%%%%%%%%%%%%%%%%%%%%%
Suppose we are in case (1) of Lemma \ref{L:integral}. We can write $(F,H)=(F(x,y,z,w), s)$ and 
treat $F$ as a sextic in $R[x,y,z,w]$. In particular, all of our coordinate changes in this subsection
will fix $s$ and come from $\Aut_{gr} (R[x,y,z,w])$.  The threefold $X$ is then a sextic $F=0$ in $\PP_R(1_x,1_y,2_z,3_w)$

The T-semistability of a complete intersection $(F(x,y,z,w), s)$ can then be interpreted in terms of $F$ alone.
Namely, if $(F,H)$ is T-semistable, then $F(x,y,z,w) \in R[x,y,z,w]_6$ satisfies:
\begin{enumerate} 
\item For every  %$\GG_m$-action
%$
%\rho(t)\cdot (x,y,z,w) = (t^{w_1}x,t^{w_2}y, t^{w_3}z, t^{w_4}w),$ we must have
weight system $\rho=(a,b,c,d)$, we have that
\begin{equation}\label{E:stab-cond} 
\mult_{\rho}(F) \leq \frac{6}{7}(a+b+c+d).
\end{equation}
\item $F_0(x,y,z,w)$ is irreducible in $k[x,y,z,w]_6$.
\end{enumerate}
A sextic $F(x,y,z,w)$ satisfying these conditions will be called T-semistable. A sextic $F$ satisfying (resp., violating) \eqref{E:stab-cond} will be called $\rho$-semistable (resp., $\rho$-unstable).

%%%%%%%%%%%%%%%%%%%%%%%%%%%%%%%%%%%%%%%%%%%%%%%%

%Suppose $F(x,y,z,w)=0$ is a T-semistable sextic in $\PP_R(1,1,2,3)$ in the sense of .  
We begin with a preparatory result:
\begin{lemma}\label{ABCD}
Suppose that, in some coordinates, a T-semistable sextic has equation
\begin{equation}\label{E:ABCD}
%F(x,,y,z,w)=
o(t)w^2+(\underbrace{p_1(x,y,t) z + p_3(x, y, t)}_{a_3(x,y,z,t)}) w + \underbrace{p_0(t) z^3 + p_2(x,y,t) z^2 + p_4(x,y,t) z+ p_6(x,y,t)}_{a_6(x,y,z,t)}.
\end{equation}
Then it must satisfy one of the following (up to scaling by a unit in $R$):
\begin{enumerate}
\item[(A)] $o(t)=1$. % (i.e., $\val_t(o(t))=0$) 
%and $p_1(t)=p_2(t)=0$. 
\item[(B)] $o(t)=t$, and either 
\begin{enumerate}[(i)]
\item $p_1(x,y,0)\neq 0$, or 
\item $\val(p_0(t))\leq 1$, or 
\item $p_2(x,y,0)\neq 0$.
\end{enumerate}
\item[(C)] $\val(o(t))\geq 2$, and $p_1(x,y,0)\neq 0$.
\item[(D)] $\val(o(t))= 2$, $p_1(x,y,0)=0$, $p_3(x,y,0)\neq 0$, and $p_0(t)=1$.
\end{enumerate}
and must have either $p_3(x,y,0)\neq 0$, or $p_4(x,y,0)\neq 0$, or $t^2\nmid p_6(x,y,t)$.

In particular, either $a_3(x,y,z,0)\neq 0$ or $\val(o(t))\leq 1$.
\end{lemma}

\begin{proof}
If $\val(o(t))\geq 1$, then stability with respect to $\rho=(1,1,1,2)$ ensures that either $p_1(x,y,0)\neq 0$, or $\val(p_0(t))\leq 1$, or $p_2(x,y,0)\neq 0$.

Assume that $\val (o(t))\geq 2$ and $p_1(x,y,0)=0$. Taking $\rho=(1,1,2,2)$,
we see that $p_3(x,y,0)\neq 0$. 
Taking $\rho=(1,1,1,1)$, we see that $p_0(0)\neq 0$. 
Taking $\rho=(1,1,1,0)$ shows that $\val (o(t))=2$. 

Finally, semistability with respect to $\rho=(0,0,1,1)$ ensures that either $p_3(x,y,0)\neq 0$, or $p_4(x,y,0)\neq 0$, or $t^2\nmid p_6(x,y,t)$.
\end{proof}

\begin{Lemma}\label{isolated-1}
Suppose $F$ is T-semistable, then $X$ has only isolated singularities.
\end{Lemma}
\begin{proof}
Suppose $X$, and hence $X_0$, is singular along a curve. We will show that in some coordinates (a component of) this curve 
is given by $L_1=L_2=0$, where $\{L_1, L_2\}\subset \{x,y,z,w\}$ (Cases I and II below). Since $X$ is singular along the locus $L_1=L_2=t=0$, we must have $F\in (L_1,L_2,t)^2$.
Then $F$ is destabilized by the one-parameter subgroup $\rho$ acting with weight $1$ on $L_i$'s and weight $0$ on the remaining coordinates,
leading to a contradiction.

Case I: $o(t)=1$.  After a change of variables, we can take $a_3(x,y,z,t)=0$. Then $a_6(x,y,z,0)$ is non-reduced.
It is not a square since $F_0$ is integral, so after a further change of coordinates we have $L_1^2 \mid a_6(x,y,z,0)$,
where either $L_1=x$ or $L_1=z$. Then $L_1=w=t=0$ is in the singular locus of $X$.

Case II: $a_3(x,y, z,0)\neq 0$ and $\val o(t)\geq 1$. We have $F_0= a_3(x,y, z,0)w +a_6(x,y,z,0)$. 
Since $a_3$ and $a_6$ are coprime by integrality, we see that the singular locus of $X_0$ must consist of 
the fibers of the projection map over the $0$-dimensional locus $a_3(x,y, z,0)=a_6(x,y, z,0)=0$ in $\PP_k(1,1,2)$. Up to a change of coordinates, we may take 
either $x=y=t=0$ or $y=z=t=0$ to be in the singular locus of $X_0$.

Case III: $a_3(x,y,z,0)=0$ and $o(t)=t$. %We can assume $a_3(x,y,z,t)=0$ after a change of variables.
For $F=tw^2+a_6(x,y,z,t)$ to have non-isolated singularities along $t=0$, we must have $a_6(x,y,z,0)$ to have non-isolated singularities as a sextic in $\PP_k(1,1,2)$.
Then $F_0=a_6(x,y,z,0)$ is non-reduced, contradicting the T-semistability. % assumption.

%%% NOTA BENE: This case was not quite right in the original submission

\end{proof}

We proceed to prove that the isolated singularities of $X$ are terminal by proving that the general elephant through every singular point has Du Val singularities.
%%We may assume that $a_3(x,y,z)$ does not have monomial $xz$ since we may add a multiple of $sH_3$ to eliminate it. -- NOT SURE WE CAN OR SHOULD DO THIS
Let $P \in X$ be a singular point. Then up to a change of coordinates, we have:
\begin{enumerate}
\item $P=[1:0:0:0]$, a Gorenstein point.
\item $P=[0:0:1:0]$, a 2-Gorenstein point
\item $P=[0:0:1:1]$, a Gorenstein point.
\item $P=[0:0:0:1]$, a 3-Gorenstein point. 
\end{enumerate}

We begin with:
\begin{lemma}\label{isol-el-1} For every choice of coordinates, let $E$ be the elephant given by the equation $y=cx$ (resp., $y=ctx$), where $c\in k$ is general. 
Then $E$ has isolated singularities (resp., isolated singularities away from $x=y=t=0$).
\end{lemma}

\begin{proof} The generic fiber of $E$ is smooth by Lemma \ref{elephant-K}.
It remains to show that $E$, given by the sextic equation $F(x,cx,z,w)$ (resp., $F(x,ctx,z,w)$) in $\PP_R(1_x,2_z,3_w)$ 
has isolated singularities in the central fiber (respectively, isolated singularities away from $x=t=0$). 
Suppose not, then every component of the one-dimensional singular locus of $E$ is a non-reduced component of $E_0$, and, %in $E_0 \subset \PP(1_x,2_z,3_w)$ 
by the degree considerations, 
has equations $L=t=0$, where $L\in \{x,z,w\}$.  

Consider first the case of $y=cx$. Then $F(x,cx,z,w)\in (L,t)^2$. If $L=x$, then $F(x,y,z,w)\in (x,y,t)^2$, and
is unstable with respect to $\rho=(1,1,0,0)$ by \eqref{E:stab-cond}.  If $L\in \{z,w\}$, then $F(x,y,z,w)\in (L,t)^2$ and is unstable with respect to $\rho=(0,0,1,0)$ or $\rho=(0,0,0,1)$.  

Consider now $y=ctx$. Then $F(x,ctx,z,w)\in (L,t)^2$, where $L\in \{z,w\}$ because we consider singularities away from $x=y=t=0$ in this case.
Then $F\in (y,L,t)^2$ and so is unstable with respect to $\rho=(0,1,1,0)$ or $\rho=(0,1,0,1)$ by \eqref{E:stab-cond}.
\end{proof}

\begin{Prop}\label{P:gorenstein-singularity}
Let $P = [1:0:0:0] \in X$ be a singular point. Let $E$ be the elephant through $P$ given by the equation $y=ctx$, where $c\in k$ is general.
Then $P\in E$ is Du Val. Consequently, $P\in X$ is cDV and hence terminal.
\end{Prop}
\begin{proof}
In the affine chart $x=1$, and local coordinates $y, z, w, t$, the equation of $X$ is $f(y,z,w,t):=F(1,y,z,w)$ and the equation of $E$ is $e(z,w,t):=f(ct, z, w, t)$.          

By Lemma \ref{isol-el-1}, $P$ is an isolated singularity of $E$. 
Since $F$ is semistable for $\rho = (0,1,1,1)$, we have \[\val e(tz,tw,t)=\val f(ct, tz, tw, t)=\val f(ty, tz, tw, t)=\mult_{\rho}(F)\leq 2.\] 
Thus $P$ is a double point of $E$,
and we can use Lemma \ref{duVal} to prove that $P\in E$ is Du Val. Let $o(t)$ be the coefficient of $w^2$ in $F$.

%By Lemma \ref{duVal}, we may assume that the tangent cone of $E$ is non-reduced, since otherwise $P \in X$ is an A singularity.
%We consider the four cases of Lemma \ref{ABCD}.

Case I: $o(t)=1$.  After a change of coordinates, we can write $f=w^2+a(y,z,t)$. If $a(y,z,t)$ has a quadratic term, then so does $a(ct,z,t)$
and so $P\in E$ is an $A$-singularity by Lemma \ref{duVal}(a). Hence we can assume that $$a=p_0(y,t)z^3+p_1(y,t)z^2+p_2(y,t)z+p_3(y,t),$$
where $\val p_i(ct,t) \geq i$.  By $\rho = (0,1,1,2)$-semistability, we must have $\val f(ty, tz, t^2w, t)=\mult_{\rho}(F)\leq 3$. 
It follows that $\val p_i(ct,t) = i$ for at least one $i$.  Consider the following three subcases:

i) If $p_0=1$, we can change coordinate $z$ to arrive at
$f=w^2+z^3+p_2(y,t)z+p_3(y,t)$. If $p_2(ct,t)z+p_3(ct,t)$ has a cubic term, then $P\in E$ is a $D$ singularity by Lemma \ref{duVal}(b). Suppose not.
Then stability with respect to $\rho=(0,1,2,3)$ ensures that either $\val p_3(ct,t) \leq 5$ or $\val p_2(ct,t) \leq 3$, giving a type $E$ singularity by Lemma \ref{duVal}(c). 

ii) If $t\mid p_0$, and either $\val p_1(ct,t)=1$ or $\val p_2(ct,t)=2$, then %the tangent cone of $a_6(1,ct,z,t)$ is a cubic with at least two distinct factors, giving a 
we have a type $D$ singularity by Lemma \ref{duVal}(b).

iii) If $t\mid p_0$, and $\val p_1(ct,t) \geq 2$ and $\val p_2(ct,t) \geq 3$, then necessarily $\val p_3(ct,t)=3$. In this case, semistability of $F$ with respect to $\rho = (0,1,0,1)$ ensures that $p_0=t$. 
It follows that $e=w^2+t^3 +tz^3+z^2 p_1(ct,t)+z p_2(ct,t)$, an $E_7$ singularity by Lemma \ref{duVal}(c).

Case II: $\val o(t)\geq 1$.   If the tangent cone of $f$ is a reduced quadric, then so is the tangent cone of $e$ and we have an $A$ singularity.  
Assume from now on that the tangent cone of $f$ is a non-reduced quadric. In particular, it cannot depend on $w$ and so lies in $(y, z,t)^2$.
In particular $xzw, x^2yw\notin F$.\footnote{We write $m\in F$ (resp., $m\notin F$) to indicate that a monomial $m$ has a nonzero (resp., zero) coefficient in $F$.} 
 %(and of course $x^3w\notin F$ 
%since $P$ is a singular point). 
Semistability of $F$ with respect to $\rho=(0,1,1,0)$,
ensures $tw^2\in F$. Semistability of $F$ with respect to $\rho=(0,1,0,0)$,
ensures that either $z^2x^2 \in F$ or $z^3\in F$.
 If $z^2x^2 \in F$, then $e(z,w,t)=z^2+tw^2+\cdots$ is a $D$ singularity. 
If $z^2x^2\notin F$ and $z^3\in F$, then the tangent cone of $e$ must be $t^2$, and so $e(z,w,t)=t^2+tw^2+z^3+\cdots$ is an $E_6$ singularity. 
\end{proof}

\subsubsection{Singularities along $x=y=0$} We now analyze the singularities of $X$ along the line $x=y=0$ passing
through both cyclic quotient singularities of $\PP(1,1,2,3)$. We consider an elephant $E$ given by $y=cx$, where $c\in k$ is general.
By Lemma \ref{isol-el-1}, the singularities of $E$ are isolated. We prove that they are Du Val along $x=y=0$, thus establishing the terminality of $X$
along this line.   %% Or do we only prove Du Valness along singularities of X?

\begin{Lemma}
Suppose $P = [0:0:0:1] \in X$ is a %an isolated 
$3$-Gorenstein singularity. Let $E$ be the elephant through $P$ given by $y=cx$, where $c\in k$ is general. 
Then $P\in E$ is Du Val.  %% Is isolated even needed here?  Is it clear that $P$ is 3-Gorenstein?
Hence $P \in X$ is a terminal singularity. 
\end{Lemma}
\begin{proof}
In the coordinates $\alpha:=x^3, \beta:=z^3, \gamma:=xz$, and $t$ on the affine chart $w=1$, the equations of $E$ are % in $\spec k[[\alpha, \beta, \gamma, t]]$ are
\begin{equation}\label{ge-1/3}
\begin{aligned}
\alpha\beta &=\gamma^3, \\
0&=o(t) + p_1(t) \gamma+p_2(t) \alpha+ q_1(t)\beta+q_2(t)\gamma^2+q_3(t)\alpha\gamma+q_4(t)\alpha^2.
\end{aligned}
\end{equation}
We proceed to consider the possibilities enumerated in Lemma \ref{ABCD}, noting that only cases (B--D) are possible.

Case (B): $o(t)=t$. Then \eqref{ge-1/3} is a surface $A_2$-singularity.  (It is easy to see that $P\in X$ is a $\frac{1}{3}(1,1,2)$-point.)

Case (C): $\val_t(o(t)) \geq 2$ and $p_1(t)=1$. Then \eqref{ge-1/3} is a surface $A_n$-singularity.  %(The index $n$ can be read off from valuations of $o(t), p_2(t)$, and $q_1(t)$?????) 

Case (D): $\val_t(o(t)) =2$ and $p_2(t)=q_1(t)=1$. Plugging in 
$\beta=-t^2-\alpha-p_1(t)\gamma-q_2(t)\gamma^2-q_3(t)\alpha\gamma-q_4(t)\beta^2$ into $\alpha\beta=\gamma^3$, we get a singularity
\[
\alpha t^2+\alpha^2+\gamma^3+\cdots=0,
\]
which is formally equivalent to $\alpha^2+\gamma^3+t^4=0$, an $E_6$-singularity.  (One can see that $P\in X$ is $\frac{1}{3}cD_4$ singularity according to the Mori classification scheme \cite{Mori},\cite[Theorem~6.5]{KSB}.)

\end{proof}

\begin{Lemma}
Suppose $P = [0:0:1:0] \in X$ is a %an isolated 
$2$-Gorenstein singularity. Let $E$ be the elephant given by $y=cx$, where $c\in k$ is general. 
Then $P\in E$ is Du Val.  %% Is isolated even needed here?
Hence $P \in X$ is a terminal singularity.
\end{Lemma}

\begin{proof} Keep the notation of Lemma \ref{ABCD}, and set $p_i(t)=p_i(c,cx,t)/x^{i}$. 
In the coordinates $\alpha:=x^2, \beta:=w^2, \gamma:=xw$, and $t$, in the affine chart $z=1$, the equations of $E$ are % in $\spec k[[\alpha,\beta,\gamma,t]]$ are
\begin{equation}\label{ge-1/2}
\begin{aligned}
\alpha \beta &=\gamma^2, \\
0&= o(t) \beta+p_1(t) \gamma+p_3(t)\alpha\gamma+ p_0(t) +p_2(t)\alpha+p_4(t)\alpha^2+p_6(t)\alpha^3 &=0.
\end{aligned}
\end{equation}
Only cases (A--C) of Lemma \ref{ABCD} are possible. We can also assume that $p_0(t)=t^k$, where $k\geq 2$,
since $\val_t(p_0(t))=1$ implies that $P\in E$ is an $A_1$-singularity. 

Case (A): $o(t)=1$. Eliminating $\beta$, we get a surface singularity:
\[ 
\gamma^2+ t^k \alpha +p_2(t)\alpha^2+p_4(t)\alpha^3+p_6(t)\alpha^4=0.
\]
If $\val_t(p_2(t))=0$, then we have an $A$-singularity. If $k=2$ or $\val_t(p_2(t))=1$, then we have a $D$-singularity. 
Suppose $k\geq 3$ and $\val_t(p_2(t))\geq 2$. Semistability of $F$ with respect to $\rho=(1,1,0,2)$, implies 
that $\val_t(p_4(t))=0$.  % \rho=(1,1,0,2)$.
Now if $k=3$, we get an $E_7$-singularity.  If $k\geq 4$, we destabilize $F$ using $\rho=(1,1,1,3)$.

Case (B): $o(t)=t$. It $p_1(t)=1$, eliminating $\gamma$, we get a surface singularity:
\[
\alpha\beta=(t\beta+ t^k +p_3(t)\alpha+p_4(t)\alpha^2+p_6(t)\alpha^3)^2+\cdots, 
\]
which is of type $A$.
If $p_1(t)=0$, then $p_3(t)=1$, and, eliminating $\alpha$, we get 
\[
(t \beta+t^k +\cdots) \beta=\gamma^2 
\]
which is a $D$-singularity.  

Case (C): $p_1(t)=1$. This is identical to the case $p_1(t)=1$ in Case (B) above.
%We eliminate $\gamma$ to get a surface in $k[[\alpha,\beta,t]]$ with tangent cone $\alpha\beta=0$, hence a Type $A$ singularity. % just as in the case (A). 
\end{proof}

\begin{lemma}\label{P:gorenstein-singularity-2}
Let $P = [0,0,1,1] \in X$ be a Gorenstein singularity. Then $P\in E$ is Du Val.  %% Is isolated even needed here?
Hence $P \in X$ is a terminal singularity. 
\end{lemma}
\begin{proof}
In the coordinates $\alpha:=x^3, \beta:=z^3-1, \gamma:=xz$ on the affine chart $w=1$, the equations of $E$ are
\begin{equation}\label{ge-gor}
\begin{aligned}
\gamma^3 &=\alpha(1+\beta) , \\
0&=(o(t)+q_1(t)) + p_1(t) \gamma+p_2(t) \alpha+ q_1(t)\beta+q_2(t)\gamma^2+q_3(t)\alpha\gamma+q_4(t)\alpha^2.
\end{aligned}
\end{equation}

For $P\in E$ to be singular, we must have $\val_t(p_1(t)), \val_t(q_1(t))\geq 1$ and $\val_t(o(t)+q_1(t))\geq 2$. However, 
if $\val_t(q_1(t)) \geq 2$ and $\val_t(o(t))\geq 2$, then $F$ is unstable for $\rho=(1,1,0,0)$. We conclude that $\val_t(q_1(t))=1$
and $o(t)=-q_1(t)$. Then $E$ is a hypersurface singularity in $\spec k[\beta,\gamma,t]$ with the tangent cone $t\beta=0$, hence of Type A.
%Semistability with respect to $\rho = (0,0,1,1)$ guarantees that either $val_t q_4(t) \lem 1$ or $p_2(t) = 1$.
%Thus $P \in E$ is isolated and hence of type $A$.

\end{proof}

This concludes the proof of Theorem \ref{mainthm2} in Case (1) of Lemma \ref{L:integral}.
%\begin{Prop} \label{Deg-1-general}
%Let $(F,H)$ be a T-semistable $(6,3)$-intersection and suppose $H_0$ defines $\MP(1,1,2,3)$.
%Let $X_R = \{F = H = 0 \} \subset \MP_R(1,1,2,3,3)$, then the tota
%\end{Prop}

%%%%%%%%%%%%%%%%%%%%%%%%%%%%%
%%%%%%%%%%%%%%%%%%%%%%%%%%%
\subsection{Central fiber in $\MP_k(1,2,9,9)$} \label{Step-3}
%%%%%%%%%%%%%%%%%%%%%%%%%%%
%%%%%%%%%%%%%%%%%%%%%%%%%%%%%

\begin{comment}
The aim of this subsection is to prove the following.

\begin{Prop}[This statement should be rewritten in view of all other changes] \label{Deg-1-special}
Suppose $F=H=0$ define a $(6,3)$-complete intersection $X$.
Suppose $(F,H)$ is a T-semistable %over $\MC[[t]]$  
and suppose $H_0 = 0$ is $\MP(1,2,9,9)$. Then $X_0$ is integral and the general elephant through every point of $X$ is Du Val.
\end{Prop}
\end{comment}

Suppose we are in Case (2) of Lemma \ref{L:integral}, so that $H_0=y^3-xz$. Then in some coordinates we have:
\begin{align*}
H&= y^3 - xz + t^n L(w,s),  \ n\geq 1, \ L(w,s)\in R[w,s]_3, \\
F&= q(w,s) + w b_3(x,y,z) + s a_3(x,y,z) + f_6(x,y,z) + t P(x,y,z,w,s),  \quad q(w,s)\neq 0.
\end{align*}
Here, and throughout the rest of this subsection, we use lowercase letters to denote elements of $k[x,y,z,w,s]$ and uppercase letters to denote elements of $R[x,y,z,w,s]$.

We first treat the easier case of $\ON{rk} q(w,s) = 2$ in the next lemma, and then consider the case $\ON{rk} q(w,s) = 1$ in \S\ref{S:rank-1}.
\begin{Lemma}\label{isolated-2}
If $\ON{rk} q(w,s) = 2$, then $X$ has isolated singularities, and for every singular point $P\in X$,  the general elephant through $P$ has Du Val Type A singularity at $P$. 
Thus $X$ is terminal.
\end{Lemma}
\begin{proof}
We can change coordinates so that
\begin{equation}\label{rk2-st}
\begin{aligned}
H&= y^3 - xz + t^n L(w, s) \\
F &= ws + f_6(x,y,z) + t P(x,y,z).
\end{aligned}
\end{equation}
%where $y^3-xz\nmid f_6$ by the irreducibility of $X_0$ established in Lemma \ref{L:integral}.
%However, if $f_6=0$, then $(F,H)$ is destabilized by both $\rho=(0,0,0,0,1)$
%and $\rho=(0,0,0,1,0)$. Hence we can assume $y^3-xz\nmid f_6$.
By Lemma  \ref{L:integral},
the central fiber $X_0$ is an irreducible degree $18$ hypersurface
$ws+f_6(\alpha^3,\alpha\beta,\beta^3)=0$ in $\PP_k(1_{\alpha},2_{\beta},9_w,9_s)$. Evidently, it 
has isolated singularities.  Hence $X$ has isolated singularities.
We proceed to show that the general elephant through a singularity $P\in X$ is Du Val.

Up to a change of coordinates, preserving the form \eqref{rk2-st}, 
we have $$P\in \{[1:0:0:0:0], [0:0:1:0:0], [0:0:1:0:1], [0:0:0:0:1]\}.$$ 
In each case we take the elephant $E$ given by $y=ctx$, or $y=cx$, with $c\in k$ general. The generic fiber of such $E$ is smooth by Lemma
\ref{elephant-K}
and has reduced central fiber $E_0$ because already the central fiber of the elephant $y=0$ is reduced.

We provide details only for two cases, with the rest being analogous (and easier than corresponding cases of $\ON{rk} q(w,s)=1$).
\begin{enumerate}
\item $P=[1:0:0:0:0]$. We take $E: y=ctx$. %and make a coordinate change $z=z+c^3t^3x^2$.  
Then in the affine chart $x=1$, we have $z=c^3t^3+t^n L(w, s)$ and 
$E$ has an isolated singularity \[
ws+f_6(1,ct, c^3t^3+t^n L(w, s))+tP(1,ct, c^3t^3+t^n L(w, s))=0,
\]
with a reduced quadric tangent cone $ws=0$, hence of Type A.

\begin{comment}
\item $P=[0:0:1:0:0]$. We take $E: y=cx$. % and make a coordinate change $z=z+c^3x^2$. 
In the affine chart $z=1$, and in local coordinates $\alpha=w^2$, $\beta=s^2$, $\gamma=ws$, %%% Note that xw=tAw, xs=tAs, x^2=t^2A^2
and $t$, we use $H$ to write $x^2=t^{2n} u(\alpha, \beta, \gamma)$, where $u(\alpha, \beta, \gamma)\in R[\alpha,\beta,\gamma]$. Then $E$ is isomorphic at $P$ to 
\begin{align*}
\alpha\beta=\gamma^2, \\
\gamma + t G(u(\alpha, \beta, \gamma), t)=0,
\end{align*}
which is again formally isomorphic to a type A singularity.
\end{comment}

\item $P=[0:0:0:0:1]$. We take $E: y=cx$. % and make a coordinate change $z=z+c^3x^2$.
 In the affine chart $s=1$, and in local coordinates $\alpha=x^3, \beta=z^3, \gamma=xz$, and $w, t$, the equations of $E$ are
\begin{align*}
\alpha\beta&=\gamma^3, \\
\gamma&=t^n L(w,1)+c^3 \alpha, \\
w&=G(\alpha,\beta,\gamma,t).
\end{align*}
This is again a Type A singularity.  
\end{enumerate}
\end{proof}

\begin{comment} %%% OLD ARGUMENT
Clearly, $X_0$ has $2$ singularities of type $\frac{1}{9}(1,2)$ at $\alpha=\beta=s=0$ and $\alpha=\beta=w=0$.
Also $X_0$ has an $A_n$-singularity along the line $w=s=0$ for every root of $f_{18}(\alpha,\beta)$ of multiplicity $n+1$.
The central fiber $X_0$ has no other singularities, hence $X_0$ has only quotient singularities.
Thus by \cite[Corollary~3.6]{KSB}, $X$ is terminal along $X_0$.
\end{proof}
\end{comment}

%%%%%%%%%%%%%%%%%%%%%%%%%%%%%
\subsubsection{The case $\ON{rk} q(w,s) = 1$}\label{S:rank-1}
%%%%%%%%%%%%%%%%%%%%%%%%%%%%%

\begin{lemma}\label{AB} 
Suppose $(F, H)$ is T-semistable with $H_0=xz-y^3$ and $F_0(0,0,0,w,s)$ a rank $1$ quadric. 
Then after a change of coordinates we have:
\begin{equation}\label{standard-rank-1}
\begin{aligned}
F &=w^2+s(ts+a_3(x,y,z))+f_6(x,y,z)+tB_6(x,y,z), \\
H&=y^3-xz+ t^n L(w,s), \ n\geq 1,
\end{aligned}
\end{equation}
where
\begin{enumerate}
\item
Either $a_3\notin (y^3-xz)$ or $f_6\notin (y^3-xz)$.
\item $ts+a_3(x,y,z)\notin (H)$.
\item Either $ts+a_3(x,y,z)\notin (x,y)^3 +(H)$, or $f_6$ has $z^3$ term.
\item Either $ts+a_3(x,y,z)\notin (y,z)^2 +(H)$, or $f_6\notin (y,z)^2$, or $B_6$ has $x^6$ term. 
\end{enumerate}
\end{lemma}
\begin{remark} If $(t^nL(w,s))\neq (ts)$, Conditions (2--4) are vacuous. In this case, the standard form \eqref{standard-rank-1} uniquely determines $F$. If $t^nL(w,s)= ts$, 
$F$ is determined only up to a multiple of $H=y^3-xz+ts$, which explains the phrasing of Conditions (2--4).
\end{remark}
\begin{proof}
Write the equations of $X$ in the form
\begin{align*}
&F = w^2 + s a_3(x,y,z) + f_6(x,y,z) + t B_6(x,y,z,s) = 0,\\
%WE CAN ELIMINATE w from B_6
&H = y^3 - xz + t A_3(x,y,z,w,s) = 0.
\end{align*}
Semistability with respect to $\rho=(1,1,1,1,0)$ %to the $\rho$-action $t\cdot (x,y,z,w,s)=(tx, ty, tz, tw, s)$  
implies that either $\mult_{\rho}(F)\leq 1$ or $\mult_{\rho}(H)\leq 1$.
The only possible terms of $\rho$-weight $1$ are $ts$ in $H$ or $ts^2$ in $F$. In either case, taking $F=F+s\lambda H$, for an appropriate $\lambda\in k$, 
ensures that $F$ has $ts^2$ term. We can then eliminate $s$ from $B_6$ to arrive at $B_6\in R[x,y,z]$.
By changing $x,y,z$, we can eliminate $x,y,z$ from $A_3$, so that $ t A_3(x,y,z,w,s)=t^n L(w,s)$ for $L(w,s)\in R[w,s]$.

Condition (1) follows from the integrality of $X_0$. Condition (2) follows from semistability with respect to $(1,1,2,3,2)$, (3) follows from semistability with respect to $\rho=(1,1,1,2,1)$,
and (4) 
follows from semistability with respect to $\rho=(0,1,1,1,0)$.
\end{proof}

\begin{Lemma}\label{isolated-3}
Suppose $(F,H)$ is T-semistable, then %$X_0$ and hence 
$X$ has only isolated singularities.
\end{Lemma}
\begin{proof}
Suppose $X$ is singular along a curve $C\subset X_0$.
%Then $X_0$ is singular along $C$.
By the Jacobian criterion, we have \[
C \subset (w = y^3-xz = a_3(x,y,z) = f_6(x,y,z)= t = 0),
\] a cone over a finite set of points in $\PP(1_x,1_y,2_z)$.
After a change of coordinates $x,y,z$ preserving $y^3-xz$, we can assume $C$ is either a line $t = w = z = y = 0$ or a line $t = w = x = y = 0$.  

Suppose $X$ is singular along $t = w = z = y = 0$. Then we can change $F$ by a multiple of $H$ to ensure $F\in (t,w,z,y)^2$. 
For $\rho = (0,1,1,1,0)$, we have $\mult_{\rho}(H)=1$ and $\mult_{\rho}(F)=2$, making $(F,H)$ $\rho$-unstable. 

Suppose $X$ is singular along $t = w = x = y = 0$. Then we can change $F$ by a multiple of $H$ to ensure $F\in (t,w,x,y)^2$. 
For $\rho = (1,1,0,1,0)$, we have $\mult_{\rho}(H)=1$ and $\mult_{\rho}(F)=2$, making $(F,H)$ $\rho$-unstable. \end{proof}

We proceed to prove that the isolated singularities of $X$ are terminal.  
%%We may assume that $a_3(x,y,z)$ does not have monomial $xz$ since we may add a multiple of $sH_3$ to eliminate it. -- NOT SURE WE CAN OR SHOULD DO THIS
Let $P \in X$ be a singular point. Up to a change of variables preserving the standard form \eqref{standard-rank-1}
\begin{enumerate}
\item $P=[1:0:0:0:0]$, a Gorenstein point.
\item $P=[0:0:1:0:0]$, a 2-Gorenstein point
\item $P=[0:0:1:0:1]$, a Gorenstein point.
\item $P=[0:0:0:0:1]$, a 3-Gorenstein point. 
\end{enumerate}

\begin{comment}
Suppose x or y are nonzero.  Then x is nonzero since y^3=xz.  If y=a is nonzero, then (y-ax)^3=x(z-3ay^2-3ya^2x-a^3x^3) and we take y'=y-ax, z=z-3ay^2-3ya^2x-a^3x^3.
This makes y=, and hence z=0. A singular point always has w=0.  Suppose s=a is nonzero. If H=y^3-xz+t^n w, then just take s=s-ax^3. If H=y^3-xz+t^n s, then s=s-ax^3 and z--> z-a t^n x^2
gives also s=z=0.

Suppose x=y=0. If z=1 and s=a nonzero, then rescaling s, we get $P=[0:0:1:0:1]$. This is Gorenstein point. 
If z=1 and s=0, we get $P=[0:0:1:0:0]$.

If z=0 as well, then we get $P=[0:0:0:0:1]$.  
\end{comment}

\begin{lemma}\label{isol-el-2} For every choice of coordinates preserving the standard form \eqref{standard-rank-1}, let $E$ be the elephant given by the equation $y=cx$ (resp., $y=ctx$), where $c\in k$ is general. 
Then $E$ has isolated singularities (resp., isolated singularities away from $x=y=t=0$).
\end{lemma}

\begin{proof} The generic fiber of $E$ is smooth by Lemma \ref{elephant-K}.
It remains to show that $E$ has isolated singularities along the central fiber (respectively, isolated singularities away from $x=y=t=0$). Suppose
not, and let $C$ be a one-dimensional component of the singular locus of $E$ in $t=0$.
By the Jacobian criterion, we have $C \subset (w = c^3x^3-xz = a_3(x,cx,z) = f_6(x,cx,z)= t = 0)$,
or $C \subset (w = c^3t^3x^3-xz = a_3(x,ctx,z) = f_6(x,ctx,z)= t = 0)$.

Since $y^3-xz=a_3(x,y,z)=f_6(x,y,z)=0$ is a finite set in $\PP(1_x,1_y,2_z)$, a general section $y=cx$ (resp., $y=ctx$) will avoid it unless this finite set contains 
a point with $y=0$. The possibilities are then $[x:y:z]=[0:0:1]$ if $y=cx$ (resp., $[x:y:z]=[0:0:1], [1:0:0]$ if $y=ctx$). 

Consider first the case of $E: y=cx$. Its singularities lie along the line $x=y=w=0$. If $f_6$ has $z^3$ term, then $w=x=y=0$ implies $z=0$, hence an isolated singularity.
If $f_6$ has no $z^3$ term, then $s(st+a_3(x,cx,z))+\lambda sH$ has either $st$ or $sxz$ term for any $\lambda \in k$ (Condition 3). In this case, $(\partial F/\partial x)(\partial H/\partial t)- (\partial F/\partial t) (\partial H/\partial x)$
is generically non-zero along $w=x=y=t=0$, and so the claim follows.

Consider now the case of $E: y=ctx$. We need to show that its singularities are isolated along the line $w=y=z=0$.  This follows from Condition 4.
\end{proof}
\begin{Lemma}
Let $P=[0:0:0:0:1] \in X$ be a $3$-Gorenstein singularity. Then the generic elephant $E$ through $P$ has Du Val singularities.
%Suppose $X$ is not terminal at $P$, then $(F,H)$ is unstable for $\rho=(1,1,1,2,2)$ or $\rho=(3,2,1,1,2)$. %$\rho=(1,1,1,1,0)$
\end{Lemma}
\begin{proof}
Take $E$ to be given by $y=cx$ where $c\in k$ is general, 
In the affine chart $s=1$, we work in the local coordinates $\alpha:=x^3, \beta:=x(z-c^3x^2), \gamma:=(z-c^3x^2)^3$, and $w, t$.
The equations of $E$ in $\spec k[\alpha,\beta,\gamma,w,t]$ are:
\begin{align*}
\alpha\gamma &=\beta^3, \\
\beta&= t^n L(w,1), \\
0&=w^2+(t+a(\alpha, \beta))+f(\alpha,\beta,\gamma)+t B(\alpha,\beta,\gamma). %%% a(\alpha,\beta) --linear in \alpha,\beta
\end{align*}
%Eliminating $\beta$ from the 3rd equation using the 2nd equation, we get
%\[
%w^2+(t+a(\alpha, t^nL(w,1)))+f(\alpha,t^n L(w,1),\gamma)+t B(\alpha,t^n L(w,1),\gamma).
%\]
If $t+a(\alpha, t^n L(w,1))$ has nonzero $t$ term, we proceed working formally locally and eliminate $\beta$ and $t$
to arrive at 
a hypersurface $$\alpha\gamma=(w^{2n}L(w,1) +\cdots )^3$$ in $k[[\alpha,\gamma,w]]$. Hence $E$ has Type A singularity at $P$.

Otherwise, we necessarily have $t^nL(w,1)=t$ (up to a nonzero scalar) by Condition 3, $t+a(\alpha, t^nL(w,1))$ is a non-zero multiple of $\alpha$, and $f$ must have a $\gamma$ term.  
Eliminating $t$ and $\alpha$, we arrive at
\[
(w^2+\gamma+(\text{higher degree terms})) \gamma=\beta^3,
\]
an $E_6$ singularity.

 \begin{comment}

In Case B, the equations of $E$ are:
\begin{align*}
\alpha\beta &=\gamma^3, \\
0&=\gamma-c^3 \alpha+t^n L(w,1), \\
0&=w^2+a(\alpha, \gamma)+f(\alpha,\beta,\gamma)+t (1+C(\alpha,\beta,\gamma,t)).
\end{align*}
Noting that $t^nL(w,1)\in (t,w)^2$, and working formally locally, we eliminate $\gamma$ and $t$ to arrive at 
a hypersurface $$\alpha\beta=(w^{2n}L(w,1) +\cdots )^3$$ in $k[[\alpha,\beta,w]]$. Hence $E$ has Type A singularity at $P$.

In Case A, the equations of $E$ are:
\begin{align*}
\alpha\beta &=\gamma^3, \\
0&=\gamma-c^3 \alpha+t, \\
0&=w^2+a(\alpha, \gamma)+f(\alpha,\beta,\gamma)+B(\alpha,\beta,\gamma,t).
\end{align*}
In Case A(1), $a=0$, and $B$ has $t$ term, so we argue as in Case B to arrive at an A singularity.

In Case A(2), $a\neq 0$ and $B$ has no $t$ term. Eliminating $t$, the third equation becomes 
\[
w^2+a(\alpha, \gamma)+(\text{higher degree terms})=0,
\]
where $a(\alpha,\gamma)$ is a nonzero linear form in $\alpha, \gamma$.  If the coefficient of $\gamma$ in $a(\alpha,\gamma)$ is not zero, we get an A singularity as in the previous cases.
If $a=\alpha$, then $f$ has $\beta$ term and the singularity at $P$ is 
\[
\beta (w^2+\beta+(\text{higher degree terms}))=\gamma^3,
\]
an $E_6$ singularity. 
\end{comment}
\end{proof}

\begin{Lemma}
Let $P=[0:0:1:0:0] \in X$ be a $2$-Gorenstein singularity. Then $P \in E$ is Du Val for a general elephant $E$.
%In particular $P \in X$ is a terminal singularity.
\end{Lemma}
\begin{proof}
We take $E$ to be given by $y=cx$, where $c\in k$ is general, and change coordinate $z=z+c^3x^2$.
Then the equation of $E$ are 
\begin{align*}
y &=c x, \\
xz &=t^n L(w,s), \\
0 & = w^2 + s(ts+a_3(x,y,z)) + f_6(x,y,z) + t B_6 (x,y,z,s).
\end{align*}
Since $f_6$ has no $z^3$ term, we must have $ts+a_3(x,y,z)\notin k[x,y] \pmod{xz-t^n L(w,s)}$. 

We work in the affine chart $z = 1$, where the local variables will be
$\alpha = w^2$, $\beta = s^2$, $\gamma = ws$, and $t$. 
Eliminating $x$ using $x=t^nL(w,s)$, we arrive at 
\begin{align*}
\gamma^2 &= \alpha \beta, \\
0 & = \alpha + (c' t \beta+c'' t \gamma+(h.o.t.)) + f(\alpha,\beta,\gamma) + t B(\alpha, \beta, \gamma),
\end{align*}
where $c'\neq 0$.

If $B_6$ has a $z^3$ term (resp., $B$ has a constant term), then we get $A_1$-singularity.
Otherwise, $f(\alpha,\beta,\gamma) + t B(\alpha, \beta, \gamma)$ is at least quadratic. 
Eliminating $\alpha$, we arrive at
\[
\gamma^2-(c't\beta^2-c''t\beta \gamma)+ (\text{higher order terms})=0
\]
in $\spec k[[\beta,\gamma,t]]$. Since $c'\neq 0$, this is a Type D singularity.
\end{proof}

\begin{Lemma} \label{Gor10000}
Suppose $P = [1,0,0,0,0]\in X$ is a singular point. A general elephant $E$ of the form $y=ctx$, $c\in k$, has Type A or D singularities at $P$.
%Then $X$ is cA or cD, hence terminal, at $P$. %or $X$ is unstable for $\rho = (1,1,2,3,2)$, or $\rho = (0,1,1,1,0)$. 
%$\rho = (0,1,2,2,1)$, , $\rho = (0,1,1,2,1)$, or $\rho = (0,2,1,3,1)$, those now appear in the next Lemma!
\end{Lemma}
\begin{proof} We change variables $z=z+c^3t^3x^2$, and work in the affine chart $x=1$, so that $z=t^nL(w,s)$. 
The equation of $E$ in $\spec k[[s,w,t]]$ then becomes
\begin{align*}
w^2 +s(ts+a_3(1,ct,t^nL(w,s))+f_6(1,ct,t^nL(w,s)) + tB(1,ct,t^nL(w,s))=0.
\end{align*}
where either $ts+a_3(1,ct,t^nL(w,s))$ has $ts$ term, or $t^nL(w,s)=ts$ and $f_6$ has $x^4z=ts$ term (Condition 4).
Hence $P$ is either Type D or Type A singularity.
\end{proof}

\begin{Lemma}
Suppose $P = [0,0,1,0,1]$ is a singular point of $X$. Then $P\in E$ is an $A_1$-singularity. 
\end{Lemma}
\begin{proof}
Let $q(t)$ be the coefficient of $F$ in Lemma \ref{AB}. Note that $\val_t(q(t))\geq 1$.
In the affine chart $s=1$, in local coordinates $\alpha:=x^3, \beta:=z^3-1,  \gamma:=xz$, and $w, t$,  the equations of $E$ are then
\begin{equation}
\begin{aligned}
\gamma^3&=\alpha(\beta+1), \\
\gamma&=c^3 \alpha+t^n L(w,1) , \\
0&=w^2 + (t+a(\alpha, \gamma))+ q(t)(\beta+1)+Q(\alpha,\gamma),   %a(\alpha,\gamma) is linear and f and B are quadratic in \alpha,\gamma, but linear in \beta.
%F &=w^2+s(ts+a_3(x,y,z))+f_6(x,y,z)+tB_6(x,y,z), \\
%H&=y^3-xz+ t^n L(w,s), \ n\geq 1,
\end{aligned}
\end{equation}
where $a(\alpha,\gamma)$ is linear and $Q(\alpha, \gamma)\in R[\alpha,\gamma]$ quadratic in $\alpha, \gamma$. If $\val(q(t))\geq 2$, then the fact that $P\in E$ is singular implies that $E$ is singular along
the whole line $\alpha=\gamma=w=t=0$, which is a contradiction. 
If $\val(q(t))=1$, then $E$ has equation 
\[
w^2+t\beta+\cdots=0,
\] which is $A_1$-singularity. (When $(t^nL(w,1))\neq (ts)$, one can see that $\val(q(t))=1$ directly from the condition that $P\in E$ is singular.) 
\end{proof}

This concludes the proof of Theorem \ref{mainthm2} in Case (2) of Lemma \ref{L:integral}.

%%%%%%%%%%%%%%%%%%%%%%%%%%%%%%%%%%%%%%%%
%%%%%%%%%%%%%%%%%%%%%%%%%%%%%%%%%%%%%%%%
\section{Questions and conjectures}\label{QandC}
%%%%%%%%%%%%%%%%%%%%%%%%%%%%%%%%%%%%%%%%
%%%%%%%%%%%%%%%%%%%%%%%%%%%%%%%%%%%%%%%%

In this section, we state some questions and conjectures arising from our definition of stability for degree $1$ and $2$ del Pezzo fibrations, and 
concerning the specific parameter spaces of such del Pezzos appearing in this article.

%%%%%%%%%%%%%%%%%%%%%%%%%%%%%%%%%%%%%%%%
\subsection{Optimality of the models}
%%%%%%%%%%%%%%%%%%%%%%%%%%%%%%%%%%%%%%%%

In defining stability for a given fibration, even for a fixed $G$ and $M$, there is a choice of a $G$-invariant divisor (or, equivalently, a $G$-linearized line bundle $\cL$
with an invariant section).
For a given del Pezzo over $K$, many choices of $\cL$ may result in stable  models over $R$ with a terminal total space and integral fibers. 
It is possible that not all of these models are isomorphic, in which case it is natural to ask ``What is the best choice?''
In particular, we can ask the following question.
\begin{Question}
Do the line bundle $\cL^{ter}:=\cO(\fD^{ter})$ defined in Proposition \ref{P:T-semistability} and the CM line bundle $\cL^{CM}$ give the same notion of stability?
\end{Question}

From the birational point of view, we can consider optimality in the following manner.

\begin{Def}
Let $\pi\colon X \to Z$ be a morphism from a terminal $\MQ$-factorial variety to a curve.
Suppose $-K_X$ is $\pi$-ample.
Let $X_0$ be the fiber over $0 \in Z$ and suppose that $X_0$ is irreducible.
We define \emph{the movable canonical threshold} of $X$ along $X_0$ to be
\[
\mct_X(X_0) = \ON{sup} \left\{ \lambda \mid \left(X,\frac{\lambda}{n}\CM\right) \text{~is~canonical~along~} X_0, \text{~where~} \CM \subset\big| -nK_X + lF \big|, n,l \in \mathbb Z\right\}.
\]
\end{Def}

\begin{Remark}
The movable canonical threshold helps in measuring the singularities of $X$ along $X_0$ and, indirectly, of $X_0$ itself.
Indeed, by inversion of adjunction we have $\mct_X(X_0) \gem \lct(X_0)$.
\end{Remark}

\begin{Th}[{\cite[Theorem~1.5]{Ch-Cubics}}]
Let $Z$ be a smooth curve. Suppose that there is a commutative diagram
\begin{displaymath}
\xymatrix
{ 
	X\ar@{-->}[r]^\rho \ar[d]^\pi & \ar[d]^{\pi_Y} Y \\
	Z\ar@{=}[r] & Z
}
\end{displaymath}
such that $\pi$ and $\pi_Y$ are flat morphisms, and $\rho$ is a birational map that induces an isomorphism
\begin{align*}
\rho\big\vert_{X\setminus X_0}\colon X\setminus X_0\to {Y}\setminus{Y_0}
\end{align*}
where $X_0$ and $Y_0$ are scheme fibers of $\pi$ and $\pi_Y$ over a point $0\in Z$, respectively. 
Suppose that the varieties $X$ and $Y$ have terminal and $\QQ$-factorial singularities, the divisors $-K_X$ and $-K_Y$ are $\pi$-ample and $\pi_Y$-ample respectively, the fibers $X_0$ and $Y_0$ are irreducible, and $\mct_X(X_0) + \mct_Y(Y_0) > 1$. 
Then $\rho$ is an isomorphism.
\end{Th}

This theorem suggests a relation between (semi)stability and movable canonical threshold:

\begin{Conj} For $1\leq d\leq 3$, there exists a notion of $\fD$-stability for degree $d$ del Pezzo fibrations such that
every semistable model $\pi\colon X \to Z$ satisfies $\mct_X(X_t) \geq \frac{1}{2}$ for all $t \in Z$.
\end{Conj}

%%%%%%%%%%%%%%%%%%%%%%%%%%%%%%%%%%%%%%%%
\subsection{Birational rigidity of del Pezzo fibrations}
%%%%%%%%%%%%%%%%%%%%%%%%%%%%%%%%%%%%%%%%

\begin{Def}[{\cite[Definition~1.2]{Corti-Rigidity}}]\label{RigDef}
\noindent A Mori fiber space $\pi\colon X\rightarrow S$ is said to be {\it birationally rigid} if the existence of a birational map $\chi \colon X\dashrightarrow Y$ to a Mori fiber space $\sigma\colon Y\to T$ implies that there exist a birational selfmap $\alpha\colon X\dashrightarrow X$ and a birational map $g\colon S\dashrightarrow T$ such that the following diagram commutes

\vspace{-0.5cm}
\begin{displaymath}
\xymatrix
{ 
	X\ar@{-->}[r]^{\chi\circ\alpha} \ar[d]_{\pi} & \ar[d]^{\sigma} Y \\
	S\ar@{-->}[r]^g & T
}
\end{displaymath}
and that the induced map on the generic fibers $X_\eta$ and $Y_\eta$ is an isomorphism.
\end{Def}

%\begin{Remark}
It is well known that a birational map between two Mori fiber spaces can be decomposed into so-called elementary Sarkisov links.
There are four types of elementary links numbered by I, II, III, IV, which are explicitly described in \cite{Corti-Sarkisov}. %and \cite{Hacon-McKernan-Sarkisov}. 
For a del Pezzo fibration, links of type III and IV could be initiated only if the $K$-condition is not satisfied, see \cite{Ahm-LMS} for an analysis in degree $2$.

%From the point of view of Sarkisov program we can state birational rigidity as follows.
%A Mori fiber space $\pi: X \to Z$ is birationally rigid if any birational map from $X$ to a Mori fiber space can be decomposed into type II Sarkisov links.
%\end{Remark}

\begin{Def}
We say that $\pi\colon X \to S$ satisfies the $K$-condition if $-K_X \not\in \overline{\ON{Mob} (X)}^\circ$, where $\overline{\ON{Mob} (X)}^\circ$ is the interior of the cone of mobile divisors.
\end{Def}

One can think of this as the condition which prevents the existence of elementary Sarkisov links of type III and IV. %Conjecturally, the $K$-condition is a necessary condition for birational rigidity.
Type I or II links are initiated by an extremal extraction (MMP blow up) of a curve or a point in $X$. The expectation is that Type I links do not exist on (semistable) del Pezzo fibrations in degree $1$, $2$ or $3$. Type II links however can exist, and these are typically birational maps where the central fiber is changed. We expect that 
for $1\leq d\leq 3$, there exists a notion of $\fD$-stability for degree $d$ del Pezzo fibrations such that the modified version of Grinenko's conjecture on birational rigidity of del Pezzo fibrations (see \cite{pukh123}, \cite{grinenko-1}, and \cite{grinenko-2}) holds:

\begin{Conj}
Let $1\leq d\leq 3$ and suppose $X \to \MP^1$ is a $\fD$-semistable del Pezzo fibration of degree $d$ satisfying the $K$-condition.
Then $X$ is birationally rigid.
%not square birational to a fibration of special type. 
\end{Conj}

\bibliographystyle{amsplain}%{alpha}
\bibliography{AFK-bib-2019}{}

\def\cprime{$'$} \def\cprime{$'$}
\providecommand{\bysame}{\leavevmode\hbox to3em{\hrulefill}\thinspace}
\providecommand{\MR}{\relax\ifhmode\unskip\space\fi MR }
% \MRhref is called by the amsart/book/proc definition of \MR.
\providecommand{\MRhref}[2]{%
  \href{http://www.ams.org/mathscinet-getitem?mr=#1}{#2}
}
\providecommand{\href}[2]{#2}
\begin{thebibliography}{10}

\bibitem{Ahm-LMS}
H.\ Ahmadinezhad, \emph{On del {P}ezzo fibrations that are not birationally
  rigid}, J. Lond. Math. Soc. (2) \textbf{86} (2012), no.~1, 36--62.
  \MR{2959294}

\bibitem{cartaniwahorimatsumoto}
Jarod Alper, Daniel Halpern-Leistner, and Jochen Heinloth,
  \emph{Cartan-iwahori-matsumoto decompositions for reductive groups}, 2019.

\bibitem{AGZV}
V.~I. Arnol\cprime~d, S.~M. Guse\u{\i}n-Zade, and A.~N. Varchenko,
  \emph{Singularities of differentiable maps. {V}ol. {I}}, Monographs in
  Mathematics, vol.~82, Birkh\"{a}user Boston, Inc., Boston, MA, 1985, The
  classification of critical points, caustics and wave fronts, Translated from
  the Russian by Ian Porteous and Mark Reynolds. \MR{777682}

\bibitem{artin-RDP-77}
M.~Artin, \emph{Coverings of the rational double points in characteristic
  {$p$}}, Complex analysis and algebraic geometry, 1977, pp.~11--22.
  \MR{0450263}

\bibitem{artin-winters}
M.~Artin and G.~Winters, \emph{Degenerate fibres and stable reduction of
  curves}, Topology \textbf{10} (1971), 373--383. \MR{476756}

\bibitem{neron-book}
Siegfried Bosch, Werner L\"{u}tkebohmert, and Michel Raynaud, \emph{N\'{e}ron
  models}, Ergebnisse der Mathematik und ihrer Grenzgebiete (3) [Results in
  Mathematics and Related Areas (3)], vol.~21, Springer-Verlag, Berlin, 1990.
  \MR{1045822}

\bibitem{Ch-Cubics}
I.\ Cheltsov, \emph{On singular cubic surfaces}, Asian J. Math. \textbf{13}
  (2009), no.~2, 191--214. \MR{2559108}

\bibitem{Corti-Sarkisov}
A.\ Corti, \emph{Factoring birational maps of threefolds after {S}arkisov}, J.
  Algebraic Geom. \textbf{4} (1995), no.~2, 223--254. \MR{1311348}

\bibitem{corti-annals}
\bysame, \emph{Del {P}ezzo surfaces over {D}edekind schemes}, Ann. of Math. (2)
  \textbf{144} (1996), no.~3, 641--683. \MR{1426888}

\bibitem{Corti-Rigidity}
\bysame, \emph{Singularities of linear systems and {$3$}-fold birational
  geometry}, Explicit birational geometry of 3-folds, London Math. Soc. Lecture
  Note Ser., vol. 281, Cambridge Univ. Press, Cambridge, 2000, pp.~259--312.

\bibitem{GKZ}
I.~M. Gel\cprime~fand, M.~M. Kapranov, and A.~V. Zelevinsky,
  \emph{Discriminants, resultants, and multidimensional determinants},
  Mathematics: Theory \& Applications, Birkh\"{a}user Boston, Inc., Boston, MA,
  1994. \MR{1264417}

\bibitem{grinenko-1}
M.\ Grinenko, \emph{Birational properties of pencils of del {P}ezzo surfaces of
  degrees 1 and 2}, Mat. Sb. \textbf{191} (2000), no.~5, 17--38. \MR{1773767}

\bibitem{grinenko-2}
\bysame, \emph{Fibrations into del {P}ezzo surfaces}, Uspekhi Mat. Nauk
  \textbf{61} (2006), no.~2(368), 67--112. \MR{2261543}

\bibitem{kollar-polynomials}
J.\ Koll\'ar, \emph{Polynomials with integral coefficients, equivalent to a
  given polynomial}, Electron. Res. Announc. Amer. Math. Soc. \textbf{3}
  (1997), 17--27. \MR{1445631}

\bibitem{KSB}
J.\ Koll{\'a}r and N.\ Shepherd-Barron, \emph{Threefolds and deformations of
  surface singularities}, Invent. Math. \textbf{91} (1988), no.~2, 299--338.
  \MR{922803 (88m:14022)}

\bibitem{kollar-MMP-recent}
János Kollár, \emph{Relative {MMP} without $ \mathbb{Q} $-factoriality},
  Electronic Research Archive \textbf{0} (2021), --.

\bibitem{lipman}
Joseph Lipman, \emph{Rational singularities, with applications to algebraic
  surfaces and unique factorization}, Inst. Hautes \'{E}tudes Sci. Publ. Math.
  (1969), no.~36, 195--279. \MR{276239}

\bibitem{Loginov}
K.\ Loginov, \emph{Standard models of degree 1 del {P}ezzo fibrations}, Mosc.
  Math. J. \textbf{18} (2018), no.~4, 721--737. \MR{3914112}

\bibitem{Mori}
S.\ Mori, \emph{On {$3$}-dimensional terminal singularities}, Nagoya Math. J.
  \textbf{98} (1985), 43--66. \MR{792770}

\bibitem{GIT}
David Mumford, John Fogarty, and Frances Kirwan, \emph{Geometric invariant
  theory}, third ed., Ergebnisse der Mathematik und ihrer Grenzgebiete (2)
  [Results in Mathematics and Related Areas (2)], vol.~34, Springer-Verlag,
  Berlin, 1994. \MR{MR1304906 (95m:14012)}

\bibitem{neron}
Andr{\'e} N{\'e}ron, \emph{Probl\`emes arithm\'etiques et g\'eom\'etriques
  rattach\'es \`a la notion de rang d'une courbe alg\'ebrique dans un corps},
  Bull. Soc. Math. France \textbf{80} (1952), 101--166. \MR{MR0056951
  (15,151a)}

\bibitem{pedersen-sturmfels}
Paul Pedersen and Bernd Sturmfels, \emph{Product formulas for resultants and
  {C}how forms}, Math. Z. \textbf{214} (1993), no.~3, 377--396. \MR{1245200}

\bibitem{pukh123}
A.\ Pukhlikov, \emph{Birational automorphisms of three-dimensional algebraic
  varieties with a pencil of del {P}ezzo surfaces}, Izv. Ross. Akad. Nauk Ser.
  Mat. \textbf{62} (1998), no.~1, 123--164. \MR{1622258}

\bibitem{special-groups}
Zinovy Reichstein and Dajano Tossici, \emph{Special groups, versality and the
  {G}rothendieck-{S}erre conjecture}, Doc. Math. \textbf{25} (2020), 171--188.
  \MR{4106889}

\bibitem{reid-YPG}
M.\ Reid, \emph{Young person's guide to canonical singularities}, Algebraic
  geometry, {B}owdoin, 1985 ({B}runswick, {M}aine, 1985), Proc. Sympos. Pure
  Math., vol.~46, Amer. Math. Soc., Providence, RI, 1987, pp.~345--414.
  \MR{MR927963 (89b:14016)}

\bibitem{Sarkisov}
V.\ Sarkisov, \emph{On conic bundle structures}, Izv. Akad. Nauk SSSR Ser. Mat.
  \textbf{46} (1982), no.~2, 371--408, 432. \MR{651652}

\bibitem{serre-special}
Jean-Pierre Serre, \emph{Espaces fibr\'es alg\'ebriques, expos\'e 1},
  S\'{e}minaire {C}. {C}hevalley; 2e ann\'{e}e: 1958. {A}nneaux de {C}how et
  applications, Secr\'{e}tariat math\'{e}matique, 11 rue Pierre Curie, Paris,
  1958. \MR{0110704}

\bibitem{tate-algorithm}
J.~Tate, \emph{Algorithm for determining the type of a singular fiber in an
  elliptic pencil}, Modular functions of one variable, {IV} ({P}roc.
  {I}nternat. {S}ummer {S}chool, {U}niv. {A}ntwerp, {A}ntwerp, 1972), 1975,
  pp.~33--52. Lecture Notes in Math., Vol. 476. \MR{0393039}

\bibitem{vistoli}
Angelo {Vistoli}, \emph{{Notes on Grothendieck topologies, fibered categories
  and descent theory}}, arXiv Mathematics e-prints (2004), math/0412512.

\end{thebibliography}

\end{document}